\documentclass[letterpaper]{amsart}

\usepackage[letterpaper]{geometry}
\usepackage{amsmath, amsthm, amsfonts, amsbsy, thmtools, amssymb}

\usepackage{thm-restate}
\usepackage{hyperref} 
\usepackage[]{ytableau}
\usepackage[mathscr]{euscript}
\usepackage{stmaryrd}
\usepackage[all, arc, curve, frame]{xy} 
\usepackage[T1]{fontenc}
\usepackage{tikz}

\usepackage{hyperref}
\usepackage{cleveref}

\usetikzlibrary{arrows}

\makeatletter
\def\namedlabel#1#2{\begingroup
	#2%
	\def\@currentlabel{#2}%
	\phantomsection\label{#1}\endgroup
}
\makeatother

\numberwithin{equation}{section}

\SetSymbolFont{stmry}{bold}{U}{stmry}{m}{n}

\DeclareMathOperator{\NW}{NW}
\DeclareMathOperator{\SW}{SW}
\DeclareMathOperator{\NE}{NE}
\DeclareMathOperator{\SE}{SE}
\DeclareMathOperator{\aug}{aug}

\newtheorem{thm}{Theorem}[section]
\newtheorem{prop}[thm]{Proposition}
\newtheorem{lem}[thm]{Lemma}
\newtheorem{cor}[thm]{Corollary}

\theoremstyle{remark}
\newtheorem{rem}[thm]{Remark}
\theoremstyle{definition}
\newtheorem{definition}[thm]{Definition}
\newtheorem{example}[thm]{Example}

\title{Arctic Boundaries of the Ice Model on Three-Bundle Domains}

\author{Amol Aggarwal}

\begin{document}

\begin{abstract} 
	
	In this paper we consider the six-vertex model at ice point on an arbitrary three-bundle domain, which is a generalization of the domain-wall ice model on the square (or, equivalently, of a uniformly random alternating sign matrix). We show that this model exhibits the arctic boundary phenomenon, whose boundary is given by a union of explicit algebraic curves. This was originally predicted by Colomo-Sportiello in 2016 as one of the initial applications of a general heuristic that they introduced for locating arctic boundaries, called the (geometric) tangent method. Our proof uses a probabilistic analysis of non-crossing directed path ensembles to provide a mathematical justification of their tangent method heuristic in this case, which might be of independent interest. 
	
\end{abstract}

\maketitle

\tableofcontents

\section{Introduction} 

\label{Model}

\subsection{Preface}

Although the six-vertex model has long been cited as an archetypal example of an exactly solvable model in two-dimensional statistical mechanics \cite{ESMSM}, little has been mathematically established about its geometry. This is in deep contrast with random tiling models, whose geometric understanding has seen considerable advances over the last two decades. 

The earliest result in this direction was due to Jockusch-Propp-Shor \cite{RTAC}, who proved that a uniformly random domino tiling of an Aztec diamond exhibits a phase transition across the inscribed circle of the diamond. More specifically, they showed that with high probability the tiling is deterministic, or \emph{frozen}, outside of this circle due to the influence of the boundary, but that it is random inside of it. They referred to this circle as an \emph{arctic boundary} separating a frozen region from a \emph{liquid} one. 

This phenomenon was soon observed to be ubiquitous within the context of highly correlated statistical mechanical systems; see, for instance, \cite{IFTIAC,NSVMDWBC,DPP,DRYG,ARTR,LTECM,VPT,TSPP,TACDWSVM,AC,TLSLASM,TSVMD,ACFFSVMD,ACSVMGD,SVMDWBCA,TMACFB,TMACWP,ACPASP,ACTM,ACOE,DI,GACPCMT,RT,NSMDWB,LSCE,DA,PSSVMBC,ACTG,ACASM,DSVM}. In particular, Cohn-Kenyon-Propp developed a variational principle \cite{VPT} that prescribes a law of large numbers for random domino tilings on almost arbitrary domains, which was used effectively by Kenyon-Okounkov \cite{LSCE} to explicitly determine the arctic boundaries of uniformly random lozenge tilings on polygonal domains. The proof of this variational principle was based on the free-fermionic (determinantal) structure underlying these tiling models \cite{SDL}, so it was also applicable to more general dimer models \cite{DA} where such structure persists. 

However, these methods do not seem to apply to the six-vertex model, whose solvability is of a substantially different nature and can be attributed to a one-parameter family of mutually commuting transfer operators \cite{ESMSM}, which can be diagonalized through the quantum inverse scattering method (algebraic Bethe ansatz) \cite{QISMCF}. It was observed by Korepin \cite{NWF} that the quantum inverse scattering method is particularly applicable under a certain class of boundary data, later called \emph{domain-wall boundary data}, which also happens to be quite appealing from a physical perspective. Using the quantum inverse scattering method, Izergin \cite{PFSVMFV} and Izergin-Coker-Korepin \cite{DSV} showed that the partition function (and also some correlation functions \cite{QISMCF}) of the domain-wall six-vertex model can be expressed as a determinant. Yet, in spite of these striking algebraic and analytic advances, these results implied little about the geometry of the six-vertex model. 

Still, extensive simulations \cite{DI,DSVM,NSVMDWBC,SVMDWBCA,NSMDWB,PSSVMBC,RT} over the past two decades have provided strong numerical evidence indicating the existence of an arctic boundary for the domain-wall six-vertex model. Based on earlier free energy results due to Lieb \cite{RESI} and Sutherland-Yang-Yang \cite{ESMTFAEEF}, variational principles have also been proposed for the six-vertex model with general boundary conditions \cite{BCSVM,SVMFBC,ILSSVM}. However, these variational principles are intricate, and it remains unknown whether they can be used to derive explicit predictions for arctic boundaries (even in the case of domain-wall boundary data). 

Through a series of works \cite{AC,EFPDWSVM,TLSLASM,TACDWSVM} starting around 2008, Colomo-Pronko provided such a prediction. In particular, in \cite{AC}, they introduced a nonlocal correlation function called the \emph{emptiness formation probability} and explained how one can derive arctic boundaries from its asymptotic behavior. Using the quantum inverse scattering method, they provided an exact identity for this probability in \cite{EFPDWSVM}, which they then formally asymptotically analyzed in \cite{TACDWSVM,TLSLASM} to provide explicit predictions for the arctic curve of the domain-wall six-vertex model. Unfortunately, the identities obtained in \cite{EFPDWSVM} are intricate, and it remains unknown how to mathematically justify the formal analysis used in \cite{TACDWSVM,TLSLASM} to study them. 

Recently, Colomo-Sportiello \cite{ACSVMGD} introduced the \emph{(geometric) tangent method} to provide an alternative, still heuristic, derivation of these arctic boundaries. Stated very briefly (see \Cref{OutlineTangent} below for a more detailed explanation), this proceeds by first introducing a new, ``augmented'' vertex model $\mathcal{P}^{\aug}$ by adding an additional path to the original model $\mathcal{P}$. If one understands the asymptotics of a certain, often accessible, quantity called the \emph{(singly) refined correlation function}, one can determine the initial part of the limiting trajectory of the new (bottommost) path in $\mathcal{P}^{\aug}$. Now, the belief is that, in the continuum limit, this part of the trajectory will be a line segment tangent to the arctic boundary of $\mathcal{P}$. Assuming this to be true, one can then use one's knowledge of the new path's asymptotic trajectory to characterize the arctic boundary of the vertex model. 

The tangent method was implemented in \cite{ACSVMGD} to predict the arctic boundary of the domain-wall six-vertex model; the result matched earlier predictions from \cite{TACDWSVM}. It was also used later to heuristically derive the arctic boundaries for the six-vertex model on other domains \cite{ACSVMGD,ACFFSVMD} by Colomo-Sportiello and Colomo-Pronko-Sportiello; for vertically symmetric alternating sign matrices \cite{ACTM} by Di Francesco-Lapa; various classes of non-intersecting path models \cite{CAM,TMACFB,TMACWP,ACPASP,TACRM,ACTM} by Debin-Granet-Ruelle, Debin-Ruelle, Di Francesco-Guitter, and Di Francisco-Lapa; for twenty-vertex models by Debin-Di Francesco-Guitter \cite{ACVMDWB}; and for random lecture hall tableaux by Corteel-Keating-Nicoletti \cite{ACBHT}. 

In this paper we use a probabilistic analysis of non-crossing directed path ensembles to provide the first mathematical justification of the geometric tangent method. Instead of attempting to implement this in the fullest possible generality, for the sake of specificity we only focus on a particular example. Perhaps the simplest (non free-fermionic) case would be the domain-wall six-vertex model at ice point, where the weights of all six vertex types are equal; this model received considerable interest over the past three decades due to its relationship with alternating sign matrices \cite{TLSLASM,ASME,ASPASMRT,PPASM,ASMDDP,PASM,PASMR}. In order to demonstrate the versatility of this framework, and since it will change little in the proof, we will in fact analyze a more general situation, given by the ice model on the \emph{three-bundle domain}. 

Introduced by Cantini-Sportiello in Section 4.2 of \cite{PRC}, this is a domain $\mathcal{T} = \mathcal{T}_{A, B, C}$ (dependent on three integers $A, B, C \ge 0$) formed by intersecting three families of $A + B$, $A + C$, and $B + C$ parallel curves. In particular, by setting $A = 0 = B$, it becomes a $C \times C$ square, and the associated six-vertex model degenerates to the domain-wall ice model; thus, the model we consider comprises a two parameter deformation of the domain-wall ice model. This three-bundle domain is also of further combinatorial interest since it is an example of a dihedral domain on which the refined Razumov-Stroganov correspondence can be established \cite{PRC,P,CGVLM}; indeed, this point served as the original stimulus in \cite{PRC} to define the ice model on this domain. As one of the initial applications of the tangent method, Colomo-Sportiello predicted the arctic boundary of the ice model on the three-bundle domain in Section 5.4 of \cite{ACSVMGD}. 

Our result, given by \Cref{boundaryb} below, confirms this prediction. We will also explain in \Cref{AB0Domain} below how this theorem can be degenerated in the $A = 0 = B$ case to yield the arctic boundary of the domain-wall ice model, thereby confirming an earlier prediction of Colomo-Pronko \cite{TLSLASM}. An alternative route towards a mathematical proof of the latter result (that is, the arctic boundary prediction in the alternating sign matrix case) is currently work in progress by Colomo-Sportiello \cite{D}. We will elaborate on this point further below, directly before \Cref{ModelDomain}, and explain how their results and methods compare to ours. 

Now let us take a moment to very briefly describe some of the probabilistic considerations we must account for in the proof of \Cref{boundaryb}; we will provide a more detailed outline in \Cref{Outline} below. To that end, recall that the assumption underlying the tangent method was that the bottommost path of the augmented model $\mathcal{P}^{\aug}$ is, in the continuum limit, tangent to the arctic boundary of $\mathcal{P}$. However, there are several issues with establishing this statement. 

The first is that it is not transparent to us that the this arctic boundary has a deterministic continuum limit. Indeed, although concentration estimates for the height function of vertex models have been proven to hold in considerable generality \cite{LSRT,VPT}, they do not appear to immediately imply that the bottommost path in a six-vertex model concentrates (and this is what one would require in order to implement the tangent method). So, one must instead show that the first (bottommost) path $\textbf{p}_1^{\aug}$ of $\mathcal{P}^{\aug}$ is approximately tangent to the second path $\textbf{p}_2^{\aug}$ of $\mathcal{P}^{\aug}$. However, this then gives rise to two issues. The first is to establish this claim, and the second is to show that the arctic boundary of $\mathcal{P}$ can be approximated by $\textbf{p}_2^{\aug}$ in some suitable sense.  

To resolve these issues, we make use of two properties satisfied by the model. The first is its Gibbs property, and the second is that its Glauber dynamics are monotone (order-preserving). The latter in particular implies the existence of monotone couplings between two ice models with different boundary data, which are both used to limit the possible long-range correlations between distant paths of six-vertex ensembles and also to compare $\textbf{p}_2^{\aug}$ to the arctic boundary of $\mathcal{P}$. The Gibbs property is used to show that distant paths of the six-vertex model are nearly linear, which is useful both for establishing the approximate tangency between $\textbf{p}_1^{\aug}$ and $\textbf{p}_2^{\aug}$ and also again for comparing $\textbf{p}_2^{\aug}$ to the arctic boundary of $\mathcal{P}$. 
 
Let us mention that earlier works \cite{LE,ELE} of Corwin-Hammond also used and developed some of the above probabilistic ideas (namely, the Gibbs property and monotonicity) to analyze different classes of line ensembles. However, their settings in those papers are different from ours. Their goal was to gain a refined qualitative understanding about fluctuations in a certain line ensemble, given some initial quantitative behavior about the geometry of the ensemble's extremal curve (deduced from the integrability of their model). Ours is to gain a quantitative understanding about the limit shape of the extremal path in our ensemble, given some initial enumerative information about the number of such path ensembles.

Before continuing, we state three additional points. The first concerns possible extensions of our methods to other systems. Although in this paper we adhere to the specific example of the ice model on the three-bundle domain, the main model-specific probabilistic feature in our proof is the monotonicity (described above) of the ice model on $\mathcal{T}$. This property is in fact quite common in the context of vertex models, so it is plausible that our methods can applied to significantly wider classes of exactly solvable systems in statistical mechanics. 

In particular, one might ask whether the methods used this paper can also be applied to the domain-wall six-vertex model with more general weights. As explained earlier, our proofs are based on a justification of the tangent method, which requires access to asymptotics for the singly refined correlation function of the vertex model of interest. At ice point, these quantities can be expressed exactly due to the works \cite{PRC} of Cantini-Sportiello on the three-bundle domain and \cite{PASMR} of Zeilberger in the domain-wall case. 

For the six-vertex model with general weights on arbitrary three-bundle domains, such a quantity is not yet available. However, for the domain-wall six-vertex model, it was explained in Appendix B of \cite{TACDWSVM} how to express and asymptotically analyze such a quantity in terms of an Izergin-Korepin determinant. Although the asymptotic analysis in that work is formal, precise mathematical asymptotic analyses on similar Izergin-Korepin type determinants were implemented by Bleher-Fokin in \cite{ESSVMDWBC} through a matrix model interpretation due to Zinn-Justin \cite{SVMDWBCMM}. If this framework can also be applied to understand asymptotics of the refined correlation function, then the probabilistic analysis in this paper would allow one to confirm the predictions from \cite{TACDWSVM,ACSVMGD} for the arctic curves of more general domain-wall six-vertex models, at least for some ranges of the weights that ensure monotonicity of the model. 

The second point concerns other variants of the tangent method. In this paper we proceed through the geometric tangent method, which is the one that was described in \cite{ACSVMGD}. However, Sportiello-Colomo have also produced a \emph{doubly refined tangent method}, which appears to be simpler to mathematically justify but requires (possibly more precise) asymptotic access to a quantity called the \emph{doubly refined correlation function}. This quantity generalizes the singly refined correlation function required for the geometric tangent method, in that asymptotics for the former imply asymptotics for the latter, but the reverse does not hold. 

In fact, asymptotics for the doubly refined correlation function do not appear to be known in many cases of interest. In the special case of the domain-wall ice model, they can be derived \cite{ACASM} through a recursion of Stroganov \cite{DTR}. Using this fact, Sportiello outlined a potential route to establish the arctic boundary for the domain-wall ice model in his slides \cite{ACASM} from 2015; this will appear in forthcoming work of Colomo-Sportiello \cite{D}. However, this doubly refined tangent method does not seem to apply to the ice model on the three-bundle domain, since the doubly refined correlation function is not known in that case. 

The third point concerns our definition of the arctic boundary, which will be described in more detail in \Cref{Boundary} below. Much of our terminology is modeled off of that from the original work of Jockusch-Propp-Shor \cite{RTAC}. For us (and for them), the ``frozen region'' of a tiling consists of all vertices $v$ in the domain such that all vertices either to the northeast, northwest, southeast, or southwest of $v$ are fixed in the tiling and determined by the boundary conditions. One might view this as an ``external frozen region,'' and our results characterize it completely for the ice model on the three-bundle domain $\mathcal{T}$.  

However, it does not (nor does any variant of the tangent method) show that there are no internal regions of $\mathcal{T}$ that are frozen with high probability. We might mention that the original work \cite{RTAC} also did not explicitly forbid such a possibility in their situation; this was later addressed in the paper of Cohn-Elkies-Propp \cite{LSRT}. In fact, it is not apparent to us that such internal frozen regions cannot exist for the model we consider. To the contrary, it was claimed in Section 5.4 of \cite{ACSVMGD} (see in particular Figure 8 there) that such regions might exist for certain ranges (or possibly a limit degeneration) of the parameters. To the best of our knowledge, there is neither a proof that such an interior arctic boundary exists, nor a prediction that explicitly parameterizes it. 

Now let us proceed to describe our results in more detail; the remainder of this section is organized as follows.  We will provide a definition of the six-vertex model on the three-bundle domain in \Cref{ModelDomain}. We will then explain our results on its arctic boundary in \Cref{Boundary} and detail the degeneration to the domain-wall ice model in \Cref{AB0Domain}.

\subsection{Six-Vertex Ensembles on the Three-Bundle Domain}

\label{ModelDomain}

In this section we define the model of interest to us, which is the six-vertex model (at ice point) on the three-bundle domain. Throughout this section, $A, B \ge 0$ and $C \ge 1$ are integers. 

The \emph{three-bundle domain} $\mathcal{T} = \mathcal{T}_{A, B, C}$ (with parameters $A$, $B$, and $C$) is defined to be the following directed planar graph. The vertex set of $\mathcal{T}$ consists of all lattice points $(x, y) \in \mathbb{Z}_{> 0}^2$ such that either $(x, y) \in [1, A + 2 B + C] \times [1, A + C]$ or $(x, y) \in [A + B + 1, A + 2B + C] \times [1, 2A + B + C]$. The edge set of $\mathcal{T}$ consists of all edges connecting $v_1 = (x_1, y_1) \in \mathcal{T}$ to $v_2 = (x_2, y_2) \in \mathcal{T}$ satisfying either one of the following two conditions. The first is that $(x_2, y_2) \in \big\{ (x_1, y_1 + 1), (x_1 + 1, y_1) \big\}$. The second is that there exists some $m \in [1, A + B]$ such that $v_1 = (A + B - m + 1, A + C)$ and $v_2 = (A + B + 1, A + C + m)$; in this case, the edge connecting $v_1$ to $v_2$ is called \emph{diagonal}. 

We refer to the right side of \Cref{vertexdomain} for a depiction of the three-bundle domain $\mathcal{T}_{2, 3, 4}$. Observe in particular that all faces of $\mathcal{T}_{A, B, C}$ are quadrilaterals, except for one that is a triangle. One might view the graph $\mathcal{T}_{A, B, C}$ as coming from the intersections between three sets consisting of $A + B$, $A + C$, and $B + C$ parallel curves.

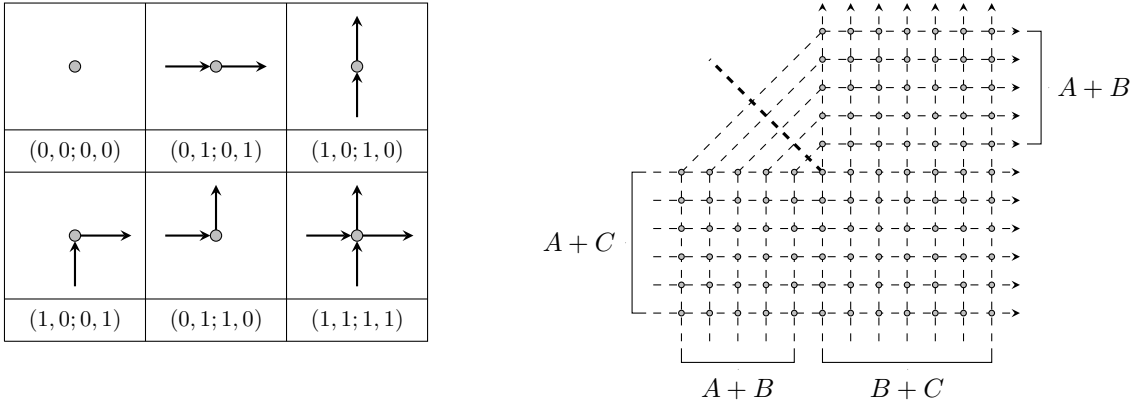
\begin{figure}[t]
	
	\begin{center}
		
		\begin{tikzpicture}[
		>=stealth,
		scale = .75 	
		]
		
		\draw[] (-11.5, 0) -- (-4, 0) -- (-4, 6) -- (-11.5, 6) -- (-11.5, 0	);
		
		\draw[] (-11.5, 3) -- (-4, 3);
		\draw[] (-9, 0	) -- (-9, 6);
		\draw[] (-6.5, 0) -- (-6.5, 6);
		\draw[] (-11.5, .75) -- (-4, .75);
		\draw[] (-11.5, 3.75) -- (-4, 3.75);
		
		\draw[] (-10.25, 3.375) circle [radius = 0] node[scale = .85]{$(0, 0; 0, 0)$};
		\draw[] (-7.75, 3.375) circle [radius = 0] node[scale = .85]{$(0, 1; 0, 1)$};
		\draw[] (-5.25, 3.375) circle [radius = 0] node[scale = .85]{$(1, 0; 1, 0)$};
		
		\draw[] (-10.25, .375) circle [radius = 0] node[scale = .85]{$(1, 0; 0, 1)$};
		\draw[] (-7.75, .375) circle [radius = 0] node[scale = .85]{$(0, 1; 1, 0)$};
		\draw[] (-5.25, .375) circle [radius = 0] node[scale = .85]{$(1, 1; 1, 1)$};

		\draw[->, black,  thick] (-8.65, 4.875) -- (-7.85, 4.875);
		\draw[->, black,  thick] (-7.65, 4.875) -- (-6.85, 4.875);
		
		\draw[->, black,  thick] (-5.25, 3.975) -- (-5.25, 4.775);
		\draw[->, black,  thick] (-5.25, 4.975) -- (-5.25, 5.775);

		\draw[->, black,  thick] (-10.15, 1.875) -- (-9.25, 1.875);
		\draw[->, black,  thick] (-10.25, .975) -- (-10.25, 1.775);

		\draw[->, black,  thick] (-8.65, 1.875) -- (-7.85, 1.875);
		\draw[->, black,  thick] (-7.75, 1.975) -- (-7.75, 2.775);
		
		\draw[->, black,  thick] (-6.15, 1.875) -- (-5.35, 1.875);
		\draw[->, black,  thick] (-5.15, 1.875) -- (-4.25, 1.875);
		\draw[->, black,  thick] (-5.25, .975) -- (-5.25, 1.775);
		\draw[->, black,  thick] (-5.25, 1.975) -- (-5.25, 2.775);
		
		\filldraw[fill=gray!50!white, draw=black] (-10.25, 1.875) circle [radius=.1];
		\filldraw[fill=gray!50!white, draw=black] (-10.25, 4.875) circle [radius=.1];
		\filldraw[fill=gray!50!white, draw=black] (-7.75, 1.875) circle [radius=.1];
		\filldraw[fill=gray!50!white, draw=black] (-7.75, 4.875) circle [radius=.1];
		\filldraw[fill=gray!50!white, draw=black] (-5.25, 1.875) circle [radius=.1];
		\filldraw[fill=gray!50!white, draw=black] (-5.25, 4.875) circle [radius=.1];

		\draw[->, black, dashed] (0, .5) -- (6.5, .5);	
		\draw[->, black, dashed] (0, 1) -- (6.5, 1);	
		\draw[->, black, dashed] (0, 1.5) -- (6.5, 1.5);	
		\draw[->, black, dashed] (0, 2) -- (6.5, 2);
		\draw[->, black, dashed] (0, 2.5) -- (6.5, 2.5);
		\draw[->, black, dashed] (0, 3) -- (6.5, 3);

		\draw[->, black, dashed] (3, 3.5) -- (6.5, 3.5);	
		\draw[->, black, dashed] (3, 4) -- (6.5, 4);	
		\draw[->, black, dashed] (3, 4.5) -- (6.5, 4.5);	
		\draw[->, black, dashed] (3, 5) -- (6.5, 5);
		\draw[->, black, dashed] (3, 5.5) -- (6.5, 5.5);
		
		\draw[-, black, dashed, very thick] (3, 3) -- (1, 5);

		\draw[-, black, dashed] (2.5, 3) -- (3, 3.5);
		\draw[-, black, dashed] (2, 3) -- (3, 4);
		\draw[-, black, dashed] (1.5, 3) -- (3, 4.5);
		\draw[-, black, dashed] (1, 3) -- (3, 5);
		\draw[-, black, dashed] (.5, 3) -- (3, 5.5);

		\draw[-, black, dashed] (.5, 0) -- (.5, 3);
		\draw[-, black, dashed] (1, 0) -- (1, 3);	
		\draw[-, black, dashed] (1.5, 0) -- (1.5, 3);	
		\draw[-, black, dashed] (2, 0) -- (2, 3);	
		\draw[-, black, dashed] (2.5, 0) -- (2.5, 3);
		
		\draw[->, black, dashed] (3, 0) -- (3, 6);
		\draw[->, black, dashed] (3.5, 0) -- (3.5, 6);	
		\draw[->, black, dashed] (4, 0) -- (4, 6);	
		\draw[->, black, dashed] (4.5, 0) -- (4.5, 6);	
		\draw[->, black, dashed] (5, 0) -- (5, 6);
		\draw[->, black, dashed] (5.5, 0) -- (5.5, 6);
		\draw[->, black, dashed] (6, 0) -- (6, 6);

		\filldraw[draw=black] (1.5, -.5) circle[radius = 0] node[scale = 1, below = 0]{$A + B$};
		\filldraw[draw=black] (4.5, -.5) circle[radius = 0] node[scale = 1, below = 0]{$B + C$};
		\filldraw[draw=black] (7, 4.5) circle[radius = 0] node[scale = 1, right = 0]{$A + B$};
		\filldraw[draw=black] (-.5, 1.75) circle[radius = 0] node[scale = 1, left = 0]{$A + C$};
		\draw[] (.5, -.125) -- (.5, -.375) -- (2.5, -.375) -- (2.5, -.125); 
		\draw[] (3, -.125) -- (3, -.375) -- (6, -.375) -- (6, -.125); 
		\draw[] (-.125, .5) -- (-.375, .5) -- (-.375, 3) -- (-.125, 3); 
		\draw[] (6.625, 3.5) -- (6.875, 3.5) -- (6.875, 5.5) -- (6.625, 5.5);

		\filldraw[fill=gray!50!white, draw=black] (.5, .5) circle [radius=.05];
		\filldraw[fill=gray!50!white, draw=black] (.5, 1) circle [radius=.05];
		\filldraw[fill=gray!50!white, draw=black] (.5, 1.5) circle [radius=.05];
		\filldraw[fill=gray!50!white, draw=black] (.5, 2) circle [radius=.05];
		\filldraw[fill=gray!50!white, draw=black] (.5, 2.5) circle [radius=.05];
		\filldraw[fill=gray!50!white, draw=black] (.5, 3) circle [radius=.05];
		
		\filldraw[fill=gray!50!white, draw=black] (1, .5) circle [radius=.05];
		\filldraw[fill=gray!50!white, draw=black] (1, 1) circle [radius=.05];
		\filldraw[fill=gray!50!white, draw=black] (1, 1.5) circle [radius=.05];
		\filldraw[fill=gray!50!white, draw=black] (1, 2) circle [radius=.05];
		\filldraw[fill=gray!50!white, draw=black] (1, 2.5) circle [radius=.05];
		\filldraw[fill=gray!50!white, draw=black] (1, 3) circle [radius=.05];

		\filldraw[fill=gray!50!white, draw=black] (1.5, .5) circle [radius=.05];
		\filldraw[fill=gray!50!white, draw=black] (1.5, 1) circle [radius=.05];
		\filldraw[fill=gray!50!white, draw=black] (1.5, 1.5) circle [radius=.05];
		\filldraw[fill=gray!50!white, draw=black] (1.5, 2) circle [radius=.05];
		\filldraw[fill=gray!50!white, draw=black] (1.5, 2.5) circle [radius=.05];
		\filldraw[fill=gray!50!white, draw=black] (1.5, 3) circle [radius=.05];

		\filldraw[fill=gray!50!white, draw=black] (2, .5) circle [radius=.05];
		\filldraw[fill=gray!50!white, draw=black] (2, 1) circle [radius=.05];
		\filldraw[fill=gray!50!white, draw=black] (2, 1.5) circle [radius=.05];
		\filldraw[fill=gray!50!white, draw=black] (2, 2) circle [radius=.05];
		\filldraw[fill=gray!50!white, draw=black] (2, 2.5) circle [radius=.05];
		\filldraw[fill=gray!50!white, draw=black] (2, 3) circle [radius=.05];

		\filldraw[fill=gray!50!white, draw=black] (2.5, .5) circle [radius=.05];
		\filldraw[fill=gray!50!white, draw=black] (2.5, 1) circle [radius=.05];
		\filldraw[fill=gray!50!white, draw=black] (2.5, 1.5) circle [radius=.05];
		\filldraw[fill=gray!50!white, draw=black] (2.5, 2) circle [radius=.05];
		\filldraw[fill=gray!50!white, draw=black] (2.5, 2.5) circle [radius=.05];
		\filldraw[fill=gray!50!white, draw=black] (2.5, 3) circle [radius=.05];

		\filldraw[fill=gray!50!white, draw=black] (3, .5) circle [radius=.05];
		\filldraw[fill=gray!50!white, draw=black] (3, 1) circle [radius=.05];
		\filldraw[fill=gray!50!white, draw=black] (3, 1.5) circle [radius=.05];
		\filldraw[fill=gray!50!white, draw=black] (3, 2) circle [radius=.05];
		\filldraw[fill=gray!50!white, draw=black] (3, 2.5) circle [radius=.05];
		\filldraw[fill=gray!50!white, draw=black] (3, 3) circle [radius=.05];
		\filldraw[fill=gray!50!white, draw=black] (3, 3.5) circle [radius=.05];
		\filldraw[fill=gray!50!white, draw=black] (3, 4) circle [radius=.05];
		\filldraw[fill=gray!50!white, draw=black] (3, 4.5) circle [radius=.05];
		\filldraw[fill=gray!50!white, draw=black] (3, 5) circle [radius=.05];
		\filldraw[fill=gray!50!white, draw=black] (3, 5.5) circle [radius=.05];

		\filldraw[fill=gray!50!white, draw=black] (3.5, .5) circle [radius=.05];
		\filldraw[fill=gray!50!white, draw=black] (3.5, 1) circle [radius=.05];
		\filldraw[fill=gray!50!white, draw=black] (3.5, 1.5) circle [radius=.05];
		\filldraw[fill=gray!50!white, draw=black] (3.5, 2) circle [radius=.05];
		\filldraw[fill=gray!50!white, draw=black] (3.5, 2.5) circle [radius=.05];
		\filldraw[fill=gray!50!white, draw=black] (3.5, 3) circle [radius=.05];
		\filldraw[fill=gray!50!white, draw=black] (3.5, 3.5) circle [radius=.05];
		\filldraw[fill=gray!50!white, draw=black] (3.5, 4) circle [radius=.05];
		\filldraw[fill=gray!50!white, draw=black] (3.5, 4.5) circle [radius=.05];
		\filldraw[fill=gray!50!white, draw=black] (3.5, 5) circle [radius=.05];
		\filldraw[fill=gray!50!white, draw=black] (3.5, 5.5) circle [radius=.05];

		\filldraw[fill=gray!50!white, draw=black] (4, .5) circle [radius=.05];
		\filldraw[fill=gray!50!white, draw=black] (4, 1) circle [radius=.05];
		\filldraw[fill=gray!50!white, draw=black] (4, 1.5) circle [radius=.05];
		\filldraw[fill=gray!50!white, draw=black] (4, 2) circle [radius=.05];
		\filldraw[fill=gray!50!white, draw=black] (4, 2.5) circle [radius=.05];
		\filldraw[fill=gray!50!white, draw=black] (4, 3) circle [radius=.05];
		\filldraw[fill=gray!50!white, draw=black] (4, 3.5) circle [radius=.05];
		\filldraw[fill=gray!50!white, draw=black] (4, 4) circle [radius=.05];
		\filldraw[fill=gray!50!white, draw=black] (4, 4.5) circle [radius=.05];
		\filldraw[fill=gray!50!white, draw=black] (4, 5) circle [radius=.05];
		\filldraw[fill=gray!50!white, draw=black] (4, 5.5) circle [radius=.05];

		\filldraw[fill=gray!50!white, draw=black] (4.5, .5) circle [radius=.05];
		\filldraw[fill=gray!50!white, draw=black] (4.5, 1) circle [radius=.05];
		\filldraw[fill=gray!50!white, draw=black] (4.5, 1.5) circle [radius=.05];
		\filldraw[fill=gray!50!white, draw=black] (4.5, 2) circle [radius=.05];
		\filldraw[fill=gray!50!white, draw=black] (4.5, 2.5) circle [radius=.05];
		\filldraw[fill=gray!50!white, draw=black] (4.5, 3) circle [radius=.05];
		\filldraw[fill=gray!50!white, draw=black] (4.5, 3.5) circle [radius=.05];
		\filldraw[fill=gray!50!white, draw=black] (4.5, 4) circle [radius=.05];
		\filldraw[fill=gray!50!white, draw=black] (4.5, 4.5) circle [radius=.05];
		\filldraw[fill=gray!50!white, draw=black] (4.5, 5) circle [radius=.05];
		\filldraw[fill=gray!50!white, draw=black] (4.5, 5.5) circle [radius=.05];

		\filldraw[fill=gray!50!white, draw=black] (5, .5) circle [radius=.05];
		\filldraw[fill=gray!50!white, draw=black] (5, 1) circle [radius=.05];
		\filldraw[fill=gray!50!white, draw=black] (5, 1.5) circle [radius=.05];
		\filldraw[fill=gray!50!white, draw=black] (5, 2) circle [radius=.05];
		\filldraw[fill=gray!50!white, draw=black] (5, 2.5) circle [radius=.05];
		\filldraw[fill=gray!50!white, draw=black] (5, 3) circle [radius=.05];
		\filldraw[fill=gray!50!white, draw=black] (5, 3.5) circle [radius=.05];
		\filldraw[fill=gray!50!white, draw=black] (5, 4) circle [radius=.05];
		\filldraw[fill=gray!50!white, draw=black] (5, 4.5) circle [radius=.05];
		\filldraw[fill=gray!50!white, draw=black] (5, 5) circle [radius=.05];
		\filldraw[fill=gray!50!white, draw=black] (5, 5.5) circle [radius=.05];

		\filldraw[fill=gray!50!white, draw=black] (5.5, .5) circle [radius=.05];
		\filldraw[fill=gray!50!white, draw=black] (5.5, 1) circle [radius=.05];
		\filldraw[fill=gray!50!white, draw=black] (5.5, 1.5) circle [radius=.05];
		\filldraw[fill=gray!50!white, draw=black] (5.5, 2) circle [radius=.05];
		\filldraw[fill=gray!50!white, draw=black] (5.5, 2.5) circle [radius=.05];
		\filldraw[fill=gray!50!white, draw=black] (5.5, 3) circle [radius=.05];
		\filldraw[fill=gray!50!white, draw=black] (5.5, 3.5) circle [radius=.05];
		\filldraw[fill=gray!50!white, draw=black] (5.5, 4) circle [radius=.05];
		\filldraw[fill=gray!50!white, draw=black] (5.5, 4.5) circle [radius=.05];
		\filldraw[fill=gray!50!white, draw=black] (5.5, 5) circle [radius=.05];
		\filldraw[fill=gray!50!white, draw=black] (5.5, 5.5) circle [radius=.05];

		\filldraw[fill=gray!50!white, draw=black] (6, .5) circle [radius=.05];
		\filldraw[fill=gray!50!white, draw=black] (6, 1) circle [radius=.05];
		\filldraw[fill=gray!50!white, draw=black] (6, 1.5) circle [radius=.05];
		\filldraw[fill=gray!50!white, draw=black] (6, 2) circle [radius=.05];
		\filldraw[fill=gray!50!white, draw=black] (6, 2.5) circle [radius=.05];
		\filldraw[fill=gray!50!white, draw=black] (6, 3) circle [radius=.05];
		\filldraw[fill=gray!50!white, draw=black] (6, 3.5) circle [radius=.05];
		\filldraw[fill=gray!50!white, draw=black] (6, 4) circle [radius=.05];
		\filldraw[fill=gray!50!white, draw=black] (6, 4.5) circle [radius=.05];
		\filldraw[fill=gray!50!white, draw=black] (6, 5) circle [radius=.05];
		\filldraw[fill=gray!50!white, draw=black] (6, 5.5) circle [radius=.05];
		\end{tikzpicture}
		
	\end{center}	
	
	\caption{\label{vertexdomain} The chart to the left shows all six possible arrow configurations. The three-bundle domain $\mathcal{T}_{A, B, C}$, with $(A, B, C) = (2, 3, 4)$, is depicted to the right; the defect line is shown as thick and dashed.} 
\end{figure}

Next, let us define six-vertex ensembles on $\mathcal{T}$, which were originally considered in Section 4.2 of \cite{PRC} and then later studied in Section 5 of \cite{ACSVMGD}. To that end, an \emph{arrow configuration} is a quadruple $(i_1, j_1; i_2, j_2)$ such that $i_1, j_1, i_2, j_2 \in \{ 0, 1 \}$ and $i_1 + j_1 = i_2 + j_2$. We view such a quadruple as an assignment of arrows to a vertex $v \in \mathcal{T}$. Specifically, $i_1$ and $j_1$ denote the numbers of vertical and horizontal arrows entering $v$, respectively; similarly, $i_2$ and $j_2$ denote the numbers of vertical and horizontal arrows exiting $v$, respectively. The fact that the $i_k$ and $j_k$ are in $\{ 0, 1 \}$ means that at most one arrow is permitted along any edge, and the fact that $i_1 + j_1 = i_2 + j_2$ means that the numbers of incoming and outgoing arrows at $v$ are equal. There are six possible arrow configurations, which are depicted on the left side of \Cref{vertexdomain}. 

Now, a \emph{(defective) six-vertex ensemble} on $\mathcal{T}$ is defined to be an assignment of an arrow configuration to each vertex of $\mathcal{T}$, satisfying the following two (in)consistency relations.

\begin{enumerate}
	\item If $v_1, v_2 \in \mathcal{T}$ are two adjacent vertices connected by a non-diagonal edge, then their arrow configurations are \emph{consistent}; this means that there is an edge to $v_2$ in the arrow configuration at $v_1$ if and only if there is an edge from $v_1$ in the arrow configuration at $v_2$.
	
	\item  If $v_1, v_2 \in \mathcal{T}$ are two adjacent vertices connected by a diagonal edge, then their arrow configurations are not consistent.
\end{enumerate}

To depict the second condition above diagrammatically, we draw a \emph{defect line} through the vertex $(A + B + 1, A + C)$ and orthogonal to each diagonal edge, across which the inconsistency of adjacent arrow configurations occurs; we refer to \Cref{figurevertexwedge} for an example. Observe in particular that the arrows in a (defective) six-vertex ensemble form directed paths connecting vertices of $\mathcal{T}$, which both emanate from and end at either the defect line or a boundary vertex of $\mathbb{Z}^2 \setminus \mathcal{T}$. 

\emph{Boundary data} for a (defective) six-vertex ensemble is prescribed by dictating which boundary vertices are entrance and exit sites for a directed path. We will be particularly interested in \emph{domain-wall boundary data}, in which $A + C$ paths horizontally enter $\mathcal{T}$ from sites of the form $(0, m)$ for $m \in [1, A + C]$, and $B + C$ paths vertically exit $\mathcal{T}$ at sites of the form $(A + B + m, 2A + B + C + 1)$ for $m \in [1, B + C]$; this is depicted in \Cref{figurevertexwedge}. 
	
In the special case $A = 0 = B$, the domain $\mathcal{T}$ becomes a $C \times C$ square $\mathcal{S} = \mathcal{S}_C$, and the domain-wall boundary data defined above becomes the more commonly studied domain-wall boundary data on $\mathcal{S}$, in which arrows enter from all sites along the left boundary and exit at all sites along the top boundary. We will discuss this point in more detail in \Cref{AB0Domain}.

\begin{rem}
	
	\label{apaths}
	
	The (in)consistency conditions quickly imply that, in any (defective) six-vertex ensemble on $\mathcal{T}_{A, B, C}$ with domain-wall boundary data, $A$ paths must end at the defect line and $B$ paths must start from it. For example, \Cref{figurevertexwedge} indicates that $A = 2$ paths end at and $B = 3$ paths start from the defect line. 
	
\end{rem}

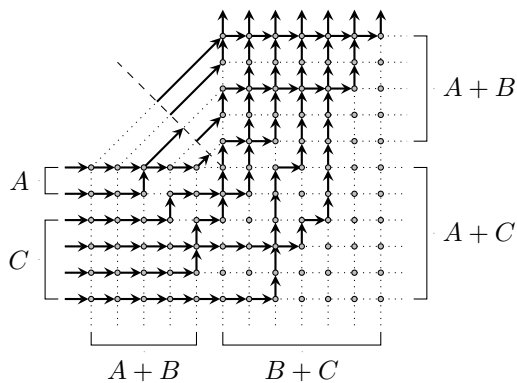
\begin{figure}[t]
	
	\begin{center}
		
		\begin{tikzpicture}[
		>=stealth,
		scale = .7
		]
		
		\draw[-, black, dotted] (0, .5) -- (6.5, .5);	
		\draw[-, black, dotted] (0, 1) -- (6.5, 1);	
		\draw[-, black, dotted] (0, 1.5) -- (6.5, 1.5);	
		\draw[-, black, dotted] (0, 2) -- (6.5, 2);
		\draw[-, black, dotted] (0, 2.5) -- (6.5, 2.5);
		\draw[-, black, dotted] (0, 3) -- (6.5, 3);

		\draw[-, black, dotted] (3, 3.5) -- (6.5, 3.5);	
		\draw[-, black, dotted] (3, 4) -- (6.5, 4);	
		\draw[-, black, dotted] (3, 4.5) -- (6.5, 4.5);	
		\draw[-, black, dotted] (3, 5) -- (6.5, 5);
		\draw[-, black, dotted] (3, 5.5) -- (6.5, 5.5);
		
		\draw[-, black, dashed] (3, 3) -- (1, 5);

		\draw[-, black, dotted] (2.5, 3) -- (3, 3.5);
		\draw[-, black, dotted] (2, 3) -- (2.5, 3.5);
		\draw[-, black, dotted] (1.5, 3) -- (3, 4.5);
		\draw[-, black, dotted] (1, 3) -- (2, 4);
		\draw[-, black, dotted] (.5, 3) -- (1.75, 4.25);

		\draw[-, black, dotted] (.5, 0) -- (.5, 3);
		\draw[-, black, dotted] (1, 0) -- (1, 3);	
		\draw[-, black, dotted] (1.5, 0) -- (1.5, 3);	
		\draw[-, black, dotted] (2, 0) -- (2, 3);	
		\draw[-, black, dotted] (2.5, 0) -- (2.5, 3);
		
		\draw[-, black, dotted] (3, 0) -- (3, 6);
		\draw[-, black, dotted] (3.5, 0) -- (3.5, 6);	
		\draw[-, black, dotted] (4, 0) -- (4, 6);	
		\draw[-, black, dotted] (4.5, 0) -- (4.5, 6);	
		\draw[-, black, dotted] (5, 0) -- (5, 6);
		\draw[-, black, dotted] (5.5, 0) -- (5.5, 6);
		\draw[-, black, dotted] (6, 0) -- (6, 6);

		\filldraw[draw=black] (-.5, 1.25) circle[radius = 0] node[scale = 1, left = 0]{$C$};
		\filldraw[draw=black] (-.5, 2.75) circle[radius = 0] node[scale = 1, left = 0]{$A$};
		\filldraw[draw=black] (1.5, -.5) circle[radius = 0] node[scale = 1, below = 0]{$A + B$};
		\filldraw[draw=black] (4.5, -.5) circle[radius = 0] node[scale = 1, below = 0]{$B + C$};
		\filldraw[draw=black] (7, 4.5) circle[radius = 0] node[scale = 1, right = 0]{$A + B$};
		\filldraw[draw=black] (7, 1.75) circle[radius = 0] node[scale = 1, right = 0]{$A + C$};
		\draw[] (.5, -.125) -- (.5, -.375) -- (2.5, -.375) -- (2.5, -.125); 
		\draw[] (3, -.125) -- (3, -.375) -- (6, -.375) -- (6, -.125); 
		\draw[] (-.125, .5) -- (-.375, .5) -- (-.375, 2) -- (-.125, 2); 
		\draw[] (-.125, 2.5) -- (-.375, 2.5) -- (-.375, 3) -- (-.125, 3); 
		\draw[] (6.625, 3.5) -- (6.875, 3.5) -- (6.875, 5.5) -- (6.625, 5.5);
		\draw[] (6.625, .5) -- (6.875, .5) -- (6.875, 3) -- (6.625, 3);

		\draw[->, black, thick] (0, .5) -- (.45, .5);	
		\draw[->, black, thick] (0, 1) -- (.45, 1);
		\draw[->, black, thick] (0, 1.5) -- (.45, 1.5);
		\draw[->, black, thick] (0, 2) -- (.45, 2);
		\draw[->, black, thick] (0, 2.5) -- (.45, 2.5);
		\draw[->, black, thick] (0, 3) -- (.45, 3);
		
		\draw[->, black, thick] (.55, .5) -- (.95, .5);	
		\draw[->, black, thick] (.55, 1) -- (.95, 1);
		\draw[->, black, thick] (.55, 1.5) -- (.95, 1.5);
		\draw[->, black, thick] (.55, 2) -- (.95, 2);
		\draw[->, black, thick] (.55, 2.5) -- (.95, 2.5);
		\draw[->, black, thick] (.55, 3) -- (.95, 3);
		
		\draw[->, black, thick] (1.05, .5) -- (1.45, .5);	
		\draw[->, black, thick] (1.05, 1) -- (1.45, 1);
		\draw[->, black, thick] (1.05, 1.5) -- (1.45, 1.5);
		\draw[->, black, thick] (1.05, 2) -- (1.45, 2);
		\draw[->, black, thick] (1.05, 2.5) -- (1.45, 2.5);
		\draw[->, black, thick] (1.05, 3) -- (1.45, 3);
		
		\draw[->, black, thick] (1.5, 2.55) -- (1.5, 2.95); 
		
		\draw[->, black, thick] (1.55, .5) -- (1.95, .5);	
		\draw[->, black, thick] (1.55, 1) -- (1.95, 1);
		\draw[->, black, thick] (1.55, 1.5) -- (1.95, 1.5);
		\draw[->, black, thick] (1.55, 2) -- (1.95, 2);
		\draw[->, black, thick] (1.55, 3) -- (1.95, 3);
		
		\draw[->, black, thick] (2.05, .5) -- (2.45, .5);	
		\draw[->, black, thick] (2.05, 1) -- (2.45, 1);
		\draw[->, black, thick] (2.05, 1.5) -- (2.45, 1.5);
		\draw[->, black, thick] (2.05, 2.5) -- (2.45, 2.5);
		\draw[->, black, thick] (2.05, 3) -- (2.45, 3);
		
		\draw[->, black, thick] (2.55, .5) -- (2.95, .5);
		\draw[->, black, thick] (2.55, 1.5) -- (2.95, 1.5);
		\draw[->, black, thick] (2.55, 2) -- (2.95, 2);
		\draw[->, black, thick] (2.55, 2.5) -- (2.95, 2.5);
		
		\draw[->, black, thick] (2, 2.05) -- (2, 2.45);
		
		\draw[->, black, thick] (2.5, 1.05) -- (2.5, 1.45);
		\draw[->, black, thick] (2.5, 1.55) -- (2.5, 1.95);
		
		\draw[->, black, thick] (2.5, 3) -- (2.75, 3.25);
		\draw[->, black, thick] (1.5, 3) -- (2.25, 3.75);
		
		\draw[->, black, thick] (2.5, 3.5) -- (2.965, 3.965);	
		\draw[->, black, thick] (2, 4) -- (2.965, 4.965);
		\draw[->, black, thick] (1.75, 4.25) -- (2.965, 5.465);

		\draw[->, black, thick] (3, 5.55) -- (3, 6);
		\draw[->, black, thick] (3.5, 5.55) -- (3.5, 6);
		\draw[->, black, thick] (4, 5.55) -- (4, 6);
		\draw[->, black, thick] (4.5, 5.55) -- (4.5, 6);
		\draw[->, black, thick] (5, 5.55) -- (5, 6);
		\draw[->, black, thick] (5.5, 5.55) -- (5.5, 6);
		\draw[->, black, thick] (6, 5.55) -- (6, 6);
		
		\draw[->, black, thick] (3, 5.05) -- (3, 5.45);
		\draw[->, black, thick] (3.5, 5.05) -- (3.5, 5.45);
		\draw[->, black, thick] (4, 5.05) -- (4, 5.45);
		\draw[->, black, thick] (4.5, 5.05) -- (4.5, 5.45);
		\draw[->, black, thick] (5, 5.05) -- (5, 5.45);
		\draw[->, black, thick] (5.5, 5.05) -- (5.5, 5.45);
		
		\draw[->, black, thick] (3.05, 5.5) -- (3.45, 5.5);
		\draw[->, black, thick] (3.55, 5.5) -- (3.95, 5.5);
		\draw[->, black, thick] (4.05, 5.5) -- (4.45, 5.5);
		\draw[->, black, thick] (4.55, 5.5) -- (4.95, 5.5);
		\draw[->, black, thick] (5.05, 5.5) -- (5.45, 5.5);
		\draw[->, black, thick] (5.55, 5.5) -- (5.95, 5.5);

		\draw[->, black, thick] (3.55, 4.5) -- (3.95, 4.5);
		\draw[->, black, thick] (4.05, 4.5) -- (4.45, 4.5);
		\draw[->, black, thick] (4.55, 4.5) -- (4.95, 4.5);
		\draw[->, black, thick] (5.05, 4.5) -- (5.45, 4.5);

		\draw[->, black, thick] (3, 4.05) -- (3, 4.45);
		\draw[->, black, thick] (3.5, 4.05) -- (3.5, 4.45);
		\draw[->, black, thick] (4, 4.05) -- (4, 4.45);
		\draw[->, black, thick] (4.5, 4.05) -- (4.5, 4.45);
		\draw[->, black, thick] (5, 4.05) -- (5, 4.45);

		\draw[->, black, thick] (3.5, 3.55) -- (3.5, 3.95);
		\draw[->, black, thick] (3.5, 3.05) -- (3.5, 3.45);
		\draw[->, black, thick] (3.5, 2.55) -- (3.5, 2.95);

		\draw[->, black, thick] (4, 3.55) -- (4, 3.95);
		
		\draw[->, black, thick] (3.55, 3.5) -- (3.95, 3.5);
		
		\draw[->, black, thick] (3, 3.05) -- (3, 3.45);

		\draw[->, black, thick] (4.5, 3.55) -- (4.5, 3.95);
		\draw[->, black, thick] (4.5, 3.05) -- (4.5, 3.45);
		
		\draw[->, black, thick] (4, 2.55) -- (4, 2.95);
		\draw[->, black, thick] (4, 2.05) -- (4, 2.55);
		\draw[->, black, thick] (4, 1.55) -- (4, 1.95);
		\draw[->, black, thick] (4, 1.05) -- (4, 1.55);
		\draw[->, black, thick] (4, .55) -- (4, .95);
		
		\draw[->, black, thick] (4.05, 3) -- (4.5, 3);

		\draw[->, black, thick] (3.55, .5) -- (3.95, .5); 
		\draw[->, black, thick] (3.55, 1.5) -- (3.95, 1.5);

		\draw[->, black, thick] (4.05, 1.5) -- (4.45, 1.5);

		\draw[->, black, thick] (4.5, 1.55) -- (4.5, 1.95);
		
		\draw[->, black, thick] (5, 2.05) -- (5, 2.45);

		\draw[->, black, thick] (4.55, 2) -- (4.95, 2);

		\draw[->, black, thick] (5, 3.55) -- (5, 3.95);
		\draw[->, black, thick] (5, 3.05) -- (5, 3.45);
		\draw[->, black, thick] (5, 2.55) -- (5, 2.95);

		\draw[->, black, thick] (3.5, 4.55) -- (3.5, 4.95);
		\draw[->, black, thick] (4, 4.55) -- (4, 4.95);
		\draw[->, black, thick] (4.5, 4.55) -- (4.5, 4.95);
		\draw[->, black, thick] (5, 4.55) -- (5, 4.95);
		\draw[->, black, thick] (5.5, 4.55) -- (5.5, 4.95);
		
		\draw[->, black, thick] (3.05, 4.5) -- (3.45, 4.5);
		\draw[->, black, thick] (3.05, 3.5) -- (3.45, 3.5);
		\draw[->, black, thick] (3.05, 2.5) -- (3.45, 2.5);
		\draw[->, black, thick] (3.05, 1.5) -- (3.45, 1.5);
		\draw[->, black, thick] (3.05, .5) -- (3.45, .5);

		\draw[->, black, thick] (3, 2.05) -- (3, 2.45);
		\draw[->, black, thick] (3, 2.55) -- (3, 2.95);

		\filldraw[fill=gray!50!white, draw=black] (.5, .5) circle [radius=.05];
		\filldraw[fill=gray!50!white, draw=black] (.5, 1) circle [radius=.05];
		\filldraw[fill=gray!50!white, draw=black] (.5, 1.5) circle [radius=.05];
		\filldraw[fill=gray!50!white, draw=black] (.5, 2) circle [radius=.05];
		\filldraw[fill=gray!50!white, draw=black] (.5, 2.5) circle [radius=.05];
		\filldraw[fill=gray!50!white, draw=black] (.5, 3) circle [radius=.05];
		
		\filldraw[fill=gray!50!white, draw=black] (1, .5) circle [radius=.05];
		\filldraw[fill=gray!50!white, draw=black] (1, 1) circle [radius=.05];
		\filldraw[fill=gray!50!white, draw=black] (1, 1.5) circle [radius=.05];
		\filldraw[fill=gray!50!white, draw=black] (1, 2) circle [radius=.05];
		\filldraw[fill=gray!50!white, draw=black] (1, 2.5) circle [radius=.05];
		\filldraw[fill=gray!50!white, draw=black] (1, 3) circle [radius=.05];

		\filldraw[fill=gray!50!white, draw=black] (1.5, .5) circle [radius=.05];
		\filldraw[fill=gray!50!white, draw=black] (1.5, 1) circle [radius=.05];
		\filldraw[fill=gray!50!white, draw=black] (1.5, 1.5) circle [radius=.05];
		\filldraw[fill=gray!50!white, draw=black] (1.5, 2) circle [radius=.05];
		\filldraw[fill=gray!50!white, draw=black] (1.5, 2.5) circle [radius=.05];
		\filldraw[fill=gray!50!white, draw=black] (1.5, 3) circle [radius=.05];

		\filldraw[fill=gray!50!white, draw=black] (2, .5) circle [radius=.05];
		\filldraw[fill=gray!50!white, draw=black] (2, 1) circle [radius=.05];
		\filldraw[fill=gray!50!white, draw=black] (2, 1.5) circle [radius=.05];
		\filldraw[fill=gray!50!white, draw=black] (2, 2) circle [radius=.05];
		\filldraw[fill=gray!50!white, draw=black] (2, 2.5) circle [radius=.05];
		\filldraw[fill=gray!50!white, draw=black] (2, 3) circle [radius=.05];

		\filldraw[fill=gray!50!white, draw=black] (2.5, .5) circle [radius=.05];
		\filldraw[fill=gray!50!white, draw=black] (2.5, 1) circle [radius=.05];
		\filldraw[fill=gray!50!white, draw=black] (2.5, 1.5) circle [radius=.05];
		\filldraw[fill=gray!50!white, draw=black] (2.5, 2) circle [radius=.05];
		\filldraw[fill=gray!50!white, draw=black] (2.5, 2.5) circle [radius=.05];
		\filldraw[fill=gray!50!white, draw=black] (2.5, 3) circle [radius=.05];

		\filldraw[fill=gray!50!white, draw=black] (3, .5) circle [radius=.05];
		\filldraw[fill=gray!50!white, draw=black] (3, 1) circle [radius=.05];
		\filldraw[fill=gray!50!white, draw=black] (3, 1.5) circle [radius=.05];
		\filldraw[fill=gray!50!white, draw=black] (3, 2) circle [radius=.05];
		\filldraw[fill=gray!50!white, draw=black] (3, 2.5) circle [radius=.05];
		\filldraw[fill=gray!50!white, draw=black] (3, 3) circle [radius=.05];
		\filldraw[fill=gray!50!white, draw=black] (3, 3.5) circle [radius=.05];
		\filldraw[fill=gray!50!white, draw=black] (3, 4) circle [radius=.05];
		\filldraw[fill=gray!50!white, draw=black] (3, 4.5) circle [radius=.05];
		\filldraw[fill=gray!50!white, draw=black] (3, 5) circle [radius=.05];
		\filldraw[fill=gray!50!white, draw=black] (3, 5.5) circle [radius=.05];

		\filldraw[fill=gray!50!white, draw=black] (3.5, .5) circle [radius=.05];
		\filldraw[fill=gray!50!white, draw=black] (3.5, 1) circle [radius=.05];
		\filldraw[fill=gray!50!white, draw=black] (3.5, 1.5) circle [radius=.05];
		\filldraw[fill=gray!50!white, draw=black] (3.5, 2) circle [radius=.05];
		\filldraw[fill=gray!50!white, draw=black] (3.5, 2.5) circle [radius=.05];
		\filldraw[fill=gray!50!white, draw=black] (3.5, 3) circle [radius=.05];
		\filldraw[fill=gray!50!white, draw=black] (3.5, 3.5) circle [radius=.05];
		\filldraw[fill=gray!50!white, draw=black] (3.5, 4) circle [radius=.05];
		\filldraw[fill=gray!50!white, draw=black] (3.5, 4.5) circle [radius=.05];
		\filldraw[fill=gray!50!white, draw=black] (3.5, 5) circle [radius=.05];
		\filldraw[fill=gray!50!white, draw=black] (3.5, 5.5) circle [radius=.05];

		\filldraw[fill=gray!50!white, draw=black] (4, .5) circle [radius=.05];
		\filldraw[fill=gray!50!white, draw=black] (4, 1) circle [radius=.05];
		\filldraw[fill=gray!50!white, draw=black] (4, 1.5) circle [radius=.05];
		\filldraw[fill=gray!50!white, draw=black] (4, 2) circle [radius=.05];
		\filldraw[fill=gray!50!white, draw=black] (4, 2.5) circle [radius=.05];
		\filldraw[fill=gray!50!white, draw=black] (4, 3) circle [radius=.05];
		\filldraw[fill=gray!50!white, draw=black] (4, 3.5) circle [radius=.05];
		\filldraw[fill=gray!50!white, draw=black] (4, 4) circle [radius=.05];
		\filldraw[fill=gray!50!white, draw=black] (4, 4.5) circle [radius=.05];
		\filldraw[fill=gray!50!white, draw=black] (4, 5) circle [radius=.05];
		\filldraw[fill=gray!50!white, draw=black] (4, 5.5) circle [radius=.05];

		\filldraw[fill=gray!50!white, draw=black] (4.5, .5) circle [radius=.05];
		\filldraw[fill=gray!50!white, draw=black] (4.5, 1) circle [radius=.05];
		\filldraw[fill=gray!50!white, draw=black] (4.5, 1.5) circle [radius=.05];
		\filldraw[fill=gray!50!white, draw=black] (4.5, 2) circle [radius=.05];
		\filldraw[fill=gray!50!white, draw=black] (4.5, 2.5) circle [radius=.05];
		\filldraw[fill=gray!50!white, draw=black] (4.5, 3) circle [radius=.05];
		\filldraw[fill=gray!50!white, draw=black] (4.5, 3.5) circle [radius=.05];
		\filldraw[fill=gray!50!white, draw=black] (4.5, 4) circle [radius=.05];
		\filldraw[fill=gray!50!white, draw=black] (4.5, 4.5) circle [radius=.05];
		\filldraw[fill=gray!50!white, draw=black] (4.5, 5) circle [radius=.05];
		\filldraw[fill=gray!50!white, draw=black] (4.5, 5.5) circle [radius=.05];

		\filldraw[fill=gray!50!white, draw=black] (5, .5) circle [radius=.05];
		\filldraw[fill=gray!50!white, draw=black] (5, 1) circle [radius=.05];
		\filldraw[fill=gray!50!white, draw=black] (5, 1.5) circle [radius=.05];
		\filldraw[fill=gray!50!white, draw=black] (5, 2) circle [radius=.05];
		\filldraw[fill=gray!50!white, draw=black] (5, 2.5) circle [radius=.05];
		\filldraw[fill=gray!50!white, draw=black] (5, 3) circle [radius=.05];
		\filldraw[fill=gray!50!white, draw=black] (5, 3.5) circle [radius=.05];
		\filldraw[fill=gray!50!white, draw=black] (5, 4) circle [radius=.05];
		\filldraw[fill=gray!50!white, draw=black] (5, 4.5) circle [radius=.05];
		\filldraw[fill=gray!50!white, draw=black] (5, 5) circle [radius=.05];
		\filldraw[fill=gray!50!white, draw=black] (5, 5.5) circle [radius=.05];

		\filldraw[fill=gray!50!white, draw=black] (5.5, .5) circle [radius=.05];
		\filldraw[fill=gray!50!white, draw=black] (5.5, 1) circle [radius=.05];
		\filldraw[fill=gray!50!white, draw=black] (5.5, 1.5) circle [radius=.05];
		\filldraw[fill=gray!50!white, draw=black] (5.5, 2) circle [radius=.05];
		\filldraw[fill=gray!50!white, draw=black] (5.5, 2.5) circle [radius=.05];
		\filldraw[fill=gray!50!white, draw=black] (5.5, 3) circle [radius=.05];
		\filldraw[fill=gray!50!white, draw=black] (5.5, 3.5) circle [radius=.05];
		\filldraw[fill=gray!50!white, draw=black] (5.5, 4) circle [radius=.05];
		\filldraw[fill=gray!50!white, draw=black] (5.5, 4.5) circle [radius=.05];
		\filldraw[fill=gray!50!white, draw=black] (5.5, 5) circle [radius=.05];
		\filldraw[fill=gray!50!white, draw=black] (5.5, 5.5) circle [radius=.05];

		\filldraw[fill=gray!50!white, draw=black] (6, .5) circle [radius=.05];
		\filldraw[fill=gray!50!white, draw=black] (6, 1) circle [radius=.05];
		\filldraw[fill=gray!50!white, draw=black] (6, 1.5) circle [radius=.05];
		\filldraw[fill=gray!50!white, draw=black] (6, 2) circle [radius=.05];
		\filldraw[fill=gray!50!white, draw=black] (6, 2.5) circle [radius=.05];
		\filldraw[fill=gray!50!white, draw=black] (6, 3) circle [radius=.05];
		\filldraw[fill=gray!50!white, draw=black] (6, 3.5) circle [radius=.05];
		\filldraw[fill=gray!50!white, draw=black] (6, 4) circle [radius=.05];
		\filldraw[fill=gray!50!white, draw=black] (6, 4.5) circle [radius=.05];
		\filldraw[fill=gray!50!white, draw=black] (6, 5) circle [radius=.05];
		\filldraw[fill=gray!50!white, draw=black] (6, 5.5) circle [radius=.05];

		\end{tikzpicture}
		
	\end{center}	
	
	\caption{\label{figurevertexwedge} 	A (defective) six-vertex ensemble on the three-bundle domain $\mathcal{T} = \mathcal{T}_{2, 3, 4}$ is depicted above. } 
\end{figure}

\begin{rem} 

One might observe that the domain $\mathcal{T} = \mathcal{T}_{A, B, C}$ is a bit reminiscent of the \emph{L-shaped domain} studied in several previous works \cite{TSVMD,ACFFSVMD} on the six-vertex model. However, we should clarify that these domains are not equivalent, since $\mathcal{T}$ includes diagonal edges, while the L-shaped domain does not. Furthermore, the six-vertex model on $\mathcal{T}$ involves defects, a property which is not shared by the six-vertex model on the L-shaped domain. These two points contribute to the differences between the arctic boundaries for the ice models on these domains.

\end{rem}

\subsection{The Arctic Boundary of the Ice Model on \texorpdfstring{$\mathcal{T}$}{}} 

\label{Boundary}

In this paper we will be interested in the geometry of a uniformly random (defective) six-vertex ensemble on $\mathcal{T}_{A, B, C}$ with domain-wall boundary data, as $A$, $B$, and $C$ tend to $\infty$. That this probability measure is uniform on the set of such ensembles corresponds to the \emph{ice point} of the six-vertex model on $\mathcal{T}$, in which each of the six arrow configurations depicted on the left side of \Cref{vertexdomain} are given equal weight.

More specifically, fix positive real numbers $a, b, c \ge 0$ such that $a + b + c = 1$, and let $N \ge 1$ be an integer. Define $A = \lfloor a N \rfloor$, $B = \lfloor b N \rfloor$, and $C = \lfloor c N \rfloor$, and consider a defective random six-vertex ensemble $\mathcal{E}$ with domain-wall boundary data on $\mathcal{T} = \mathcal{T}_{A, B, C}$. Since paths only enter and exit $\mathcal{T}$ through its left and top boundaries, respectively, one might expect (with high probability) that no paths in $\mathcal{E}$ should exist around some neighborhood of the southeast vertex $(A + 2B + C, 1)$ of $\mathcal{T}$. Stated equivalently, all vertices in $\mathcal{T}$ near its southeast corner should all be assigned arrow configuration $(0, 0; 0, 0)$; see, for instance, \Cref{figurevertexwedge}. 

Through similar reasoning, one might expect all vertices near the southwest corner $(1, 1)$ of $\mathcal{T}$ to be assigned $(0, 1; 0, 1)$ and all vertices near the northeast corner $(A + 2B + C, 2A + B + C)$ of $\mathcal{T}$ to be assigned $(1, 0; 1, 0)$. In addition, one might expect all vertices near one of the two northwest corners of $\mathcal{T}$ (namely $(1, A + C)$ or $(A + B + 1, 2A + B + C)$) to be assigned $(1, 1; 1, 1)$. However, in view of the fact that arrow configurations are inconsistent across the defect line, at most one of these two vertices can support the ``full'' arrow configuration $(1, 1; 1, 1)$. 

The existence of these ``frozen facets'' suggests that the (defective) six-vertex model on $\mathcal{T}$ exhibits an arctic boundary separating these frozen regions from ``liquid'' ones. The goal of this paper will be to understand this boundary explicitly. 

To make this more precise, we require some additional notation. To that end, for any $v = (x, y) \in \mathbb{R}^2$, define the four sets 
\begin{flalign} 
\label{vquadrants}
\begin{aligned}
& \text{NE} (v) = \big\{ (x', y') \in \mathbb{R}^2: x' \ge x, y' \ge y \big\}; \qquad \text{NW} (v) = \big\{ (x', y') \in \mathbb{R}^2: x' \le x, y' \ge y \big\}; \\
& \text{SE} (v) = \big\{ (x', y') \in \mathbb{R}^2: x' \ge x, y' \le y \big\}; \qquad \text{SW} (v) = \big\{ (x', y') \in \mathbb{R}^2: x' \le x, y' \le y \big\}.
\end{aligned}
\end{flalign}

\noindent In particular, $\text{NE} (v)$, $\text{NW} (v)$, $\text{SE} (v)$, and $\text{SW} (v)$ denote the subsets of $\mathbb{R}^2$ to the northeast, northwest, southeast, and southwest of $v$, respectively; see the left side of \Cref{avbv} for an example. 

\begin{figure}[t]
	
	\begin{center}
		
		\begin{tikzpicture}[
		>=stealth,
		]

		\fill[fill=white!65!gray] (3.5, -1.5) -- (4, -1.5) -- (7, 1.5) -- (7, 2) -- (3.5, 2) -- (3.5, -1.5);
		\fill[fill=white!65!gray] (4, -2) -- (4, -1.5) -- (7, 1.5) -- (7.5, 1.5) -- (7.5, -2) -- (4, -2);

		\fill[fill=white!65!gray] (8.5, 1.5) -- (9, 1.5) -- (12, -1.5) -- (12, -2) -- (8.5, -2) -- (8.5, 1.5);
		\fill[fill=white!65!gray] (9, 2) -- (9, 1.5) -- (12, -1.5) -- (12.5, -1.5) -- (12.5, 2) -- (9, 2);

		\draw[-, black] (-2, 0) -- (2, 0);
		\draw[-, black] (0, -2) -- (0, 2);

		\draw[-, black] (9, 1.5) -- (12, -1.5);

		\draw[fill=white, draw=black] (0, 0) circle [radius=.1] node[below = 5, left = 0, scale = .8]{$v$};
		\draw[fill=white, draw=black] (1, 1) circle [radius=0] node[scale = .8]{$\text{NE}(v)$};
		\draw[fill=white, draw=black] (-1, 1) circle [radius=0] node[scale = .8]{$\text{NW}(v)$};
		\draw[fill=white, draw=black] (1, -1) circle [radius=0] node[scale = .8]{$\text{SE}(v)$};
		\draw[fill=white, draw=black] (-1, -1) circle [radius=0] node[scale = .8]{$\text{SW}(v)$};
		
		\draw[-, black] (4, -1.5) -- (7, 1.5); 
		
		\draw[fill=white, draw=black] (4, -1.5) circle [radius=0] node[left = 6, below = 0,  scale = .8]{$\mathfrak{C}$};
		\draw[fill=white, draw=black] (6, -1) circle [radius=0] node[scale = .8]{$\text{SE}(\mathfrak{C})$};
		\draw[fill=white, draw=black] (5, 1) circle [radius=0] node[scale = .8]{$\text{NW}(\mathfrak{C})$};
		
		\draw[fill=white, draw=black] (9, 1.5) circle [radius=0] node[left = 6, above = 0, scale = .8]{$\mathfrak{C}$};
		\draw[fill=white, draw=black] (11, 1) circle [radius=0] node[scale = .8]{$\text{NE}(\mathfrak{C})$};
		\draw[fill=white, draw=black] (10, -1) circle [radius=0] node[scale = .8]{$\text{SW}(\mathfrak{C})$};
		
		\end{tikzpicture}
		
	\end{center}	
	
	\caption{\label{avbv} Depicted to the left are the four sets $\text{NE} (v)$, $\text{NW} (v)$, $\text{SE} (v)$, and $\text{SW} (v)$ at some $v \in \mathbb{R}^2$. Depicted in the middle is a nondecreasing curve $\mathfrak{C}$ and the two (shaded) sets $\text{NW} (\mathfrak{C})$ and $\text{SE} (\mathfrak{C})$. Depicted to the right is a nonincreasing curve $\mathfrak{C}$ and the two (shaded) sets $\text{NE} (\mathfrak{C})$ and $\text{SW} (\mathfrak{C})$.} 
\end{figure}
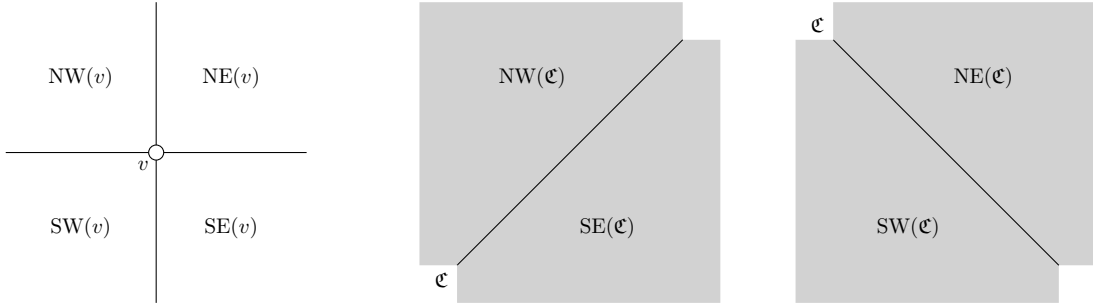

It will be useful to extend the definitions of these sets to situations in which $v$ is replaced by a curve. To that end, if $\mathfrak{C}$ is a nondecreasing curve in $\mathbb{R}^2$ (meaning that, if $v_1, v_2 \in \mathfrak{C}$, then $v_2 \in \text{NE} (v_1) \cup \text{SW} (v_1)$), then let
\begin{flalign} 
\label{c1quadrant} 
\begin{aligned}
\text{NW} (\mathfrak{C}) = \bigcup_{v \in \mathfrak{C}} \text{NW} (v); \qquad \text{SE} (\mathfrak{C}) = \bigcup_{v \in \mathfrak{C}} \text{SE} (v).
\end{aligned}
\end{flalign}

\noindent Similarly, if $\mathfrak{C}$ is a nonincreasing curve in $\mathbb{R}^2$ (meaning that, if $v_1, v_2 \in \mathfrak{C}$, then $v_2 \in \text{NW} (v_1) \cup \text{SE} (v_1)$), then let
\begin{flalign} 
\label{c2quadrant} 
\begin{aligned}
& \text{NE} (\mathfrak{C}) = \bigcup_{v \in \mathfrak{C}} \text{NE} (v); \qquad \text{SW} (\mathfrak{C}) = \bigcup_{v \in \mathfrak{C}} \text{SW} (v).
\end{aligned}
\end{flalign}

\noindent We refer to the middle and right sides of \Cref{avbv} for examples. 

Now, we can define the liquid and frozen regions of a (defective) six-vertex ensemble on $\mathcal{T}$, in a way similar to in \cite{RTAC}.

\begin{definition} 
	
\label{frozenregiondefinition}

Let $\mathcal{E}$ denote a (defective) six-vertex ensemble on $\mathcal{T}$ with domain-wall boundary data. Then, the \emph{frozen region of $\mathcal{E}$} consists of all vertices $v \in \mathcal{T}$ such that at least one of the following four conditions holds. 

\begin{enumerate}
	
	\item If $v' \in \text{NE} (v) \cap \mathcal{T}$, then the arrow configuration assigned to $v'$ is $(1, 0; 1, 0)$. 
	
	\item If $v' \in \text{NW} (v) \cap \mathcal{T}$, then the arrow configuration assigned to $v'$ is $(1, 1; 1, 1)$. 
	
	\item If $v' \in \text{SE} (v) \cap \mathcal{T}$, then the arrow configuration assigned to $v'$ is $(0, 0; 0, 0)$. 
	
	\item If $v' \in \text{SW} (v) \cap \mathcal{T}$, then the arrow configuration assigned to $v'$ is $(0, 1; 0, 1)$. 
\end{enumerate}

\noindent The \emph{liquid} region of $\mathcal{E}$ is defined to be the complement in $\mathcal{T}$ of the frozen region of $\mathcal{E}$. 

\end{definition}

\begin{rem}
	
\label{pathboundary} 

Observe in particular that the southeast part of the boundary between the liquid and frozen regions of a (defective) six-vertex ensemble $\mathcal{E}$ is given by the rightmost directed path of $\mathcal{E}$.
\end{rem}

We would like to explicitly evaluate the arctic boundary of a typical (defective) six-vertex model $\mathcal{E}$ on $\mathcal{T}$, namely the boundary between the liquid and frozen regions of $\mathcal{E}$. Upon scaling $\mathcal{T}$ by $\frac{1}{N}$, this will be the union of four algebraic curves, which were originally predicted by equations (5.10) and (5.11) of \cite{ACSVMGD} and are given by the following definition. 

\begin{definition} 

\label{zetacurves} 

Fix real numbers $a, b, c \ge 0$ such that $a + b + c = 1$. Define the subset $\mathfrak{T} = \mathfrak{T}_{a, b, c} \subset \mathbb{R}^2$ by 
\begin{flalign*}
\mathfrak{T} = \big( [0, 1 + b] \times [0, a + c] \big) \cup \big( [a + b, 1 + b] \times [a + c, 1 + a] \big).
\end{flalign*}

\noindent Observe in particular that this is the limit of the domain $N^{-1} \mathcal{T}_{A, B, C}$ as $N$ tends to $\infty$, where $A = \lfloor a N \rfloor$, $B = \lfloor b N \rfloor$, and $C = \lfloor c N \rfloor$, as above. 

Now, for any real number $z > 0$, define the function $\zeta (z) = \zeta (z; a, b, c)$ by 
\begin{flalign}
\label{zeta1z}
\zeta (z) = \sqrt{z^2 + z + 1} + \displaystyle\frac{\sqrt{(zb + zc + a + c)^2 - 4abz}}{2} + \displaystyle\frac{(b - c) z - a - c}{2} - 1, 
\end{flalign} 

\noindent so that 
\begin{flalign*}
\zeta' (z) = \displaystyle\frac{2z + 1}{2 \sqrt{z^2 + z + 1}} + \displaystyle\frac{(b + c)^2 z - ab + ac + bc + c^2}{2 \sqrt{(zb + zc + a + c)^2 - 4abz}} + \displaystyle\frac{b - c}{2}. 
\end{flalign*} 

\noindent Further define the functions $x (z) = x_{\text{SE}} (z) = x_{\text{SE}} (z; a, b, c)$ and $y (z) = y_{\text{SE}} (z) = y_{\text{SE}} (z; a, b, c)$ by
\begin{flalign}
\label{x1y1z}
x (z) = \zeta' (z); \qquad y(z) = z \zeta' (z) - \zeta (z),
\end{flalign}

\noindent and additionally set
\begin{flalign}
\label{xequations} 
\begin{aligned}
& x_{\text{SW}} (z) = y (z; b, c, a); \qquad \qquad \qquad y_{\text{SW}} (z) = 1 + c - x (z; b, c, a); \\
& x_{\text{NW}}^{(\text{W})} (z) = 1 + a - x (z; c, a, b); \qquad y_{\text{NW}}^{(\text{W})} (z) = 1 + c - y (z; c, a, b);  \\
& x_{\text{NW}}^{(\text{N})} (z) = 1 + b - x (z; b, c, a); \qquad y_{\text{NW}}^{(\text{N})} (z) = 1 + a - y (z; b, c, a);  \\
& x_{\text{NE}} (z) = 1 + b - y (z; c, a, b); \qquad \quad y_{\text{NE}} (z) = x (z; c, a, b). 
\end{aligned}
\end{flalign}

\noindent Now define the curves 
\begin{flalign}
\label{bcurvesequations} 
\begin{aligned}
& \mathfrak{B} = \mathfrak{B}_{\text{SE}} = \Big\{ \big( x (z), y(z) \big) \Big\}_{z \in [0, \infty]}; \\
& \mathfrak{B}_{\text{SW}} = \Big\{ \big( x_{\text{SW}} (z), y_{\text{SW}} (z) \big) \Big\}_{z \in [0, \infty]} \cap \mathfrak{T}; \qquad  \mathfrak{B}_{\text{NE}} = \Big\{ \big( x_{\text{NE}} (z), y_{\text{NE}} (z) \big) \Big\}_{z \in [0, \infty]} \cap \mathfrak{T}; \\
& \mathfrak{B}_{\text{NW}}^{(\text{W})} = \Big\{ \big( x_{\text{NW}}^{(\text{W})} (z), y_{\text{NW}}^{(\text{W})} (z) \big) \Big\}_{z \in [0, \infty]} \cap \mathfrak{T}; \qquad \mathfrak{B}_{\text{NW}}^{(\text{N})} = \Big\{ \big( x_{\text{NW}}^{(\text{N})} (z), y_{\text{NW}}^{(\text{N})} (z) \big) \Big\}_{z \in [0, \infty]} \cap \mathfrak{T},
\end{aligned}
\end{flalign}

\noindent and let $\mathfrak{B}_{\text{NW}} = \mathfrak{B}_{\text{NW}}^{(\text{N})} \cup \mathfrak{B}_{\text{NW}}^{(\text{W})}$.  
\end{definition}

\begin{rem}
	
	The curve $\mathfrak{B}$ is the Legendre transform of $\zeta$, formed by the convex envelope of the family of lines $\big\{ y = zx - \zeta (z) \big\}_{z > 0}$; see \Cref{lztangentb} below. 
	
\end{rem}

It can be quickly seen that $\zeta (z)$ and $\zeta' (z)$ both have finite limits as $z$ tends to $\infty$, so these curves are well-defined and closed. The curves $\mathfrak{B}_{\text{SE}}$, $\mathfrak{B}_{\text{SW}}$, $\mathfrak{B}_{\text{NE}}$, and $\mathfrak{B}_{\text{NW}}$ will be the southeast, southwest, northeast, and northwest parts of the arctic boundary. They are depicted in \Cref{c1c2c3figure2}, where on the left $(a, b, c) = \big( \frac{1}{2}, \frac{1}{4}, \frac{1}{4} \big)$ and on the right  $(a, b, c) = \big( \frac{1}{4}, \frac{1}{2}, \frac{1}{4} \big)$.

\begin{figure}[t]
	
	\begin{center}
		
		\begin{tikzpicture}[
		>=stealth,
		scale = 4
		]

		\draw[-, black] (0, 0) -- (1.25, 0) -- (1.25, 1.5) -- (.75, 1.5) -- (.75, .75) -- (0, .75) -- (0, 0);
		
		\filldraw[fill = black] (0, .5833) circle[radius = .01] node[right, scale = .75]{$\big(0, \frac{1}{2} + \frac{bc}{a + b} \big)$};
		\filldraw[fill = black] (1.25, .75) circle[radius = .01] node[right, scale = .75]{$\big(1 + b, \frac{1}{2} + \frac{ab}{b + c} \big)$};
		\filldraw[fill = black] (.5833, 0) circle[radius = .01] node[below, scale = .75]{$\big(0, \frac{1}{2} + \frac{bc}{a + c} \big)$};
		
		\filldraw[-,black, fill = white!75!gray] (0.583333, 0) -- (0.605406, 0.000274323) -- (0.626723, 0.00107214) -- (0.667136, 0.00409099) -- (0.73952, 0.0148578) -- (0.801643, 0.0303083) -- (0.8548, 0.0488442) -- (0.900267, 0.0692449) -- (0.95655, 0.101446) -- (1.00143, 0.133858) -- (1.04793, 0.175488) -- (1.08323, 0.214168) -- (1.0907, 0.223316) -- (1.12179, 0.265859) -- (1.14492, 0.303311) -- (1.16251, 0.336195) -- (1.1869, 0.390601) -- (1.20252, 0.433275) -- (1.22048, 0.495097) -- (1.22995, 0.537274) -- (1.23553, 0.567699) -- (1.24037, 0.600061) -- (1.24315, 0.622748) -- (1.24531, 0.64416) -- (1.24704, 0.665528) -- (1.24852, 0.689857) -- (1.24994, 0.737594) -- (1.25, .75) -- (1.25, 0) -- (.583333, 0);
		
		\filldraw[-,black, fill = white!50!gray] (0, 0.583333) -- (0.000273084, 0.561335) -- (0.0010626, 0.540233) -- (0.00402017, 0.500611) -- (0.0143703, 0.430935) -- (0.0288901, 0.37251) -- (0.0459412, 0.323587) -- (0.0643352, 0.282578) -- (0.0927054, 0.232962) -- (0.120554, 0.194386) -- (0.155421, 0.155421) -- (0.187017, 0.126571) -- (0.194386, 0.120554) -- (0.228189, 0.0958474) -- (0.257362, 0.077825) -- (0.282578, 0.0643352) -- (0.323587, 0.0459412) -- (0.355229, 0.0343596) -- (0.400428, 0.021223) -- (0.430935, 0.0143703) -- (0.452821, 0.0103614) -- (0.476024, 0.00688475) -- (0.492257, 0.00490181) -- (0.507566, 0.00335639) -- (0.522838, 0.00211757) -- (0.540233, 0.0010626) -- (.5833, 0) -- (0, 0) -- (0, .583333);
		
		\filldraw[-, black, fill = white!50!gray]	(1.25, 0.75) -- (1.24962, 0.780662) -- (1.24852, 0.810143) -- (1.24438, 0.865569) -- (1.22995, 0.962726) -- (1.20996, 1.0432) -- (1.1869, 1.1094) -- (1.16251, 1.16381) -- (1.12579, 1.22806) -- (1.0907, 1.27668) -- (1.04793, 1.32451) -- (1.01013, 1.35904) -- (1.00143, 1.36614) -- (0.962046, 1.39494) -- (0.928683, 1.41555) -- (0.900267, 1.43076) -- (0.8548, 1.45116) -- (0.820294, 1.46379) -- (0.771784, 1.47789) -- (0.75, 1.48296) -- (.75, 1.5) -- (1.25, 1.5) -- (1.25, .75);
		
		\filldraw[-, black, fill = white!25!gray] (0, .75) -- (0.0170366, 0.75) -- (0.010658, 0.716595) -- (0.00704441, 0.692481) -- (0.00499734, 0.675724) -- (0.00341033, 0.660005) -- (0.00214451, 0.644401) -- (0.00107214, 0.626723) -- (0.0000445001, 0.592254) -- (0, .5883) -- (0, 0) -- (0, .75);

		\draw[black] (1, .16) circle [radius = 0] node[left]{$\mathfrak{B}_{\text{SE}}$};
		\draw[black] (.17, .16) circle [radius = 0] node[right]{$\mathfrak{B}_{\text{SW}}$};
		\draw[black] (0, .685) circle [radius = 0] node[right]{$\mathfrak{B}_{\text{NW}}$};
		\draw[black] (.97, 1.35) circle [radius = 0] node[below]{$\mathfrak{B}_{\text{NE}}$};
		
		\draw[black] (0, 0) circle [radius = 0] node[below, scale = .75]{$(0, 0)$};
		\draw[black] (1.25, 0) circle [radius = 0] node[below, scale = .75]{$(1 + b, 0)$};
		\draw[black] (0, .75) circle [radius = 0] node[above, scale = .75]{$(0, a + c)$};
		\draw[black] (1.25, 1.5) circle [radius = 0] node[above, scale = .75]{$(1 + b, 1 + a)$};
		\draw[black] (.75, 1.5) circle [radius = 0] node[above, scale = .75]{$(a + b, 1 + a)$};

		\draw[-, black] (2, 0) -- (3.5, 0) -- (3.5, 1.25) -- (2.75, 1.25) -- (2.75, .5) -- (2, .5) -- (2, 0);

		\filldraw[fill = black] (2.9167, 1.25) circle[radius = .01] node[above, scale = .75]{$\big( \frac{1}{2} + b - c + \frac{bc}{a + b}, 1 + a \big)$};
		\filldraw[fill = black] (3.5, .6667) circle[radius = .01] node[left, scale = .75]{$\big(1 + b, \frac{1}{2} + \frac{ab}{b + c} \big)$};
		\filldraw[fill = black] (2.75, 0) circle[radius = .01] node[below, scale = .75]{$\big(0, \frac{1}{2} + \frac{bc}{a + c} \big)$};
		
		\filldraw[-, black,fill= white!75!gray] (2.75, 0) -- (2.78066, 0.000380822) -- (2.81014, 0.0014839) -- (2.86557, 0.00562129) -- (2.96273, 0.0200468) -- (3.0049, 0.02952) -- (3.0432, 0.0400353) -- (3.1094, 0.0630976) -- (3.16381, 0.0874917) -- (3.22806, 0.124209) -- (3.26191, 0.147856) -- (3.32451, 0.202073) -- (3.35904, 0.239874) -- (3.36614, 0.248572) -- (3.39494, 0.287954) -- (3.41555, 0.321317) -- (3.43076, 0.349733) -- (3.45116, 0.3952) -- (3.46379, 0.429706) -- (3.47789, 0.478216) -- (3.48514, 0.51048) -- (3.48934, 0.533405) -- (3.49296, 0.557519) -- (3.495, 0.574276) -- (3.49659, 0.589995) -- (3.49786, 0.605599) -- (3.49893, 0.623277) -- (3.5, .6667) -- (3.5, 0) -- (2.75, 0);
		
		\filldraw[-, black, fill = white!50!gray] (2, .5) -- (2.01702, 0.5) -- (2.03031, 0.448357) -- (2.04884, 0.3952) -- (2.06924, 0.349733) -- (2.10145, 0.29345) -- (2.17549, 0.202073) -- (2.21417, 0.166766) -- (2.22332, 0.159297) -- (2.26586, 0.128213) -- (2.30331, 0.105081) -- (2.33619, 0.0874917) -- (2.3906, 0.0630976) -- (2.43327, 0.0474815) -- (2.4951, 0.02952) -- (2.53727, 0.0200468) -- (2.5677, 0.0144743) -- (2.60006, 0.00962539) -- (2.62275, 0.00685428) -- (2.64416, 0.00469272) -- (2.66553, 0.00295948) -- (2.68986, 0.0014839) -- (2.75, 0) -- (2, 0) -- (2, .5);

		\filldraw[-, black, fill= white!50!gray] (3.5, 0.666667) -- (3.49973, 0.688665) -- (3.49894, 0.709767) -- (3.49598, 0.749389) -- (3.48563, 0.819065) -- (3.47111, 0.87749) -- (3.45406, 0.926413) -- (3.43566, 0.967422) -- (3.40729, 1.01704) -- (3.38127, 1.05333) -- (3.34458, 1.09458) -- (3.31298, 1.12343) -- (3.30561, 1.12945) -- (3.27181, 1.15415) -- (3.24264, 1.17218) -- (3.21742, 1.18566) -- (3.17641, 1.20406) -- (3.14477, 1.21564) -- (3.09957, 1.22878) -- (3.06907, 1.23563) -- (3.04718, 1.23964) -- (3.02398, 1.24312) -- (3.00774, 1.2451) -- (2.99243, 1.24664) -- (2.97716, 1.24788) -- (2.95977, 1.24894) -- (2.9167, 1.25) -- (3.5, 1.25) -- (3.5, .666667);
		
		\filldraw[-, black, fill=white!25!gray] (2.916667, 1.25) -- (2.894594, 1.24973) -- (2.873277, 1.24893) -- (2.832864, 1.24591) -- (2.76048, 1.23514) -- (2.75, 1.23296) -- (2.75, 1.25) -- (2.916667, 1.25);

		\draw[black] (3.25, .18) circle [radius = 0] node[left]{$\mathfrak{B}_{\text{SE}}$};
		\draw[black] (2.22, .18) circle [radius = 0] node[right]{$\mathfrak{B}_{\text{SW}}$};
		\draw[black] (2.87, 1.25) circle [radius = 0] node[below]{$\mathfrak{B}_{\text{NW}}$};
		\draw[black] (3.25, 1.15) circle [radius = 0] node[below]{$\mathfrak{B}_{\text{NE}}$};

		\end{tikzpicture}
		
	\end{center}	
	
	\caption{\label{c1c2c3figure2} Depicted above are the four curves $\mathfrak{B}_{\text{SE}}$, $\mathfrak{B}_{\text{SW}}$, $\mathfrak{B}_{\text{NE}}$, and $\mathfrak{B}_{\text{NW}}$. The situation to the left is a case when $a > b$, so that $\mathfrak{B}_{\text{NW}} = \mathfrak{B}_{\text{NW}}^{(\text{W})}$. The situation to the left is a case when $a < b$, so that $\mathfrak{B}_{\text{NW}} = \mathfrak{B}_{\text{NW}}^{(\text{N})}$.} 
\end{figure}
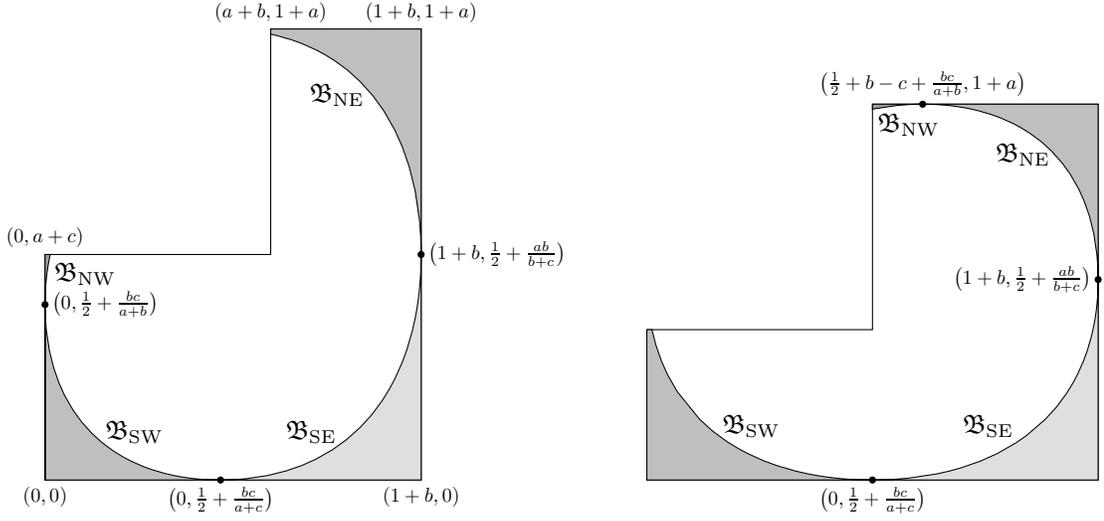

\begin{rem}
	
	\label{bcurves1}
	
	It can quickly be seen that $\mathfrak{B} = \mathfrak{B}_{\text{SE}}$ is convex and increasing and, moreover, that it is tangent to the $x$-axis and the line $x = 1 + b$ at 
	\begin{flalign*} 
	\big( x(0), y(0) \big) = \left( \frac{1}{2} + \frac{bc}{a + c}, 0 \right) \quad \text{and} \quad \lim_{z \rightarrow \infty} \big( x (z), y(z) \big) = \left( 1 + b, \frac{1}{2} + \frac{ab}{b + c} \right),
	\end{flalign*}
	
	\noindent respectively; see \Cref{c1c2c3figure2}. The fact that these two endpoints lie on the boundary of $\mathfrak{T}$ in particular implies that $\mathfrak{B} \subset \mathfrak{T}$. 

	Then the identity \eqref{xequations} implies that $\mathfrak{B}_{\text{SW}}$ and $\mathfrak{B}_{\text{NE}}$ are decreasing and that $\mathfrak{B}_{\text{NW}}$ is increasing. They furthermore imply that $\mathfrak{B}_{\text{SW}}$ and $\mathfrak{B}_{\text{SE}}$ are both tangent to the $x$-axis at the same point and that $\mathfrak{B}_{\text{NE}}$ and $\mathfrak{B}_{\text{SE}}$ are both tangent to the line $x = 1 + b$ at the same point. See \Cref{c1c2c3figure2}. 
\end{rem}

\begin{rem}
	
	\label{bcurves2}

	It can also be quickly verified from \eqref{xequations} and \Cref{bcurves1} that, if $a > b$, then $\mathfrak{B}_{\text{NE}}$ does not intersect the line $y = 1 + a$; that $\mathfrak{B}_{\text{NW}}^{(\text{N})}$ is empty (meaning that $\mathfrak{B}_{\text{NW}} = \mathfrak{B}_{\text{NW}}^{(\text{W})}$); and that $\mathfrak{B}_{\text{SW}}$ and $\mathfrak{B}_{\text{NW}}$ are both tangent to the $y$-axis at $\big( 0, \frac{1}{2} + \frac{bc}{a + b}\big)$. It can similarly be checked that, if $a < b$, then $\mathfrak{B}_{\text{SW}}$ does not intersect the $y$-axis; that $\mathfrak{B}_{\text{NW}}^{(\text{W})}$ is empty (meaning that $\mathfrak{B}_{\text{NW}} = \mathfrak{B}_{\text{NW}}^{(\text{N})}$); and that $\mathfrak{B}_{\text{NE}}$ and $\mathfrak{B}_{\text{NW}}$ are both tangent to the line $y = 1 + a$ at $\big( \frac{1}{2} + b - c + \frac{bc}{a + b}, 1 + a\big)$. If $a = b$, then $\mathfrak{B}_{\text{NW}}$ is empty; $\mathfrak{B}_{\text{SW}}$ is tangent to the $y$-axis at $\big( 0, \frac{1}{2} + \frac{c}{2}\big)$; and $\mathfrak{B}_{\text{NE}}$ is tangent to the line $y = 1 + a$ at $\big( \frac{1}{2} + a - \frac{c}{2}, 1 + a \big)$. 
	
	We refer to \Cref{c1c2c3figure2} for examples of the first and second of these three phenomena.  
	
\end{rem}

Now the following theorem establishes that $\mathfrak{B}_{\text{SE}} \cup \mathfrak{B}_{\text{SW}} \cup \mathfrak{B}_{\text{NE}} \cup \mathfrak{B}_{\text{NW}}$ is the arctic boundary of the (defective) ice-model on $\mathcal{T}_{A, B, C}$; this was originally predicted in Section 5.4 of \cite{ACSVMGD}. In what follows, $d (v, u) = \| v - u \|_2$ denotes the Euclidean distance between any two points $v, u \in \mathbb{R}^2$ and, for any sets $\mathcal{S}_1, \mathcal{S}_2 \subseteq \mathbb{R}^2$, we set $d (\mathcal{S}_1, \mathcal{S}_2) = \inf_{v_1 \in \mathcal{S}_1} \inf_{v_2 \in \mathcal{S}_2} d (v_1, v_2)$. Furthermore, for any subset $\mathcal{S} \subseteq \mathbb{R}^2$ and constant $R > 0$, let $R \mathcal{S}$ denote the set of points of the form $(Rx, Ry)$ for some $(x, y) \in \mathcal{S}$.

\begin{thm}
	
	\label{boundaryb}
	
	Fix real numbers $a, b, c \in (0, 1)$ such that $a + b + c = 1$; let $N \ge 1$ be an integer; denote $A = \lfloor a N \rfloor$, $B = \lfloor b N \rfloor$, and $C = \lfloor c N \rfloor$; and let $\delta \in (N^{-1 / 30}, 1)$ be a real number (possibly dependent on $N$). There exists a (small) constant $\gamma = \gamma (a, b, c) > 0$, only dependent on $a$, $b$, and $c$, such that the following holds. 
	
	Let $\mathcal{E}$ be a (defective) six-vertex ensemble on $\mathcal{T} = \mathcal{T}_{A, B, C}$ with domain-wall boundary data, chosen uniformly at random. Then, off of an event with probability at most $\gamma^{-1} \exp \big(- \gamma \delta^{24} N \big)$, the following statement holds. 
	
	Let $v \in \mathcal{T}$ be any vertex such that 
	\begin{flalign*}
	d \big( N^{-1} v, \mathfrak{B}_{\text{\emph{SE}}} \cup \mathfrak{B}_{\text{\emph{SW}}} \cup \mathfrak{B}_{\text{\emph{NE}}} \cup \mathfrak{B}_{\text{\emph{NW}}} \big) > \delta.
	\end{flalign*}
	
	\noindent If $v$ is outside of the curve $N \mathfrak{B}_{\text{\emph{SE}}} \cup N \mathfrak{B}_{\text{\emph{SW}}} \cup N \mathfrak{B}_{\text{\emph{NE}}} \cup N \mathfrak{B}_{\text{\emph{NW}}}$, that is, if
	\begin{flalign*}
	N^{-1} v \in \text{\emph{SE}} \big( \mathfrak{B}_{\text{\emph{SE}}} \big) \cup \text{\emph{SW}} \big( \mathfrak{B}_{\text{\emph{SW}}} \big) \cup\text{\emph{NE}} \big( \mathfrak{B}_{\text{\emph{NE}}} \big) \cup \text{\emph{NW}} \big( \mathfrak{B}_{\text{\emph{NW}}} \big), 
	\end{flalign*}
	
	\noindent then $v$ is in the frozen region of $\mathcal{E}$. Otherwise, $v$ is in the liquid region of $\mathcal{E}$. 
	
\end{thm}

In addition to identifying the arctic boundary of the ice model on the three-bundle domain, \Cref{boundaryb} can be used together with the Borel-Cantelli lemma to show a law of large numbers limit for the bottommost path of this model. More specifically, after scaling by $N^{-1}$, this path almost surely converges as $N$ tends to $\infty$ to the curve 
\begin{flalign*} 
\bigg[ (0, 0), \Big( \displaystyle\frac{1}{2} + \frac{bc}{a + c}, 0 \Big) \bigg] \cup \mathfrak{B}_{\text{SE}} \cup \bigg[ \Big( 1 + b, \displaystyle\frac{1}{2} + \displaystyle\frac{ab}{b + c}\Big), (1 + b, 1 + a) \bigg].
\end{flalign*}

\noindent One can also likely use the methods of this paper to improve the range of $\delta$ and the error probability in \Cref{boundaryb}, but we will not attempt to optimize these quantities here.

\subsection{Arctic Boundary for the Domain-Wall Ice Model on a Square} 

\label{AB0Domain}

As mentioned previously, in the case $A = B = 0$, the three-bundle domain $\mathcal{T} = \mathcal{T}_{A, B, C}$ degenerates to a $C \times C$ square subgraph of $\mathbb{Z}^2$ (in particular, it has no triangular face). Let us relabel $C$ by $N$ and define the directed subgraph\footnote{Here, we always orient the edges of $\mathbb{Z}^2$ to the north or to the east; this makes $\mathbb{Z}^2$ into a directed graph. } $\mathcal{S} = \mathcal{S}_N \subset \mathbb{Z}^2$ whose vertices consist of all lattice points in the square $[1, N] \times [1, N]$. 

A \emph{six-vertex ensemble} on $\mathcal{S}$ is again an assignment of an arrow configuration to each vertex of $\mathcal{S}$, such that each pair of neighboring arrow configurations is consistent; unlike for (defective) six-vertex ensembles on $\mathcal{T}$, there is no inconsistency condition. Under \emph{domain-wall boundary data}, $N$ paths enter $\mathcal{S}$ horizontally at sites on the line $x = 0$ and vertically exit $\mathcal{S}$ at sites on the line $y = N + 1$. We refer to the left side of \Cref{domainvertex} for an example. 

Let us describe how to use the content of \Cref{Boundary} to explicitly evaluate the arctic boundary for a uniformly random six-vertex ensemble on $\mathcal{S}$ with domain-wall boundary data. To that end, let us adopt the notation from that section and set $a = 0 = b$; further denote the square $\mathfrak{S} = [0, 1] \times [0, 1] \subset \mathbb{R}$, which is the limit of the rescaled domain $N^{-1} \mathcal{S}_N$ as $N$ tends to $\infty$. 

Then, the definition \eqref{zeta1z} for $\zeta (z)$ implies that 
\begin{flalign*}
\zeta (z) = \sqrt{z^2 + z + 1} - 1; \qquad \zeta' (z) = \displaystyle\frac{2z + 1}{2 \sqrt{z^2 + z + 1}},
\end{flalign*}

\noindent so that 
\begin{flalign*}
\big( x(z), y(z) \big) = \left(\displaystyle\frac{2z + 1}{2 \sqrt{z^2 + z + 1}}, \displaystyle\frac{2 \sqrt{z^2 + z + 1} - z - 2}{2 \sqrt{z^2 + z + 1}} \right),
\end{flalign*}

\noindent which parameterizes a part of an ellipse given by
\begin{flalign*}
\mathfrak{A} = \mathfrak{A}_{\text{SE}} = \big\{ (x, y) \in \mathbb{R}^2: (2x - 1)^2 + (2y - 1)^2 - 4(1 - x)y = 1 \big\} \cap \Bigg( \bigg[ \displaystyle\frac{1}{2}, 1 \bigg] \times \bigg[ 0, \displaystyle\frac{1}{2} \bigg] \Bigg).
\end{flalign*}

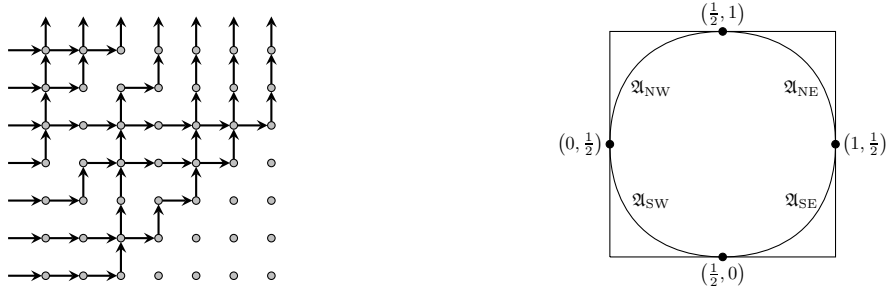
\begin{figure}[t]
	
	\begin{center}
		
		\begin{tikzpicture}[
		>=stealth,
		]

		\draw[->,black, thick] (0, .5) -- (.45, .5);
		\draw[->,black, thick] (0, 1) -- (.45, 1);
		\draw[->,black, thick] (0, 1.5) -- (.45, 1.5);
		\draw[->,black, thick] (0, 2) -- (.45, 2);
		\draw[->,black, thick] (0, 2.5) -- (.45, 2.5);
		\draw[->,black, thick] (0, 3) -- (.45, 3);
		\draw[->,black, thick] (0, 3.5) -- (.45, 3.5);

		\draw[->,black, thick] (.55, .5) -- (.95, .5);
		\draw[->,black, thick] (.55, 1) -- (.95, 1);
		\draw[->,black, thick] (.55, 1.5) -- (.95, 1.5);
		\draw[->,black, thick] (.55, 2.5) -- (.95, 2.5);
		\draw[->,black, thick] (.55, 3) -- (.95, 3);
		\draw[->,black, thick] (.55, 3.5) -- (.95, 3.5);
		\draw[->,black, thick] (.5, 2.05) -- (.5, 2.45);
		\draw[->,black, thick] (.5, 2.55) -- (.5, 2.95);
		\draw[->,black, thick] (.5, 3.05) -- (.5, 3.45);
		\draw[->,black, thick] (.5, 3.55) -- (.5, 3.95);
		
		\draw[->,black, thick] (1.05, .5) -- (1.45, .5);
		\draw[->,black, thick] (1.05, 1) -- (1.45, 1);
		\draw[->,black, thick] (1.05, 2) -- (1.45, 2);
		\draw[->,black, thick] (1.05, 2.5) -- (1.45, 2.5);
		\draw[->,black, thick] (1.05, 3.5) -- (1.45, 3.5);
		\draw[->,black, thick] (1, 1.55) -- (1, 1.95);
		\draw[->,black, thick] (1, 3.05) -- (1, 3.45);
		\draw[->,black, thick] (1, 3.55) -- (1, 3.95);
		
		\draw[->,black, thick] (1.5, .55) -- (1.5, .95);
		\draw[->,black, thick] (1.5, 1.05) -- (1.5, 1.45);
		\draw[->,black, thick] (1.5, 1.55) -- (1.5, 1.95);
		\draw[->,black, thick] (1.5, 2.05) -- (1.5, 2.45);
		\draw[->,black, thick] (1.5, 2.55) -- (1.5, 2.95);
		\draw[->,black, thick] (1.5, 3.55) -- (1.5, 3.95);
		\draw[->,black, thick] (1.55, 1) -- (1.95, 1);
		\draw[->,black, thick] (1.55, 2) -- (1.95, 2);
		\draw[->,black, thick] (1.55, 2.5) -- (1.95, 2.5);
		\draw[->,black, thick] (1.55, 3) -- (1.95, 3);

		\draw[->,black, thick] (2, 1.05) -- (2, 1.45);
		\draw[->,black, thick] (2, 3.05) -- (2, 3.45);
		\draw[->,black, thick] (2, 3.55) -- (2, 3.95);
		\draw[->,black, thick] (2.05, 1.5) -- (2.45, 1.5);
		\draw[->,black, thick] (2.05, 2) -- (2.45, 2);
		\draw[->,black, thick] (2.05, 2.5) -- (2.45, 2.5);

		\draw[->,black, thick] (2.5, 1.55) -- (2.5, 1.95);
		\draw[->,black, thick] (2.5, 2.05) -- (2.5, 2.45);
		\draw[->,black, thick] (2.5, 2.55) -- (2.5, 2.95);
		\draw[->,black, thick] (2.5, 3.05) -- (2.5, 3.45);
		\draw[->,black, thick] (2.5, 3.55) -- (2.5, 3.95);
		\draw[->,black, thick] (2.55, 2) -- (2.95, 2);
		\draw[->,black, thick] (2.55, 2.5) -- (2.95, 2.5);

		\draw[->, black, thick] (3, 2.05) -- (3, 2.45); 
		\draw[->,black, thick] (3, 2.55) -- (3, 2.95);
		\draw[->,black, thick] (3, 3.05) -- (3, 3.45);
		\draw[->,black, thick] (3, 3.55) -- (3, 3.95);
		\draw[->, black, thick] (3.05, 2.5) -- (3.45, 2.5);

		\draw[->,black, thick] (3.5, 2.55) -- (3.5, 2.95);
		\draw[->,black, thick] (3.5, 3.05) -- (3.5, 3.45);
		\draw[->,black, thick] (3.5, 3.55) -- (3.5, 3.95);

		\filldraw[fill=gray!50!white, draw=black] (.5, .5) circle [radius=.05];
		\filldraw[fill=gray!50!white, draw=black] (.5, 1) circle [radius=.05];
		\filldraw[fill=gray!50!white, draw=black] (.5, 1.5) circle [radius=.05];
		\filldraw[fill=gray!50!white, draw=black] (.5, 2) circle [radius=.05];
		\filldraw[fill=gray!50!white, draw=black] (.5, 2.5) circle [radius=.05];
		\filldraw[fill=gray!50!white, draw=black] (.5, 3) circle [radius=.05];
		\filldraw[fill=gray!50!white, draw=black] (.5, 3.5) circle [radius=.05];

		\filldraw[fill=gray!50!white, draw=black] (1, .5) circle [radius=.05];
		\filldraw[fill=gray!50!white, draw=black] (1, 1) circle [radius=.05];
		\filldraw[fill=gray!50!white, draw=black] (1, 1.5) circle [radius=.05];
		\filldraw[fill=gray!50!white, draw=black] (1, 2) circle [radius=.05];
		\filldraw[fill=gray!50!white, draw=black] (1, 2.5) circle [radius=.05];
		\filldraw[fill=gray!50!white, draw=black] (1, 3) circle [radius=.05];
		\filldraw[fill=gray!50!white, draw=black] (1, 3.5) circle [radius=.05];
		
		\filldraw[fill=gray!50!white, draw=black] (1.5, .5) circle [radius=.05];
		\filldraw[fill=gray!50!white, draw=black] (1.5, 1) circle [radius=.05];
		\filldraw[fill=gray!50!white, draw=black] (1.5, 1.5) circle [radius=.05];
		\filldraw[fill=gray!50!white, draw=black] (1.5, 2) circle [radius=.05];
		\filldraw[fill=gray!50!white, draw=black] (1.5, 2.5) circle [radius=.05];
		\filldraw[fill=gray!50!white, draw=black] (1.5, 3) circle [radius=.05];
		\filldraw[fill=gray!50!white, draw=black] (1.5, 3.5) circle [radius=.05];
		
		\filldraw[fill=gray!50!white, draw=black] (2, .5) circle [radius=.05];
		\filldraw[fill=gray!50!white, draw=black] (2, 1) circle [radius=.05];
		\filldraw[fill=gray!50!white, draw=black] (2, 1.5) circle [radius=.05];
		\filldraw[fill=gray!50!white, draw=black] (2, 2) circle [radius=.05];
		\filldraw[fill=gray!50!white, draw=black] (2, 2.5) circle [radius=.05];
		\filldraw[fill=gray!50!white, draw=black] (2, 3) circle [radius=.05];
		\filldraw[fill=gray!50!white, draw=black] (2, 3.5) circle [radius=.05];
		
		\filldraw[fill=gray!50!white, draw=black] (2.5, .5) circle [radius=.05];
		\filldraw[fill=gray!50!white, draw=black] (2.5, 1) circle [radius=.05];
		\filldraw[fill=gray!50!white, draw=black] (2.5, 1.5) circle [radius=.05];
		\filldraw[fill=gray!50!white, draw=black] (2.5, 2) circle [radius=.05];
		\filldraw[fill=gray!50!white, draw=black] (2.5, 2.5) circle [radius=.05];
		\filldraw[fill=gray!50!white, draw=black] (2.5, 3) circle [radius=.05];
		\filldraw[fill=gray!50!white, draw=black] (2.5, 3.5) circle [radius=.05];
		
		\filldraw[fill=gray!50!white, draw=black] (3, .5) circle [radius=.05];
		\filldraw[fill=gray!50!white, draw=black] (3, 1) circle [radius=.05];
		\filldraw[fill=gray!50!white, draw=black] (3, 1.5) circle [radius=.05];
		\filldraw[fill=gray!50!white, draw=black] (3, 2) circle [radius=.05];
		\filldraw[fill=gray!50!white, draw=black] (3, 2.5) circle [radius=.05];
		\filldraw[fill=gray!50!white, draw=black] (3, 3) circle [radius=.05];
		\filldraw[fill=gray!50!white, draw=black] (3, 3.5) circle [radius=.05];
		
		\filldraw[fill=gray!50!white, draw=black] (3.5, .5) circle [radius=.05];
		\filldraw[fill=gray!50!white, draw=black] (3.5, 1) circle [radius=.05];
		\filldraw[fill=gray!50!white, draw=black] (3.5, 1.5) circle [radius=.05];
		\filldraw[fill=gray!50!white, draw=black] (3.5, 2) circle [radius=.05];
		\filldraw[fill=gray!50!white, draw=black] (3.5, 2.5) circle [radius=.05];
		\filldraw[fill=gray!50!white, draw=black] (3.5, 3) circle [radius=.05];
		\filldraw[fill=gray!50!white, draw=black] (3.5, 3.5) circle [radius=.05];

		\draw[-, black] (8, .75) -- (11, .75) -- (11, 3.75) -- (8, 3.75) -- (8, .75); 	
		
		\draw[-, black] (9.5, .75) -- (9.65, .755179) -- (9.8, .771539) -- (9.95, .800633) -- (10.1, .844602) -- (10.25, .906534) -- (10.4, .991154) -- (10.475, 1.04424) -- (10.55, 1.10638) -- (10.625, 1.17971) -- (10.7, 1.26752) -- (10.775, 1.37532) -- (10.85, 1.51375) -- (10.91, 1.6634) -- (10.94, 1.76299) -- (10.97, 1.8985) -- (10.982, 1.97482) -- (10.988, 2.02385) -- (10.994, 2.08877) -- (11, 2.25);
			
		\draw[-, black] (9.5, 3.75) -- (9.35, 3.74482) -- (9.2, 3.72846) -- (9.05, 3.69937) -- (8.9, 3.6554) -- (8.75, 3.59347) -- (8.6, 3.50885) -- (8.525, 3.45576) -- (8.45, 3.39362) -- (8.375, 3.32029) -- (8.3, 3.23248) -- (8.225, 3.12468) -- (8.15, 2.98625) -- (8.09, 2.8366) -- (8.06, 2.73701) -- (8.03, 2.6015) -- (8.018, 2.52518) -- (8.012, 2.47615) -- (8.006, 2.41123) -- (8, 2.25);

		\draw[-, black] (8, 2.25) -- (8.006, 2.08877) -- (8.012, 2.02385) -- (8.018, 1.97482) -- (8.03, 1.8985) -- (8.06, 1.76299) -- (8.09, 1.6634) -- (8.15, 1.51375) -- (8.225, 1.37532) -- (8.3, 1.26752) -- (8.375, 1.17971) -- (8.45, 1.10638) -- (8.525, 1.04424) -- (8.6, .991154) -- (8.75, .906534) -- (8.9, .844602) -- (9.05, .800633) -- (9.2, .771539) -- (9.35, .755179) -- (9.5, .75);

		\draw[-, black] (9.5, 3.75) -- (9.65, 3.74482) -- (9.8, 3.72846) -- (9.95, 3.69937) -- (10.1, 3.6554) -- (10.25, 3.59347) -- (10.4, 3.50885) -- (10.475, 3.45576) -- (10.55, 3.39362) -- (10.625, 3.32029) -- (10.7, 3.23248) -- (10.775, 3.12468) -- (10.85, 2.98625) -- (10.91, 2.8366) -- (10.94, 2.73701) -- (10.97, 2.6015) -- (10.982, 2.52518) -- (10.988, 2.47615) -- (10.994, 2.41123) -- (11, 2.25);

		\filldraw[fill = black] (9.5, .75) circle[radius = .05] node[below, scale = .7]{$\big( \frac{1}{2}, 0 \big)$};
		\filldraw[fill = black] (8, 2.25) circle[radius = .05] node[left, scale = .7]{$\big( 0, \frac{1}{2} \big)$};
		\filldraw[fill = black] (11, 2.25) circle[radius = .05] node[right, scale = .7]{$\big( 1, \frac{1}{2} \big)$};
		\filldraw[fill = black] (9.5, 3.75) circle[radius = .05] node[above, scale = .7]{$\big( \frac{1}{2}, 1 \big)$};

		\filldraw[fill = black] (10.55, 1.5) circle[radius = 0] node[scale = .7]{$\mathfrak{A}_{\text{SE}}$};
		\filldraw[fill = black] (8.55, 1.5) circle[radius = 0] node[scale = .7]{$\mathfrak{A}_{\text{SW}}$};
		\filldraw[fill = black] (10.55, 3) circle[radius = 0] node[scale = .7]{$\mathfrak{A}_{\text{NE}}$};
		\filldraw[fill = black] (8.55, 3) circle[radius = 0] node[scale = .7]{$\mathfrak{A}_{\text{NW}}$};

		\end{tikzpicture}
		
	\end{center}	
	
	\caption{\label{domainvertex} A sample of the six-vertex model with domain-wall boundary data on the square $\mathcal{S}_7$ is depicted to the left. The four arctic curves $\mathfrak{A}_{\text{SE}}$, $\mathfrak{A}_{\text{SW}}$, $\mathfrak{A}_{\text{NE}}$, and $\mathfrak{A}_{\text{NW}}$ are depicted to the right. } 
\end{figure}

The curve $\mathfrak{A}_{\text{SE}}$ will form the southeast part of the arctic boundary separating the frozen and liquid regions for a typical six-vertex ensemble on $\mathcal{S}$ with domain-wall boundary data. The remaining three parts of this boundary are obtained from reflecting $\mathfrak{A}$ into the lines $x = \frac{1}{2}$ and $y = \frac{1}{2}$, or equivalently by 
\begin{flalign*}
& \mathfrak{A}_{\text{SW}} = \big\{ (x, y) \in \mathbb{R}^2: (1 - x, y) \in \mathfrak{A}_{\text{SE}} \big\}; \qquad  \mathfrak{A}_{\text{NE}} = \big\{ (x, y) \in \mathbb{R}^2: (x, 1 - y) \in \mathfrak{A}_{\text{SE}} \big\}; \\
 &  \mathfrak{A}_{\text{NW}} = \big\{ (x, y) \in \mathbb{R}^2: (1 - x, 1 - y) \in \mathfrak{A}_{\text{SE}} \big\}.
\end{flalign*}	

\noindent We refer to the right side of \Cref{domainvertex} for depictions of these curves. 

The following result, which was originally predicted in equation (13) of \cite{TLSLASM} (and then again in equation (7.5) of \cite{TACDWSVM} and equation (2.17) of \cite{ACSVMGD}), provides the arctic boundary for the uniformly random six-vertex ensemble on $\mathcal{S}_N$ with domain-wall boundary data. 

\begin{cor}
	
	\label{boundarya0b0}
	
	Let $N \ge 1$ be a positive integer, and let $\delta \in (N^{-1 / 30}, 1)$ be a real number. There exists a constant $\gamma > 0$ (independent of $N$) such that the following holds. 
	
	Let $\mathcal{E}$ be a six-vertex ensemble on the $N \times N$ square domain $[1, N] \times [1, N]$ with domain-wall boundary data, chosen uniformly at random. Then, off of an event with probability at most $\gamma^{-1} \exp \big(- \gamma \delta^{24} N \big)$, the following statement holds. 
	
	Let $v \in [1, N] \times [1, N]$ be any lattice point such that $d \big( N^{-1} v, \mathfrak{A}_{\text{\emph{SE}}} \cup \mathfrak{A}_{\text{\emph{SW}}} \cup \mathfrak{A}_{\text{\emph{NE}}} \cup \mathfrak{A}_{\text{\emph{NW}}} \big) > \delta$. If $N^{-1} v \in \text{\emph{SE}} \big( \mathfrak{A}_{\text{\emph{SE}}} \big) \cup \text{\emph{SW}} \big( \mathfrak{A}_{\text{\emph{SW}}} \big) \cup\text{\emph{NE}} \big( \mathfrak{A}_{\text{\emph{NE}}} \big) \cup \text{\emph{NW}} \big( \mathfrak{A}_{\text{\emph{NW}}} \big)$, then $v$ is in the frozen region of $\mathcal{E}$. Otherwise, $v$ is in the liquid region of $\mathcal{E}$. 
	
\end{cor}

Observe that \Cref{boundarya0b0} does not quite follow from \Cref{boundaryb} as stated, since we assume that $a, b, c > 0$ in the latter result. In fact, although the southeast curve $\mathfrak{A}_{\text{SE}}$ from \Cref{boundarya0b0} coincides with the $a = 0 = b$ case of the southeast curve $\mathfrak{B}_{\text{SE}}$ from \Cref{boundaryb}, the remaining (southwest, northeast, and northwest) curves in these two results do not coincide.\footnote{This discrepancy arises from the fact that slightly different symmetries are required to transform the southeast part of the arctic boundary to its other parts in the cases of the three-bundle domain $\mathcal{T}$ and the square domain $\mathcal{S}$.} However, it can be quickly verified that the proof of \Cref{boundaryb} can be directly applied in the case $A = 0 = B$ to yield \Cref{boundarya0b0}. 

More specifically, \Cref{p1testimate} below shows that the southeast part of the boundary between the liquid and frozen regions (which, as mentioned in \Cref{pathboundary}, is the trajectory of the rightmost directed path) of a uniformly random domain-wall (defective) six-vertex ensemble on $\mathcal{T}_{A, B, C}$ concentrates with high probability around the curve $\mathfrak{B} = \mathfrak{B}_{\text{SE}}$, including in the case when $A = 0 = B$. This confirms that the southeast boundary of the arctic curve for a typical uniformly random domain-wall six-vertex ensemble on $\mathcal{S}$ is given by $\mathfrak{A}_{\text{SE}}$. That the remaining parts of the boundary are given by $\mathfrak{A}_{\text{SW}}$, $\mathfrak{A}_{\text{NE}}$, and $\mathfrak{A}_{\text{NW}}$ follow from rotating the square and using symmetries of the domain-wall six-vertex model; we omit further details.

\subsection*{Acknowledgments}

The author heartily thanks Filippo Colomo and Andrea Sportiello for stimulating conversations and enlightening explanations, as well as Alexei Borodin for valuable encouragement and helpful discussions and suggestions. The author would also like to thank two anonymous referees for their helpful comments on an earlier draft of this paper. The author is additionally grateful to the workshop, ``Conference on Quantum Integrable Systems, Conformal Field Theories and Stochastic Processes,'' held in 2016 at the Institut d'\'{E}tudes Scientifiques de Carg\'{e}se (funded by NSF grant DMS:1637087), where he was introduced to the tangent method. This work was partially supported by the NSF Graduate Research Fellowship under grant number DGE1144152.

\section{Miscellaneous Preliminaries}

\label{TangentDomain}

In this section we collect miscellaneous definitions and reductions that will facilitate the proof of \Cref{boundaryb}. Specifically, in \Cref{SixVertexDomain} we recall certain symmetries of defective six-vertex ensembles that allow us to reinterpret the model from \Cref{Model} as a non-defective one. Then, in \Cref{Property} we define several domains and describe a Gibbs property that will be useful for the proof of \Cref{boundaryb}.

\subsection{An Equivalent Model Without Defects}

\label{SixVertexDomain}

In order to apply certain monotonicity results to be stated in \Cref{MonotonicityPaths}, it will be useful to reformulate the defective six-vertex model described in \Cref{Model} as a non-defective one, in which all pairs of adjacent arrow configurations are consistent.

\begin{figure}[t]
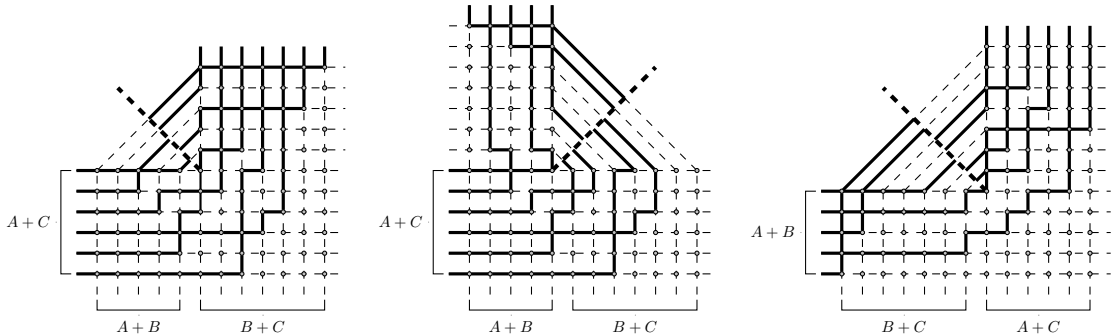

	
	\begin{center}
		

		
	\end{center}	
	
	\caption{\label{domainrotate} Depicted above is an example of how to ``rotate'' a defective six-vertex ensemble on $\mathcal{T}_{A, B, C}$ to form one on $\mathcal{T}_{B, C, A}$, as summarized in the beginning of \Cref{SixVertexDomain}. } 
\end{figure}

However, before explaining this in more detail, let us begin by recalling several symmetries of the model that allow us to ``rotate'' the three-bundle domain, thereby reducing the determination of the arctic boundary to that of its southeast part. More specifically, we fix integers $A, B, C \ge 1$, consider a defective six-vertex ensemble $\mathcal{E}$ on $\mathcal{T} = \mathcal{T}_{A, B, C}$ (as on the left side of \Cref{domainrotate}), and explain how to associate with it a defective six-vertex ensemble on $\mathcal{T}_{B, C, A}$. 

To that end, we first consider the $(B + C) \times (A + B)$ rectangle occupying the northeast corner of $\mathcal{T}_{A, B, C}$; rotate it 90 degrees counterclockwise around its southwest vertex $(A + B + 1, A + C)$; and ``reverse'' all arrow configurations in the rectangle, that is, replace each edge with (or without) an arrow with one without (or with, respectively) an arrow. This is shown in the middle of \Cref{domainrotate}; in particular, observe that the arrow reversal removes the previous defect line but creates a new one. Next, we rotate the ensemble 90 degrees counterclockwise and then reverse only the vertical arrows in the $(A + C) \times (A + 2B + C)$ rectangle in the east part of the rotated domain. This produces a defective six-vertex ensemble $\widetilde{\mathcal{E}}$ on $\mathcal{T}_{B, C, A}$. 

It is quickly verified that this sequence of transformations maps the southwest part of the arctic boundary of $\mathcal{E}$ into the southeast part of arctic boundary of $\widetilde{\mathcal{E}}$. Thus, upon scaling by $\frac{1}{N}$, we obtain the first relation in \eqref{xequations} expressing the southwest part of the arctic boundary in terms of the southeast part of associated with the cyclically shifted triple $(b, c, a)$. The other equations in \eqref{xequations} can be derived through similar rotations. Thus, in order to establish \Cref{boundaryb}, it suffices to only analyze the southeast part of the arctic boundary of $\mathcal{E}$, which (as mentioned in \Cref{pathboundary}) is given by the trajectory of its rightmost path.

Now let us describe the previously mentioned equivalence between defective six-vertex ensembles and non-defective path ensembles. We begin with the following definition.

\begin{figure}[t]
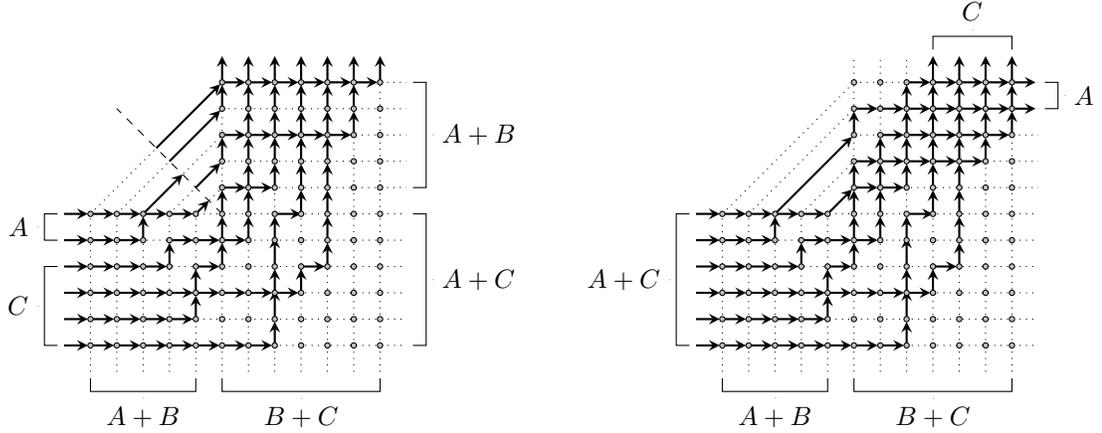

	
	\begin{center}
		

		
	\end{center}	
	
	\caption{\label{defectivenotdefectiveensemble} 	A (defective) six-vertex ensemble on the three-bundle domain $\mathcal{T}_{2, 3, 4}$ is depicted to the left. Shown to the right is the associated (non-defective) $2$-restricted directed path ensemble. } 
\end{figure}

\begin{definition} 
	
	\label{directedensemble} 
	
	A \emph{directed path ensemble} on $\mathcal{T}$ is an assignment of an arrow configuration to each vertex of $\mathcal{T}$ such that the arrow configurations of adjacent vertices are consistent. An example of such an ensemble is depicted on the right side of \Cref{defectivenotdefectiveensemble}; the fact that we do not impose an inconsistency condition implies that the ensemble does not exhibit defects. 
	
	A \emph{directed path} on $\mathcal{T}$ (or $\mathbb{Z}^2$) is a nondecreasing curve on $\mathbb{R}^2$ connecting a sequence of adjacent vertices in $\mathcal{T}$ (or $\mathbb{Z}^2$, respectively) by edges. Then, we can view a directed path ensemble on $\mathcal{T}$ as a family $\mathcal{P} = (\textbf{p}_1, \textbf{p}_2, \ldots , \textbf{p}_k)$ of directed paths on $\mathcal{T}$ that do not share edges (but might share vertices) and all enter and exit $\mathcal{T}$ through a boundary vertex of $\mathbb{Z}^2 \setminus \mathcal{T}$. We assume that these paths are ordered such that $\textbf{p}_i \subset \text{SE} (\textbf{p}_j)$ for any $1 \le i \le j \le k$; in particular, the rightmost (equivalently, the bottommost) path is $\textbf{p}_1$.

	For any integer $r \in [0, A + 2B + C]$, we refer to a directed path ensemble $\mathcal{P}$ on $\mathcal{T}$ as \emph{$r$-restricted} if the number of paths in $\mathcal{P}$ that contain a diagonal edge of $\mathcal{T}$ is equal to $r$. For instance, the ensemble on the right side of \Cref{defectivenotdefectiveensemble} is $A$-restricted (where $A = 2$).  
	
\end{definition}

\begin{rem} 
	
	In what follows, we will sometimes formally refer to $-\infty$ and $\infty$ as directed paths, which are defined to be maximally northwest and southeast, respectively.\footnote{One might view $-\infty$ and $\infty$ as consisting of the unique vertices $\big\{ (-\infty, -\infty) \big\}$ and $\big\{ (\infty, \infty) \big\}$, respectively.} So, for any directed path $\textbf{p}$, we have $\textbf{p} \in \text{SE} (-\infty) \cap \text{NW} (\infty)$. 
	
\end{rem}

As in \Cref{SixVertexDomain}, \emph{boundary data} for a directed path ensemble is prescribed by indicating which vertices of $\mathbb{Z}^2 \setminus \mathcal{T}$ are entrance and exit sites for paths. In particular, we define \emph{domain-wall boundary data} to be that in which paths enter through the $A + C$ vertices of the form $(0, m)$, with $m \in [1, A + C]$, and exit through the $C$ vertices of the form $(A + 2 B + m, 2A + B + C + 1)$, with $m \in [1, C]$, and the $A$ vertices of the form $(A + 2 B + C + 1, 2 A + B + C - m + 1)$, with $m \in [1, A]$. This is depicted on the right side of \Cref{defectivenotdefectiveensemble}. 

There is a direct correspondence between defective domain-wall six-vertex ensembles and $A$-restricted, domain-wall directed path ensembles on $\mathcal{T}$. To see this, let $\mathcal{E}$ denote an ensemble of the former type, as shown on the left side of \Cref{defectivenotdefectiveensemble}. As mentioned in \Cref{apaths}, $B$ paths of $\mathcal{E}$ must start from the defect line; this leaves $A$ sites along the defect line at which there are no exiting arrows, which are the places where a path of $\mathcal{E}$ ends at the defect line. 

In particular, if we temporarily reverse all arrows in $\mathcal{E}$, then the new topmost $A$ paths will enter from the defect line and will then exit $\mathcal{T}$ at sites of the form $(A + 2B + C + 1, 2A + B + C - m + 1)$ for $m \in [1, A]$. Thus, adding these $A$ ``dual'' paths to $\mathcal{E}$ and removing the original topmost $B$ paths from $\mathcal{E}$ yields a domain-wall directed path ensemble $\mathcal{P}$ on $\mathcal{T}$, which is $A$-restricted; see the right side of \Cref{defectivenotdefectiveensemble} for an example. This transformation can quickly be seen to produce a bijection between (defective) domain-wall six-vertex ensembles on $\mathcal{T}$ and (non-defective) $A$-restricted, domain-wall directed path ensembles on $\mathcal{T}$.

\begin{definition} 
	
	\label{definitionf}
	
	Fix $A, B, C \in \mathbb{Z}$, with $A, B \ge 0$ and $C \ge 1$. Let $\mathfrak{F} = \mathfrak{F}_{A, B, C}$ denote the set of (non-defective) $A$-restricted, domain-wall directed path ensembles on $\mathcal{T} = \mathcal{T}_{A, B, C}$. 
\end{definition} 

In view of this equivalence and the previously mentioned symmetries, in order to show \Cref{boundaryb} it suffices to establish the following theorem for the behavior of the rightmost path of a uniformly random ensemble from $\mathfrak{F}$.

\begin{thm}
	
	\label{p1testimate}
	
	Fix real numbers $a, b \in [0, 1)$ and $c \in (0, 1]$ such that $a + b + c = 1$; let $N \ge 1$ be an integer; set $A = \lfloor a N \rfloor$, $B = \lfloor b N \rfloor$, and $C = \lfloor c N \rfloor$; and let $\delta \in (N^{-1 / 30}, 1)$ be a real number. There exists a constant $\gamma = \gamma (a, b, c) > 0$ such that the following holds. 
	
	Let $\mathcal{P}$ denote a uniformly random element of $\mathfrak{F} = \mathfrak{F}_{A, B, C}$, and denote its rightmost path by $\textbf{\emph{p}}_1$. Then, off of an event of probability at most $\gamma^{-1} \exp \big( - \gamma \delta^{24} N \big)$, the following statement holds. 
	
	Let $v \in \mathcal{T}$ be any vertex such that $d \big(N^{-1} v, \mathfrak{B} \big) > \delta$. If $v \in \text{\emph{SE}} (\mathfrak{B})$ then $v \in \text{\emph{SE}} (\textbf{\emph{p}}_1)$, and if $v \in \text{\emph{NW}} (\mathfrak{B})$ then $v \in \text{\emph{NW}} (\textbf{\emph{p}}_1)$.
	
\end{thm}

Thus, for the remainder of this article, we will only consider non-defective directed path ensembles on $\mathcal{T}$ and make no further reference to defective six-vertex ensembles.

\subsection{A Gibbs Property and Augmented Domains} 

\label{Property}

In this section we state a Gibbs property satisfied by the uniform measure on directed path ensembles on $\mathcal{T} = \mathcal{T}_{A, B, C}$ and also define augmented three-bundle domains. We begin with the former.

Although this Gibbs property will hold in greater generality, we only state it for \emph{rectangular subdomains} of $\mathcal{T}$, which are defined to be directed subgraphs $\Lambda \subseteq \mathcal{T}$ induced by intersecting a rectangle in $\mathbb{Z}^2$ with $\mathcal{T}$; equivalently, these are subsets of the form $\big( [m, n] \times [s, t] \big) \cap \mathcal{T}$ for integers $1 \le m \le n \le A + 2 B + C$ and $1 \le s \le t \le 2 A + B + C$. Observe in particular that $\mathcal{T}$ is a rectangular subdomain $\mathcal{T}$, obtained when the rectangle $[m, n] \times [s, t] = [1, A + 2 B + C] \times [1, 2A + B + C]$.

We can define ($r$-restricted) directed path ensembles $\mathcal{P} = (\textbf{p}_1, \textbf{p}_2, \ldots, \textbf{p}_k)$ on any rectangular subdomain $\Lambda \subseteq \mathcal{T}$ analogously to as we did in \Cref{directedensemble} (which addresses the case $\Lambda = \mathcal{T}$). In particular, any directed path ensemble $\mathcal{P}$ on $\mathcal{T}$ restricts to one $\mathcal{P} |_{\Lambda}$ on $\Lambda$. 

Next, let us introduce some notation for boundary data on (restricted) directed path ensembles on rectangular subdomains $\Lambda \subseteq \mathcal{T}$. 

\begin{definition}
	
\label{boundarylambda} We first define the \emph{west} and \emph{east boundaries} of $\Lambda = \big( [m, n] \times [s, t] \big) \cap \mathcal{T}$ to be the sets of lattice points in $\mathbb{Z}^2 \setminus \Lambda$ adjacent to $\Lambda$ that are of the form $(m - 1, y)$ and $(n + 1, y)$, respectively, for some $y \in [s, t]$. Similarly, its \emph{south} and \emph{north boundaries} are defined to be the sets of lattice points in $\mathbb{Z}^2 \setminus \Lambda$ adjacent to $\Lambda$ of the form $(x, s - 1)$ and $(x, t + 1)$, respectively, for some $x \in [m, n]$. 

We say that an ordered pair $(\textbf{u}, \textbf{w})$ of two $k$-tuples of vertices $\textbf{u} = (u_1, u_2, \ldots , u_k)$ and $\textbf{w} = (w_1, w_2, \ldots , w_k)$ constitutes \emph{admissible boundary data} (on $\Lambda$) if the following three conditions hold. First, each $u_i$ is a vertex along either the west or south boundary of $\mathbb{Z}^2 \setminus \Lambda$. Second, each $w_i$ is a vertex along either the east or north boundary of $\mathbb{Z}^2 \setminus \Lambda$. Third, $u_i \in \text{SE} (u_j)$ and $w_i \in \text{SE} (w_j)$ for any integers $1 \le i \le j \le k$. 

\emph{Boundary data} for a (possibly restricted) directed path ensemble $\mathcal{P} = (\textbf{p}_1, \textbf{p}_2, \ldots , \textbf{p}_k)$ on $\Lambda$ is given by a pair $(\textbf{u}, \textbf{w})$ of admissible boundary data. Here, $\textbf{u} = (u_1, u_2, \ldots , u_m)$ dictates the entrance sites for the $k$ paths in the ensemble, and $\textbf{w} = (w_1, w_2, \ldots , w_k)$ dictates the $k$ exit sites (meaning that the path $\textbf{p}_i$ enters and exits $\Lambda$ through $u_i$ and $w_i$, respectively, for each $i \in [1, k]$). 
	
\end{definition}

\begin{example}
	
\label{uw1} 

Domain-wall boundary data on $\mathcal{T}$ corresponds to when $\textbf{u} =  (u_1, u_2, \ldots,  u_{A + C})$ and $\textbf{w} = (w_1, w_2, \ldots , w_{A + C})$ are defined by 
\begin{flalign*}
u_i = (0, i) \quad \text{if $i \in [1, A + C]$}; \quad & w_i = (A + 2B + C + 1, A + B + C + i) \quad \text{if $i \in [1, A]$}; \\
& w_i = (2A + 2B + C - i + 1, 2A + B + C + 1) \quad \text{if $i \in [A + 1, A + C]$}.
\end{flalign*}

\end{example} 

\begin{example}
	
	\label{uw2}

	Let $k \in [1, A + 2B + C]$ be an integer. \emph{Singly-refined domain-wall boundary data at $(k, 0)$ (on $\mathcal{T}$)} is defined by the pair $(\textbf{u}, \textbf{w})$, where $\textbf{u} =  (u_1, u_2, \ldots,  u_{A + C + 1})$ and $\textbf{w} = (w_1, w_2, \ldots , w_{A + C + 1})$ are defined by 
	\begin{flalign*}
	& u_1 = (k, 0); \qquad u_i = (0, i - 1) \quad \text{if $i \in [2, A + C + 1]$}; \\
	& w_i = (A + 2B + C + 1, A + B + C + i - 1) \quad \text{if $i \in [1, A + 1]$}; \\
	& w_i = (2A + 2B + C - i + 2, 2A + B + C + 1) \quad \text{if $i \in [A + 2, A + C + 1]$}.
	\end{flalign*}
	
	This is obtained from the domain-wall boundary data from \Cref{uw1}, but with an additional path that enters $\mathcal{T}$ at $(k, 0)$ and exits at $(A + 2B + C + 1, A + B + C)$; in particular, there are $A + C + 1$ paths in the ensemble. We refer to the left side of \Cref{figurevertexwedge2} for a depiction.
	
\end{example}

\begin{figure}[t]
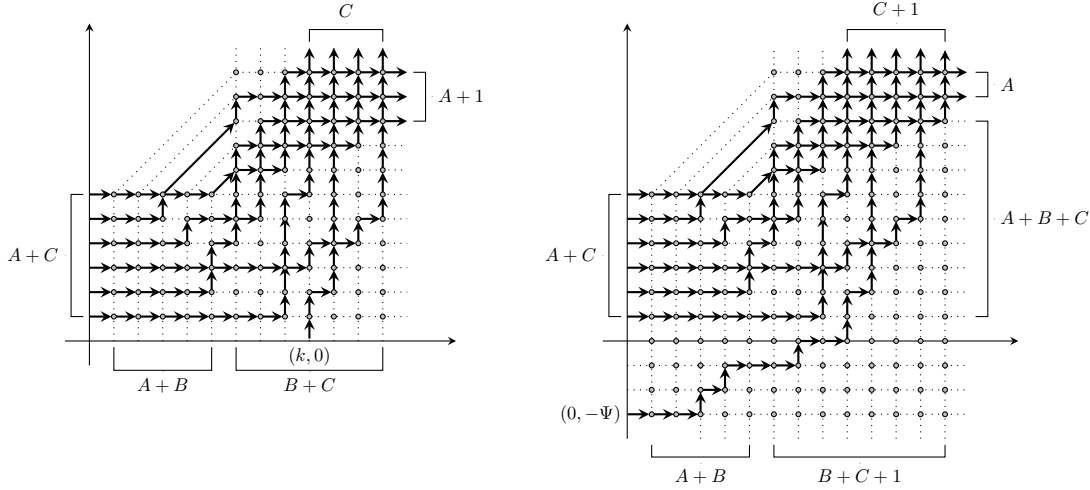

	
	\begin{center}
		

		
	\end{center}	
	
	\caption{\label{figurevertexwedge2} Depicted to the left is a directed path ensemble on $\mathcal{T} = \mathcal{T}_{2, 3, 4}$ with singly refined domain-wall boundary data at $(k, 0)$ when $k = 9$. Depicted to the right is a directed path ensemble on the augmented domain $\mathcal{X}_3$ with augmented domain-wall boundary data. } 
\end{figure}

We now define sets of (restricted) directed path ensembles on rectangular subdomains of $\mathcal{T}$ with given boundary data that are bounded to the left by some path $\textbf{f}$ and to the right by some path $\textbf{g}$. 

\begin{definition}
	
	\label{probabilitypaths}
	
	Suppose $\Lambda \subseteq \mathcal{T}$ is a rectangular subdomain; $r \ge 0$ and $k \ge 1$ are positive integers; and $(\textbf{u}, \textbf{w})$ is a pair of $k$-tuples that defines admissible boundary data on $\Lambda$. Let $\mathfrak{E}_{-\infty, \infty}^{\textbf{u}, \textbf{w}; r} = \mathfrak{E}_{-\infty, \infty; \Lambda}^{\textbf{u}, \textbf{w}; r}$ denote the set of $r$-restricted directed path ensembles $\mathcal{P} = \big( \textbf{p}_1, \textbf{p}_2, \ldots , \textbf{p}_m \big)$ on $\Lambda$, with boundary data given by $(\textbf{u}, \textbf{w})$. 
	
	Now let $\textbf{f}$ and $\textbf{g}$ be paths on $\Lambda$ such that $\textbf{g} \subset \text{SE} (\textbf{f})$. Suppose that $\textbf{f}, \textbf{g}$ enter $\Lambda$ through vertices $u_0, u_{k + 1} \in \mathbb{Z}^2 \setminus \Lambda$, respectively, and exit $\Lambda$ through $w_0, w_{k + 1} \in \mathbb{Z}^2 \setminus \Lambda$, respectively. Then let $\mathfrak{E}_{\textbf{f}, \textbf{g}}^{\textbf{u}, \textbf{w}; r} = \mathfrak{E}_{\textbf{f}, \textbf{g}; \Lambda}^{\textbf{u}, \textbf{w}; r}$ denote the set of $r$-restricted directed path ensembles $\mathcal{P} = \big( \textbf{p}_1, \textbf{p}_2, \ldots , \textbf{p}_m \big)$ on $\Lambda$, with boundary data given by $(\textbf{u}, \textbf{w})$, and such that $\textbf{p}_i \in \text{SE} (\textbf{f}) \cap \text{NW} (\textbf{g})$ for each $i \in [1, m]$. 
	
	Similarly let $\mathfrak{E}_{-\infty, \textbf{g}}^{\textbf{u}, \textbf{w}; r} = \mathfrak{E}_{-\infty, \textbf{g}; \Lambda}^{\textbf{u}, \textbf{w}; r}$ denote the set of such $\mathcal{P}$ with $\textbf{p}_i \in \text{NW} (\textbf{g})$ for each $i$, and let $\mathfrak{E}_{\textbf{f}, \infty}^{\textbf{u}, \textbf{w}; r} = \mathfrak{E}_{\textbf{f}, \infty; \Lambda}^{\textbf{u}, \textbf{w}; r}$ denote the set of such ensembles with  $\textbf{p}_i \in \text{SE} (\textbf{f})$ for each $i$. 
	
\end{definition}

In particular, if $\textbf{u}$ and $\textbf{w}$ are given by \Cref{uw1}, then $\mathfrak{E}_{-\infty, \infty; \Lambda}^{\textbf{u}, \textbf{w}; A}$ is $\mathfrak{F}$ from \Cref{definitionf}. 

The uniform measure on $\mathfrak{E}_{\textbf{f}, \textbf{g}; \Lambda}^{\textbf{u}, \textbf{w}; r}$ satisfies a \emph{Gibbs property}, which essentially states the following. If $\mathcal{P}$ is a uniformly random ensemble from $\mathfrak{E}_{\textbf{f}, \textbf{g}; \Lambda}^{\textbf{u}, \textbf{w}; r}$ and $\Lambda' \subseteq \Lambda$ is a rectangular subdomain then, conditional on the part of $\mathcal{P}$ outside $\Lambda'$, the law of $\mathcal{P}$ inside $\Lambda'$ is uniform on the set of directed path ensembles on $\Lambda'$ with specified boundary conditions. This is stated more precisely through the following lemma, whose proof is a direct consequence of \Cref{probabilitypaths} (and is therefore omitted).

\begin{lem} 

\label{property} 

Adopt the notation of \Cref{probabilitypaths}, and let $\mathcal{P} = ( \textbf{\emph{p}}_1, \textbf{\emph{p}}_2, \ldots , \textbf{\emph{p}}_k ) \in \mathfrak{E}_{\textbf{\emph{f}}, \textbf{\emph{g}}; \Lambda}^{\textbf{\emph{u}}, \textbf{\emph{w}}; r}$ be sampled uniformly at random. Let $\Lambda' \subseteq \Lambda$ denote a rectangular subdomain; $0 \le i \le j \le k$ and $0 \le r' \le r$ be integers; and $(\textbf{\emph{u}}', \textbf{\emph{w}}')$ be a pair of $(j - i + 1)$-tuples constituting admissible boundary data on $\Lambda'$, with $\textbf{\emph{u}}' = (u_1', u_2', \ldots , u_{j - i + 1}')$ and $\textbf{\emph{w}}' = (w_1', w_2', \ldots , w_{j - i + 1}')$. Now condition on the paths $\textbf{\emph{p}}_m \in \mathcal{P}$ for $m \notin [i, j]$; the restrictions $\textbf{\emph{p}}_m |_{\Lambda \setminus \Lambda'}$ for $m \in [i, j]$; the event that $\textbf{\emph{p}}_m$ enters $\Lambda'$ at $u_{m - i + 1}'$ and exits $\Lambda'$ at $w_{m - i + 1}'$ for each $m \in [i, j]$; and the event that the number of indices $m \notin [i, j]$ for which $\textbf{\emph{p}}_m$ contains a diagonal edge of $\Lambda$ is equal to $r - r'$.  

Then, the joint law of $\big( \textbf{\emph{p}}_i |_{\Lambda'}, \textbf{\emph{p}}_{i + 1} |_{\Lambda'}, \ldots , \textbf{\emph{p}}_j |_{\Lambda'} \big)$ is given by the uniform measure on $\mathfrak{E}_{\textbf{\emph{p}}_{i - 1}, \textbf{\emph{p}}_{j + 1}; \Lambda'}^{\textbf{\emph{u}}', \textbf{\emph{w}}'; r'}$, where we set $\textbf{\emph{p}}_{i - 1} = - \infty$ if $i = 1$ and $\textbf{\emph{p}}_{j + 1} = \infty$ if $j = k$.
\end{lem}

We conclude this section by defining an augmented version of the three-bundle domain $\mathcal{T}$ and its boundary data, which will be useful for implementing the tangent method later in the paper.

\begin{definition}
	
\label{xpsi} Let $\Psi \ge 1$ be an integer; the \emph{augmented three-bundle domain} $\mathcal{X} = \mathcal{X}_{\Psi} = \mathcal{X}_{A, B, C; \Psi}$ is a directed graph whose vertex set is given by 
\begin{flalign*}
\mathcal{T}_{A, B, C} \cup \big( [1, A + 2 B + C + 1] \times [-\Psi, 0] \big) \cup \big( \{ A + 2B + C + 1\} \times [1, 2A + B + C] \big), 
\end{flalign*}

\noindent and whose edge set consists of the (directed) edges of $\mathcal{T}_{A, B, C}$, along with any directed edges in $\mathbb{Z}^2$ connecting vertices on $\mathcal{X}$. Stated alternatively, $\mathcal{X}$ is the domain obtained by adding a $(A + 2B + C + 1) \times \Psi$ rectangle (and all associated edges) below the graph $\mathcal{T}_{A, B, C + 1}$, and then by translating the result down by $1$. We refer to \Cref{figurevertexwedge2} for an example. If $\Psi = 0$, we set $\mathcal{X} = \mathcal{X}_0 = \mathcal{T}_{A, B, C}$. 

\end{definition}

\begin{definition} 
	
\label{domainaugmentedboundary}

\emph{Augmented domain-wall boundary data} (on $\mathcal{X} = \mathcal{X}_{\Psi}$) is defined by the pair $(\textbf{u}, \textbf{w})$, where $\textbf{u} =  (u_1, u_2, \ldots,  u_{A + C + 1})$ and $\textbf{w} = (w_1, w_2, \ldots , w_{A + C + 1})$ are defined by 
\begin{flalign*}
& u_1 = (0, -\Psi); \qquad u_i = (0, i - 1) \quad \text{if $i \in [1, A + C + 1]$}; \\
& w_i = (A + 2B + C + 2, A + B + C + i) \quad \text{if $i \in [1, A]$}; \\
& w_i = (2A + 2B + C - i + 2, 2A + B + C + 1) \quad \text{if $i \in [A + 1, A + C + 1]$}.
\end{flalign*}

This essentially coincides with domain-wall boundary data on $\mathcal{T}_{A, B, C + 1}$ (translated down by $1$), but where the entrance site of rightmost path $\textbf{p}_1$ is shifted down by $\Psi$; see the right side of \Cref{figurevertexwedge2}. 

Let $\mathfrak{G}_{\Psi}^A$ denote the set of all $A$-restricted directed path ensembles on $\mathcal{X}_{A, B, C; \Psi}$ with augmented domain-wall boundary data. In particular, setting $\Psi = 0$, $\mathfrak{G}_0^A = \mathfrak{F}_{A, B, C + 1}$ denotes the set of all $A$-restricted directed path ensembles on $\mathcal{T}_{A, B, C + 1}$ with domain-wall boundary data.

\end{definition}

\section{The Tangent Method Heuristic}

\label{Tangent}

The proof of \Cref{p1testimate} proceeds by justifying the \emph{tangent method}, which is a heuristic proposed by Colomo-Sportiello in \cite{ACSVMGD} as a way of predicting the arctic boundaries of certain exactly solvable statistical mechanical models. Thus, we begin by recalling this heuristic in \Cref{OutlineTangent}, and then we outline the proof of \Cref{p1testimate} in \Cref{Outline}.  

Throughout this section we fix real numbers $a, b, c \in [0, 1]$ such that $a + b + c = 1$; let $N$ be a positive integer; and denote $A = \lfloor a N \rfloor$, $B = \lfloor b N \rfloor$, and $C = \lfloor c N \rfloor$.

\subsection{The Heuristic} 

\label{OutlineTangent}

The goal of the tangent method is to provide a heuristic that explicitly evaluates (part of) the arctic boundary of a vertex model, assuming that one has access to certain quantities, called refined correlation functions, associated with it. In this section we outline this heuristic in the special case of the six-vertex model on the three-bundle domain, following Section 3 of \cite{ACSVMGD}. More specifically, we will explain how to derive the southeast part of the arctic boundary of a typical $A$-restricted directed path ensemble $\mathcal{P}$ on $\mathcal{T} = \mathcal{T}_{A, B, C}$, which (as explained in \Cref{pathboundary}) is the trajectory of the rightmost directed path of $\mathcal{P}$. 

To that end, first observe that for any such ensemble $\mathcal{P}$ there is a unique integer $K = K (\mathcal{P}) \in [1, A + 2B + C]$ such that the vertex $(K, 1)$ is assigned arrow configuration $(0, 1; 1, 0)$. Stated alternatively, this is the integer such that the rightmost (equivalently, the bottommost) directed path in $\mathcal{P}$ exits the line $y = 1$ at $(K, 1)$. 

For any integer $k \in [1, A + 2B + C]$, define the \emph{(singly) $k$-refined correlation function} $H (k) = H_{A, B, C} (k) = \mathbb{P} \big[ K (\mathcal{P}) = k \big]$, where the probability measure is uniform over all domain-wall $A$-restricted directed path ensembles $\mathcal{P}$ on $\mathcal{T}$. In what follows, we assume it is possible to asymptotically evaluate $H(k)$ when $k = x N$ scales linearly with $N$, namely, that
\begin{flalign*}
H ( xN ) \sim \exp \Big( - \big( \mathfrak{h} (x) + o(1) \big) N  \Big),
\end{flalign*}

\noindent for some explicit, nonnegative function $\mathfrak{h} (x) = \mathfrak{h}_{a, b, c} (x)$. This is indeed the case in our setting, since equation (59) of \cite{PRC} provides an identity for $H(k)$ as a sum of binomial coefficients; see \Cref{hkabcidentity} below. Observe that $\mathfrak{h}$ immediately provides some information about the arctic boundary, since the value $\theta$ at which $\mathfrak{h} (\theta) = 0$ will be where we expect $K \approx \theta N$ with high probability; thus, $(\theta, 0)$ is the location where the arctic boundary should be tangent to the $x$-axis. 

To obtain more precise information about the arctic boundary, we consider directed path ensembles on augmented domains. More specifically, let $\psi \ge 0$ be a real number, and denote $\Psi = \psi N$ (which we assume to be an integer). Consider a uniformly random $A$-restricted directed path ensemble $\mathcal{P}^{\text{aug}}$ on the augmented domain $\mathcal{X}_{\Psi}$ (from \Cref{xpsi}) with augmented domain-wall boundary data (from \Cref{domainaugmentedboundary}). See the left side of \Cref{ensemblex} for a depiction. 

The premise underlying the tangent method is that, with high probability, the trajectory of the rightmost path $\textbf{p}_1^{\aug}$ of $\mathcal{P}^{\text{aug}}$ will initially approximately follow a line that is tangent to the southeast part of the arctic boundary $\mathcal{P}$; see the right side of \Cref{ensemblex}. Then, upon meeting this boundary, the rightmost path of $\mathcal{P}^{\text{aug}}$ will merge with it.

\begin{figure}[t]
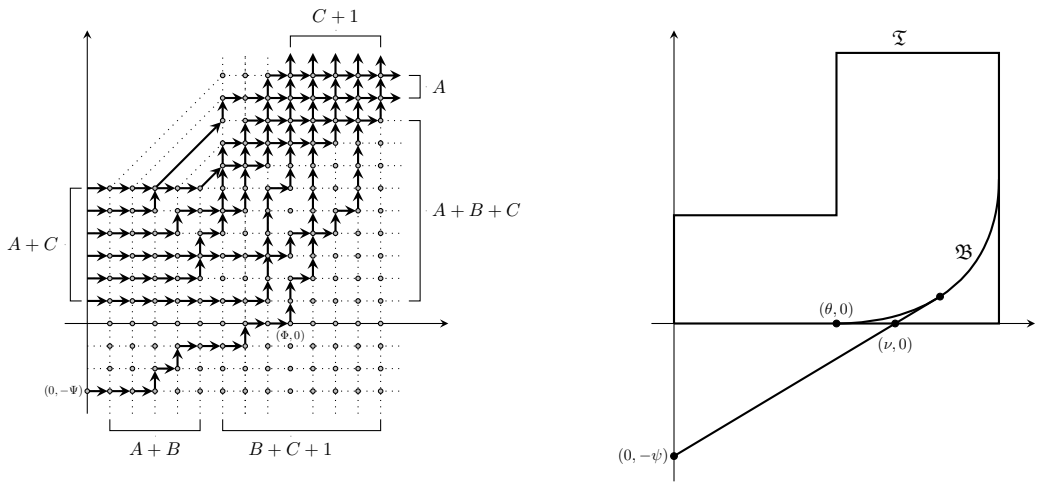

	
	\begin{center}
		

		
	\end{center}	
	
	\caption{\label{ensemblex} 	Shown to the left is a directed path ensemble on the augmented three-bundle domain $\mathcal{X} = \mathcal{X}_3 = \mathcal{X}_{2, 3, 4; 3}$; here, $\Phi = 9$. Shown to the right is the heuristic macroscopic limit of this ensemble, where the new (added) path is tangent to the arctic boundary $\mathfrak{B}$.  } 
\end{figure}

Assuming such a statement to be true, one might attempt to determine $\mathfrak{B}$ as follows. Suppose that one is able to approximately understand the location $(\Phi, 0) = \big( \Phi (\mathcal{P}^{\text{aug}}), 0 \big) \approx (\nu N, 0)$ (for some $\nu = \nu (\psi) \in \mathbb{R}$) where the path $\textbf{p}_1^{\aug}$ leaves the $x$-axis. Then, in view of the previously mentioned assumption, one would conclude that the line passing through $(0, -\psi)$ and $(\nu, 0)$ should be tangent to the limiting arctic boundary $\mathfrak{B}$; see \Cref{ensemblex}. By then varying the parameter $\psi \in \mathbb{R}_{\ge 0}$, one obtains a complete family of lines tangent to $\mathfrak{B}$; the convex envelope of these lines would then determine $\mathfrak{B}$. 

To evaluate $\nu = \nu (\psi)$, we estimate the probability that $\Phi = x N$ for some $x \in [0, a + 2b + c]$. To that end, observe that a domain-wall $A$-restricted directed path ensemble $\mathcal{P}^{\text{aug}}$ on $\mathcal{X}_{\Psi}$ such that $\Phi (\mathcal{P}^{\text{aug}}) = xN$ is the union of a directed up-right path from $(0, -\Psi)$ to $(\Phi, 0)$ and an $A$-restricted directed path ensemble on $\mathcal{T}_{A, B, C + 1}$ with singly refined domain-wall boundary data at $(xN, 0)$, as in \Cref{uw2} (after being shifted down by $1$). 
	
Since the number of such directed up-right paths is equal to $\binom{\Psi + \Phi - 1}{\Psi}$ and the number of such directed path ensembles is proportional to the refined correlation function $H_{A, B, C + 1} (xN)$, it follows that $\mathbb{P} \big[ \Phi (\mathcal{P}^{\text{aug}}) = x N \big]$ is proportional to 
\begin{flalign*}
H_{A, B, C + 1} (x N) \binom{\psi N + x N - 1}{\psi N} \approx \exp \Big( \big( g_{\psi} (x) + o(1) \big) N \Big),
\end{flalign*}

\noindent where we have denoted 
\begin{flalign*}
g_{\psi} (x) = (\psi + x) \log (\psi + x) - \psi \log \psi - x \log x - \mathfrak{h} (x),
\end{flalign*}

\noindent which is explicit if $\mathfrak{h}$ is. In particular, $g_{\psi} (x)$ will admit a unique maximum at some $\nu = \nu_{\psi} \in [0, a + 2b + c]$, which can be evaluated explicitly. Thus, with high probability, we have that $\Phi (\mathcal{P}^{\text{aug}}) \approx \nu N$. 

Under the tangency assumption, it would follow that the (explicit) line through $(0, -\psi)$ and $(\nu, 0)$ is tangent to $\mathfrak{B}$ for any $\psi \in \mathbb{R}_{\ge 0}$. As mentioned above, varying $\psi$ then yields a complete family of tangent lines to the curve $\mathfrak{B}$, which one can use to parameterize it as in \eqref{zeta1z}, \eqref{x1y1z}, and \eqref{bcurvesequations}.

\subsection{Outline of the Proof of \Cref{p1testimate}}

\label{Outline}

In this section we briefly outline the proof of \Cref{p1testimate}, which constitutes the remainder of this paper. Denote $\mathcal{P} = (\textbf{p}_1, \textbf{p}_2, \ldots, \textbf{p}_{A + C})$ and $\mathcal{P}^{\aug} = \big(\textbf{p}_1^{\aug}, \textbf{p}_2^{\aug} \ldots , \textbf{p}_{A + C + 1}^{\aug} \big)$.

\begin{enumerate} 
	
	\item \label{p2approximate} \emph{Tangency}: Letting $\ell$ denote the tangent line to $\textbf{p}_2^{\text{aug}}$ through $(-\Psi, 0)$, we show that the vertex $(-\Phi, 0)$ where $\textbf{p}_1^{\aug}$ enters the $x$-axis is close to $\ell$; see \Cref{p1p2tangent} below.

	\item \label{estimatep2} \emph{Concentration Estimate}: We use the heuristic explained in \Cref{OutlineTangent} to establish a concentration estimate for $\Phi = \Phi \big( \mathcal{P}^{\aug} \big)$; see \Cref{lzpsip1} below.
	
	\item \emph{Comparing $\textbf{\emph{p}}_1$ to $\textbf{\emph{p}}_2^{\aug}$}: We show that $\textbf{p}_1$ and $\textbf{p}_2^{\aug}$ approximately follow the same trajectory with high probability. To do this, we establish the following two statements. 
	\begin{enumerate}
		\item \label{p1p2augmented} There exist two couplings of the laws of $\mathcal{P}$ and $\mathcal{P}^{\aug}$ such that $\textbf{p}_1 \in \SE (\textbf{p}_2^{\aug})$ under the first and $\textbf{p}_2^{\aug} \in \SE (\textbf{p}_2)$ under the second; see \Cref{ensemblesnn1} below. 
		\item \label{nearp1p2} The paths $\textbf{p}_1, \textbf{p}_2 \in \mathcal{P}$ remain close to each other with high probability; see \Cref{p1p2near} below.  
	\end{enumerate}
	Together, these will essentially imply $\textbf{p}_1 \approx \textbf{p}_2^{\aug} \approx \textbf{p}_2$. 

\end{enumerate}

Given these statements, the proof of \Cref{p1testimate} will appear in \Cref{ProofBoundary}. 

Fundamental to implementing the above outline will be a monotonicity result, given by \Cref{couplemonotone} in \Cref{MonotonicityPaths} below, that essentially states the following. If the boundary data for two ice models are ordered, then these two models can be coupled in such a way that their paths are also ordered in the interior of the domain. The proof of this statement is based on monotonicity properties for the Glauber dynamics on the set of restricted directed path ensembles; similar dynamical ideas to establish monotonicity statements in related settings were used in \cite{LE,ELE}. The existence of the two couplings explained in part \ref{p1p2augmented} above will follow as a quick consequence of \Cref{couplemonotone}.

Next, in \Cref{PathLinear}, we derive a series of miscellaneous linearity estimates for random walks (possibly with boundary conditions) that follow from direct combinatorial considerations together with the monotonicity result \Cref{couplemonotone}. Using these linearity estimates; \Cref{couplemonotone}; and the Gibbs property given by \Cref{property}, we will establish \Cref{p1p2tangent} from part \ref{p2approximate} above. Although this statement is slightly weaker than the claim that $\textbf{p}_1^{\aug}$ is approximately tangent to $\textbf{p}_2^{\aug}$, it will be sufficient for our purposes. 

Then, in \Cref{PathsNearp1p2} we	 implement part \ref{nearp1p2} above to show that $\textbf{p}_1$ and $\textbf{p}_2$ are close to each other with high probability. To do this, we will first in \Cref{PathConvex} use \Cref{couplemonotone}, the linearity estimates from \Cref{PathLinear}, and \Cref{property} to show as \Cref{convexp1x} and \Cref{convexp2x} that the paths $\textbf{p}_1$ and $\textbf{p}_2$ are ``approximately convex.'' Then, we will establish \Cref{p1p2near} in \Cref{PathsNearp1p2}. 

In \Cref{Equationp1} we establish the concentration estimate for $\Phi$ given by \Cref{lzpsip1} of part \ref{estimatep2} above. This will closely follow Section 3.4 of \cite{ACSVMGD} (also explained in \Cref{OutlineTangent}), using exact results from \cite{PRC,ACSVMGD} on the refined partition function for the ice vertex model on the three-bundle domain. We then conclude the proof of \Cref{p1testimate} in \Cref{ProofBoundary}.

\section{Path Monotonicity} 

\label{MonotonicityPaths} 

In this section we derive several monotonicity properties for paths in a uniformly random directed path ensemble with general boundary data. We begin by stating the result (given by \Cref{couplemonotone}) and some of its consequences in \Cref{Monotonicity}; we then provide its proof in \Cref{ProofMonotone}. Throughout this section, we fix nonnegative integers $A, B \ge 0$ and $C \ge 1$.

\subsection{Monotonicity Results and Consequences} 

\label{Monotonicity}

In what follows (and for the remainder of this paper), it will be convenient to introduce a partial ordering on non-decreasing curves in $\mathbb{R}^2$. Specifically, if $\textbf{f}$ and $\textbf{g}$ are non-decreasing curves, we write $\textbf{f} \le \textbf{g}$ (or equivalently $\textbf{g} \ge \textbf{f}$) if, for every two points $(x, y) \in \textbf{f}$ and $(x', y) \in \textbf{g}$ on the same horizontal line, we have that $x \le x'$. Furthermore, for any two points $v = (x, y) \in \mathbb{R}^2$ and $v' = (x', y') \in \mathbb{R}^2$, we say that $v \le v'$ (or $v' \ge v$) if $v' \in \SE(v)$; this does not coincide with the lexicographic ordering on $\mathbb{R}^2$ but will be useful for us. 

Moreover, if $\textbf{v} = (v_1, v_2, \ldots , v_m)$ and $\textbf{v}' = (v_1', v_2', \ldots , v_m')$ are sequences of lattice points then we say that $\textbf{v} \ge \textbf{v}'$ if $v_i \ge v_i'$ for each $i \in [1, m]$. Additionally, for any directed path ensembles $\mathcal{P} = (\textbf{p}_1, \textbf{p}_2, \ldots , \textbf{p}_m)$ and $\mathcal{P}' = (\textbf{p}_1', \textbf{p}_2', \ldots , \textbf{p}_m')$ with the same number of paths (on some rectangular subdomain of $\mathcal{T} = \mathcal{T}_{A, B, C}$), we say that $\mathcal{P} \le \mathcal{P}'$ if $\textbf{p}_i \le \textbf{p}_i'$ for each $i \in [1, m]$. 

The following proposition, which is similar to Lemma 2.6 and Lemma 2.7 of \cite{LE} and Lemma 2.6 and Lemma 2.7 of \cite{ELE}, provides a monotone coupling between uniform measures on directed path ensembles on rectangular subdomains of $\mathcal{T}$ with different boundary data.

\begin{prop}
	
	\label{couplemonotone}
	
	Let $\Lambda \subseteq \mathcal{T}_{A, B, C}$ be a rectangular domain, and let $r, r' \ge 0$ and $m \ge 1$ be integers. Suppose that $\textbf{\emph{f}}$, $\textbf{\emph{g}}$, $\textbf{\emph{f}}'$, and $\textbf{\emph{g}}'$ are directed paths on $\Lambda$ (possibly equal to $-\infty$ or $\infty$) such that $\textbf{\emph{f}} \le \textbf{\emph{g}}$ and $\textbf{\emph{f}}' \le \textbf{\emph{g}}'$. Let $\textbf{\emph{f}}, \textbf{\emph{g}}, \textbf{\emph{f}}', \textbf{\emph{g}}'$ enter and exit $\Lambda$ at vertices $u_0, u_{m + 1}, u_0', u_{m + 1}'$ and $w_0, w_{m + 1}, w_0', w_{m + 1}'$, respectively. Further let $(\textbf{\emph{u}}, \textbf{\emph{w}})$ and $( \textbf{\emph{u}}', \textbf{\emph{w}}')$ denote two pairs of $m$-tuples of vertices on $\Lambda$ that each constitutes an admissible boundary condition on $\Lambda$. Assume that $\mathfrak{E}_{\textbf{\emph{f}}, \textbf{\emph{g}}}^{\textbf{\emph{u}}, \textbf{\emph{w}}; r}$ and $\mathfrak{E}_{\textbf{\emph{f}}', \textbf{\emph{g}}'}^{\textbf{\emph{u}}', \textbf{\emph{w}}'; r'}$ are both nonempty. 
	
	If $\textbf{\emph{f}} \le \textbf{\emph{f}}'$, $\textbf{\emph{g}} \le \textbf{\emph{g}}'$, $\textbf{\emph{u}} \le \textbf{\emph{u}}'$, $\textbf{\emph{w}} \le \textbf{\emph{w}}'$, and $r \ge r'$, then it is possible to couple the uniform measures on $\mathfrak{E}_{\textbf{\emph{f}}, \textbf{\emph{g}}}^{\textbf{\emph{u}}, \textbf{\emph{w}}; r}$ and $\mathfrak{E}_{\textbf{\emph{f}}', \textbf{\emph{g}}'}^{\textbf{\emph{u}}', \textbf{\emph{w}}'; r'}$ on a common probability space such that the following holds. If the pair $(\mathcal{P}, \mathcal{P}') \in \mathfrak{E}_{\textbf{\emph{f}}, \textbf{\emph{g}}}^{\textbf{\emph{u}}, \textbf{\emph{w}}; r} \times \mathfrak{E}_{\textbf{\emph{f}}', \textbf{\emph{g}}'}^{\textbf{\emph{u}}', \textbf{\emph{w}}'; r'}$ is chosen with respect to this coupled measure, then $\mathcal{P} \le \mathcal{P}'$ almost surely. 
\end{prop}

\begin{rem}

\label{couplemonotonedomains}

Observe that \Cref{couplemonotone} also applies to rectangular subdomains $[m, n] \times [s, t]$ of $\mathbb{Z}^2$ or an augmented three-bundle domain $\mathcal{X}_{\Psi}$ (recall \Cref{xpsi}), since any such domain is a rectangular subdomain of $\mathcal{T}_{A, B, C}$ (after a suitable shift) for sufficiently large integers $A, B, C$. 

\end{rem} 

\begin{rem} 
	
\label{monotonepathf}

If $\textbf{f}$ does not pass through a diagonal edge of $\Lambda$, then the region between $\textbf{f}$ and $\textbf{g}$ lies to the right of the triangular face of $\mathcal{T} = \mathcal{T}_{A, B, C}$. Therefore, we can replace $\Lambda$ with a rectangular subdomain of $\mathbb{Z}^2$ (as opposed to of $\mathcal{T}$); the same holds if $(\textbf{f}, \textbf{g})$ is replaced by $(\textbf{f}', \textbf{g}')$. This allows us to compare random directed path ensembles on (rectangular subdomains of) $\mathcal{T}$ to those on (rectangular subdomains of) $\mathbb{Z}^2$. 
\end{rem}

The following corollary states that it is possible to couple uniformly chosen domain-wall, $A$-restricted directed path ensembles on $\mathcal{T}$ and $\mathcal{X}_{\Psi}$ in two different ways; the first bounds the second path in $\mathcal{X}_{\Psi}$ ensemble from right, and the latter bounds it from the left. In what follows, we recall the sets $\mathfrak{F} = \mathfrak{F}_{A, B, C}$ and $\mathfrak{G}_{\Psi}^A$ from \Cref{definitionf} and \Cref{domainaugmentedboundary}, respectively.

\begin{cor} 
	
	\label{ensemblesnn1}
	
	Let $\Psi$ denote a positive integer. It is possible to couple the uniform measures on $\mathfrak{F} = \mathfrak{F}_{A, B, C}$ and on $\mathfrak{G}_{\Psi}^A$ such that, if $\big(\mathcal{P}, \mathcal{P}^{\aug} \big) \in \mathfrak{F} \times \mathfrak{G}_{\Psi}^A$ is chosen with respect to this coupled measure, then the following holds. Denoting $\mathcal{P} = \big( \textbf{\emph{p}}_1, \textbf{\emph{p}}_2, \ldots , \textbf{\emph{p}}_{A + C} \big)$ and $\mathcal{P}^{\aug} = \big( \textbf{\emph{p}}_1^{\aug}, \textbf{\emph{p}}_2^{\aug}, \ldots , \textbf{\emph{p}}_{A + C + 1}^{\aug} \big)$, we have that $\textbf{\emph{p}}_2^{\aug} \le \textbf{\emph{p}}_1$ almost surely. It is also possible to couple these measures in such a way that, under the same notation, $\textbf{\emph{p}}_2 \le \textbf{\emph{p}}_2^{\aug}$ almost surely. 
\end{cor} 

\begin{proof}
	
	Let us denote the uniform measures on $\mathfrak{F}$ and $\mathfrak{G}_{\Psi}^A$ by $\mathbb{P}$ and $\mathbb{P}^{\aug}$, respectively. 
	
	We begin by establishing the first statement of the corollary. To that end, let $(\textbf{u}, \textbf{w})$ denote any pair constituting admissible boundary data on $\mathcal{X}$, and let $\textbf{g}$ be a path that enters and exits $\mathcal{X} = \mathcal{X}_{A, B, C; \Psi}$ at $u_0' = (0, -\Psi)$ and $w_0' = (A + 2B + C + 2, A + B + C + 1)$, respectively. In particular, $\textbf{g}$ enters and exits $\mathcal{X}$ in the same way as does $\textbf{p}_1^{\aug} \in \mathcal{P}^{\aug}$. Then, \Cref{couplemonotone} and \Cref{couplemonotonedomains} together imply that it is possible to couple the uniform measures on $\mathfrak{E} = \mathfrak{E}_{-\infty, \infty; \mathcal{X}}^{\textbf{u}, \textbf{w}; A}$ and $\mathfrak{E}' = \mathfrak{E}_{-\infty, \textbf{g}; \mathcal{X}}^{\textbf{u}, \textbf{w}; A}$ such that the following holds. If $(\mathcal{D}', \mathcal{D}) \in \mathfrak{E}' \times \mathfrak{E}$ is a pair of directed path ensembles on $\mathcal{X}$ chosen from this coupled measure, then $\mathcal{D}' \le \mathcal{D}$. 
	
	We in particular let $(\textbf{u}, \textbf{w})$ be obtained by removing the rightmost path from augmented domain-wall boundary data on $\mathcal{X}$ (recall \Cref{domainaugmentedboundary}). Stated alternatively, $(\textbf{u}, \textbf{w})$ is defined so that $\big( \textbf{u} \cup \{ u_0' \}, \textbf{w} \cup \{ w_0' \} \big)$ coincides with augmented domain-wall boundary data on $\mathcal{X}$. Observe that the law of the directed path ensemble $\mathcal{D}$, when restricted to $\mathcal{T}$, coincides with that of $\mathcal{P}$. 
	
	Now, average $\textbf{g}$ over the marginal of $\textbf{p}_1^{\aug} \in \mathcal{P}^{\aug}$ under $\mathbb{P}^{\aug}$. Then, after this averaging, the law of $\mathcal{D}'$ coincides with that of $\big( \textbf{p}_2^{\aug}, \textbf{p}_3^{\aug}, \ldots , \textbf{p}_{A + C + 1}^{\aug} \big)$. In view of this and the fact that $\mathcal{D}|_{\mathcal{T}}$ is has the same law as $\mathcal{P}$, upon averaging the coupling between $\mathcal{D}'$ and $\mathcal{D}$ in the same way and then comparing rightmost directed paths, we deduce the existence of a coupling between $\mathcal{P}$ and $\mathcal{P}^{\aug}$ such that $\textbf{p}_2^{\aug} \le \textbf{p}_1$ almost surely. This establishes the first statement of the corollary. 
	
	To establish the second statement of the corollary, we proceed similarly, except that we now average over the leftmost path in $\mathcal{P}^{\aug}$ instead of over the rightmost one. More specifically, let $A'' \le A$ denote any nonnegative integer, and let $\textbf{f}$ denote any path that enters and exits $\mathcal{X}$ at  $u_{A + C + 1}'' = (0, A + C)$ and  $w_{A + C + 1}'' = (A + 2B + 1, 2A + B + C + 1)$, respectively. In particular, $\textbf{f}$ enters and exits $\mathcal{X}$ in the same way as does $\textbf{p}_{A + C + 1}^{\aug} \in \mathcal{P}^{\aug}$.
	
	Define $\textbf{u}$ and $\textbf{w}$ as above (from augmented domain-wall boundary data on $\mathcal{X}$ with the rightmost path removed). Further let $(\textbf{u}'', \textbf{w}'')$ be obtained from removing the leftmost path in augmented domain-wall boundary data on $\mathcal{X}$. Stated alternatively, $(\textbf{u}'', \textbf{w}'')$ is defined so that $\big( \textbf{u}'' \cup \{ u_{A + C + 1}'' \}, \textbf{w}'' \cup \{ w_{A + C + 1}'' \} \big)$ coincides with augmented domain-wall boundary data on $\mathcal{X}$. Observe that $\textbf{u}'' \ge \textbf{u}$ and $\textbf{w}'' \ge \textbf{w}$. 
	
	Then, \Cref{couplemonotone} and \Cref{couplemonotonedomains} together imply that it is possible to couple the uniform measures on $\mathfrak{E} = \mathfrak{E}_{-\infty, \infty; \mathcal{X}}^{\textbf{u}, \textbf{w}; A}$ and $\mathfrak{E}'' = \mathfrak{E}_{\textbf{f}, \infty; \mathcal{X}}^{\textbf{u}'', \textbf{w}''; A''}$ such that the following holds. If $(\mathcal{D}'', \mathcal{D}) \in \mathfrak{E}'' \times \mathfrak{E}$ is a pair of directed path ensembles on $\mathcal{X}$ chosen from this coupled measure, then $\mathcal{D} \le \mathcal{D}''$. 
	
	Now, once again average $\textbf{f}$ over the marginal of $\textbf{p}_{A + C + 1}^{\aug} \in \mathcal{P}^{\aug}$ under $\mathbb{P}^{\aug}$, and set $A'' = \max \{ A - 1, 0 \}$ (which is the number of paths in $\mathcal{D}''$ that contain a diagonal edge). Then, the law of $\mathcal{D}''$ after this averaging coincides with that of $\big( \textbf{p}_1^{\aug}, \textbf{p}_2^{\aug}, \ldots , \textbf{p}_{A + C}^{\aug} \big)$. It follows as previously (using the fact that $\mathcal{D}|_{\mathcal{T}}$ is has the same law as $\mathcal{P}$) that averaging the above coupling between $\mathcal{D}$ and $\mathcal{D}''$ with respect to $\textbf{f}$ in this way yields a coupling between $\mathcal{P}$ and $\mathcal{P}^{\aug}$ such that $\textbf{p}_2 \le \textbf{p}_2^{\aug}$ almost surely. This establishes  the second statement of the corollary. 
\end{proof}

\subsection{Proof of \Cref{couplemonotone}} 

\label{ProofMonotone} 

The proof of \Cref{couplemonotone} will be similar to those of Lemma 2.6 and Lemma 2.7 of \cite{LE} and Lemma 2.6 and Lemma 2.7 of \cite{ELE}. In particular, it will use the Glauber dynamics on directed path ensembles on $\mathcal{T}$, which can be described as follows. We first define two switching operations on the quadrilateral faces of $\mathcal{T}$.

\begin{definition} 
	
	\label{movepathsdefinition}

Let $F$ be a quadrilateral face of $\mathcal{T}$, and denote its vertices by $v_1, v_2, v_3, v_4$, listed in counterclockwise order with $v_1$ to the southwest; see the left side of \Cref{ensemblemovepath}. The upwards and downwards switching operations on $F$ are defined to be procedures that alter the arrow configurations at the vertices of $F$. 

More specifically, if our path ensemble contains the two arrows connecting $(v_1, v_2)$ and $(v_2, v_3)$ but does not contain either of the arrows connecting $(v_1, v_4)$ and $(v_4, v_3)$, then the applying the \emph{upwards switching operation} to $F$ removes the former two arrows and inserts the latter two arrows (as shown in the left side of \Cref{ensemblemovepath}). Otherwise, applying this operation leaves the arrows along $F$ unchanged.

Similarly, if our path ensemble contains the two arrows connecting $(v_1, v_4)$ and $(v_4, v_3)$ but does not contain either of the arrows connecting $(v_1, v_2)$ and $(v_2, v_3)$, then the \emph{downwards switching operation} again removes the former two arrows and inserts the latter two arrows (see the right side of \Cref{ensemblemovepath}). Otherwise, it leaves the arrows along $F$ unchanged.

\end{definition}

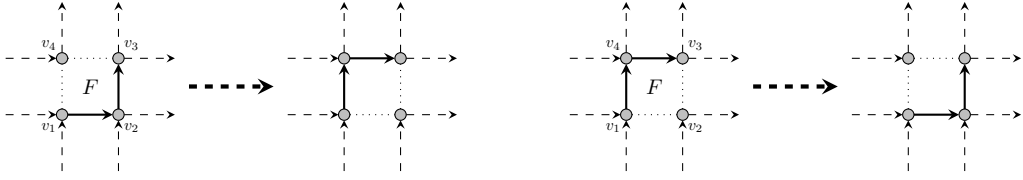
\begin{figure}[t]
	
	\begin{center}
		
		\begin{tikzpicture}[
		>=stealth,
		scale = .75
		]
		
		\draw[->, black, thick] (.1, 0) -- (.9, 0);
		\draw[->, black, thick] (1, .1) -- (1, .9);
		
		\draw[-, black, dotted] (.1, 1) -- (.9, 1);
		\draw[-, black, dotted] (0, .1) -- (0, .9);
				
		\draw[->, black, thick] (5, .1) -- (5, .9);
		\draw[->, black, thick] (5.1, 1) -- (5.9, 1);

		\draw[-, black, dotted] (5.1, 0) -- (5.9, 0);
		\draw[-, black, dotted] (6, .1) -- (6, .9);
		
		\draw[->, black, dashed] (4, 0) -- (4.9, 0);
		\draw[->, black, dashed] (4, 1) -- (4.9, 1);
		\draw[->, black, dashed] (6.1, 0) -- (7, 0);
		\draw[->, black, dashed] (6.1, 1) -- (7, 1);
		\draw[->, black, dashed] (5, -1) -- (5, -.1);
		\draw[->, black, dashed] (6, -1) -- (6, -.1);
		\draw[->, black, dashed] (5, 1.1) -- (5, 2);
		\draw[->, black, dashed] (6, 1.1) -- (6, 2);
		
		\filldraw[fill=gray!50!white, draw=black] (5, 0) circle [radius=.1];
		\filldraw[fill=gray!50!white, draw=black] (5, 1) circle [radius=.1];
		\filldraw[fill=gray!50!white, draw=black] (6, 0) circle [radius=.1];
		\filldraw[fill=gray!50!white, draw=black] (6, 1) circle [radius=.1];
		
		\draw[->, black, dashed] (-1, 0) -- (-.1, 0);
		\draw[->, black, dashed] (-1, 1) -- (-.1, 1);
		\draw[->, black, dashed] (1.1, 0) -- (2, 0);
		\draw[->, black, dashed] (1.1, 1) -- (2, 1);
		\draw[->, black, dashed] (0, -1) -- (0, -.1);
		\draw[->, black, dashed] (1, -1) -- (1, -.1);
		\draw[->, black, dashed] (0, 1.1) -- (0, 2);
		\draw[->, black, dashed] (1, 1.1) -- (1, 2);
		
		\filldraw[fill=gray!50!white, draw=black] (0, 0) circle [radius=.1] node[below = 5, left = 0, scale = .6]{$v_1$};
		\filldraw[fill=gray!50!white, draw=black] (1, 0) circle [radius=.1] node[below = 5, right = 0, scale = .6]{$v_2$};
		\filldraw[fill=gray!50!white, draw=black] (1, 1) circle [radius=.1] node[above = 5, right = 0, scale = .6]{$v_3$};
		\filldraw[fill=gray!50!white, draw=black] (0, 1) circle [radius=.1] node[above = 5, left = 0, scale = .6]{$v_4$};
		
		\draw[->, dashed, ultra thick] (2.25, .5) -- (3.75, .5);
		
		\filldraw[fill=gray!50!white, draw=black] (.5, .5) circle [radius=0] node[scale = .8]{$F$};

		\draw[-, black, dotted] (10.1, 0) -- (10.9, 0);
		\draw[-, black, dotted] (11, 0.1) -- (11, 0.9);
		
		\draw[-, black, dotted] (15.1, 1) -- (15.9, 1);
		\draw[-, black, dotted] (15, 0.1) -- (15, 0.9);
		
		\draw[->, black, thick] (15.1, 0) -- (15.9, 0);
		\draw[->, black, thick] (16, 0.1) -- (16, 0.9);
		
		\draw[->, black, dashed] (14, 0) -- (14.9, 0);
		\draw[->, black, dashed] (14, 1) -- (14.9, 1);
		\draw[->, black, dashed] (16.1, 0) -- (17, 0);
		\draw[->, black, dashed] (16.1, 1) -- (17, 1);
		\draw[->, black, dashed] (15, -1) -- (15, -.1);
		\draw[->, black, dashed] (16, -1) -- (16, -.1);
		\draw[->, black, dashed] (15, 1.1) -- (15, 2);
		\draw[->, black, dashed] (16, 1.1) -- (16, 2);
		
		\filldraw[fill=gray!50!white, draw=black] (15, 0) circle [radius=.1];
		\filldraw[fill=gray!50!white, draw=black] (15, 1) circle [radius=.1];
		\filldraw[fill=gray!50!white, draw=black] (16, 0) circle [radius=.1];
		\filldraw[fill=gray!50!white, draw=black] (16, 1) circle [radius=.1];
		
		\draw[->, black, thick] (10, 0.1) -- (10, 0.9);
		\draw[->, black, thick] (10.1, 1) -- (10.9, 1);
		
		\draw[->, black, dashed] (9, 0) -- (9.9, 0);
		\draw[->, black, dashed] (9, 1) -- (9.9, 1);
		\draw[->, black, dashed] (11.1, 0) -- (12, 0);
		\draw[->, black, dashed] (11.1, 1) -- (12, 1);
		\draw[->, black, dashed] (10, -1) -- (10, -.1);
		\draw[->, black, dashed] (11, -1) -- (11, -.1);
		\draw[->, black, dashed] (10, 1.1) -- (10, 2);
		\draw[->, black, dashed] (11, 1.1) -- (11, 2);
		
		\filldraw[fill=gray!50!white, draw=black] (10, 0) circle [radius=.1] node[below = 5, left = 0, scale = .6]{$v_1$};
		\filldraw[fill=gray!50!white, draw=black] (11, 0) circle [radius=.1] node[below = 5, right = 0, scale = .6]{$v_2$};
		\filldraw[fill=gray!50!white, draw=black] (11, 1) circle [radius=.1] node[above = 5, right = 0, scale = .6]{$v_3$};
		\filldraw[fill=gray!50!white, draw=black] (10, 1) circle [radius=.1] node[above = 5, left = 0, scale = .6]{$v_4$};
		
		\draw[->, dashed, ultra thick] (12.25, 0.5) -- (13.75, 0.5);
		
		\filldraw[fill=gray!50!white, draw=black] (10.5, 0.5) circle [radius=0] node[scale = .8]{$F$};

		\end{tikzpicture}
		
	\end{center}	
	
	\caption{\label{ensemblemovepath} Depicted above are the two potential results of applying a switching operation to a (quadrilateral) face $F$ of $\mathcal{T}$. } 
\end{figure}

Now we can define the Glauber dynamics on the set $\mathfrak{E}_{\textbf{f}, \textbf{g}; \Lambda}^{\textbf{u}, \textbf{w}; r}$ from \Cref{probabilitypaths}.

\begin{definition} 
	
	\label{ensemblemove}
	
	Adopt the notation of \Cref{probabilitypaths} and assign two exponential clocks, each of unit rate, to every quadrilateral face $F$ of $\Lambda \subseteq \mathcal{T}$ between $\textbf{f}$ and $\textbf{g}$ and that does not share an edge with either $\textbf{f}$ or $\textbf{g}$ (although it might share vertices with these two paths); we call the first clock the \emph{upwards clock} and the second the \emph{downwards clock} associated with $F$. If the triangular face is in $\Lambda$, it is not assigned a clock. 
	
	The \emph{Glauber dynamics} is defined to be the Markov chain on $\mathfrak{E}_{\textbf{f}, \textbf{g}; \Lambda}^{\textbf{u}, \textbf{w}; r}$ that applies the upwards (or downwards) switching operation to a quadrilateral face $F$ of $\Lambda$ whenever its upwards (or downwards, respectively) clock rings. 
	
\end{definition}

Observe that, since the triangular face of $\mathcal{T}$ is never switched, the local switching operation of \Cref{movepathsdefinition} preserves the set of $r$-restricted directed path ensembles on $\Lambda$, and so the Glauber dynamics indeed acts on $\mathfrak{E}_{\textbf{f}, \textbf{g}; \Lambda}^{\textbf{u}, \textbf{w}; r}$. The following lemma indicates that the stationary measure for these dynamics is uniform. 

\begin{lem}

\label{stationaryt}

Adopt the notation of \Cref{ensemblemove}. The unique stationary measure for the Glauber dynamics on $\mathfrak{E}_{\textbf{\emph{f}}, \textbf{\emph{g}}; \Lambda}^{\textbf{\emph{u}}, \textbf{\emph{w}}; r}$ is the uniform one. 
	
\end{lem} 

\begin{proof}
	
	The uniqueness of the stationary measure follows from the fact the Glauber dynamics is an irreducible Markov chain on the finite set $\mathfrak{E}_{\textbf{f}, \textbf{g}; \Lambda}^{\textbf{u}, \textbf{w}; r}$ (which is quickly verified by induction on $m = |\textbf{u}|$). That the uniform measure is stationary follows from its reversibility with respect to the Glauber dynamics, which is a consequence of the detailed balance condition and the fact that all switchings occur at unit rate. We refer to Section 1.6 and Section 3.3.2 of \cite{CMT} for more details. 
\end{proof}

To state the next lemma, first observe that the Glauber dynamics can be simultaneously coupled on all of the sets $\mathfrak{E}_{\textbf{f}, \textbf{g}; \Lambda}^{\textbf{u}, \textbf{w}; r}$ as follows. Each quadrilateral face of $\mathcal{T}$ is assigned an exponential clock of unit rate and, when the upwards (or downwards) clock assigned to a quadrilateral face $F$ rings, the Glauber dynamics on $\mathfrak{E}_{\textbf{f}, \textbf{g}; \Lambda}^{\textbf{u}, \textbf{w}; r}$ applies the upwards (or downwards, respectively) switching operation to $F$ if and only if $F$ is in $\Lambda$; is between $\textbf{f}$ and $\textbf{g}$; and does not share an edge with either of these two paths. If the face $F$ does not satisfy these conditions, then the Glauber dynamics on $\mathfrak{E}_{\textbf{f}, \textbf{g}; \Lambda}^{\textbf{u}, \textbf{w}; r}$ ignores this clock ring, and $F$ is not changed. 

The following lemma shows that monotonicity is preserved under the Glauber dynamics. 

\begin{lem}

\label{dynamicalmonotone}

Adopt the notation of \Cref{couplemonotone}, let $t \ge 0$ be a real number, and assume that $\mathcal{P} \in \mathfrak{E}_{\textbf{\emph{f}}, \textbf{\emph{g}}; \Lambda}^{\textbf{\emph{u}}, \textbf{\emph{w}}; r}$ and $\mathcal{P}' \in \mathfrak{E}_{\textbf{\emph{f}}', \textbf{\emph{g}}'; \Lambda}^{\textbf{\emph{u}}', \textbf{\emph{w}}'; r'}$ are directed path ensembles on $\Lambda$ such that $\mathcal{P} \le \mathcal{P}'$. 

Apply the Glauber dynamics on $\mathfrak{E}_{\textbf{\emph{f}}, \textbf{\emph{g}}; \Lambda}^{\textbf{\emph{u}}, \textbf{\emph{w}}; r}$ and $\mathfrak{E}_{\textbf{\emph{f}}', \textbf{\emph{g}}'; \Lambda}^{\textbf{\emph{u}}', \textbf{\emph{w}}'; r'}$, coupled as above. If $\textbf{\emph{G}}_t \mathcal{P}$ and $\textbf{\emph{G}}_t \mathcal{P}'$ denote the images of $\mathcal{P}$ and $\mathcal{P}'$, respectively, after running these dynamics for time $t$, then $\textbf{\emph{G}}_t \mathcal{P} \le \textbf{\emph{G}}_t \mathcal{P}'$ almost surely. 

\end{lem}

\begin{proof}
	
	It suffices to show that monotonicity is preserved under each clock ring. More specifically, let $F$ be any quadrilateral face of $\mathcal{T}$, and let $\mathcal{S} (\mathcal{P}) = \mathcal{S}_F (\mathcal{P}) \in \mathfrak{E}_{\textbf{f}, \textbf{g}; \Lambda}^{\textbf{u}, \textbf{w}; r}$ and $\mathcal{S} (\mathcal{P}') = \mathcal{S}_F (\mathcal{P}') \in \mathfrak{E}_{\textbf{f}', \textbf{g}'; \Lambda}^{\textbf{u}', \textbf{w}'; r'}$ denote the directed path ensembles obtained from $\mathcal{P}$ and $\mathcal{P}'$, respectively, after one of the two clocks assigned to $F$ rings; let us assume that it was an upwards clock, as the case of a downwards clock is entirely analogous. We claim that $\mathcal{S} (\mathcal{P}) \le \mathcal{S} (\mathcal{P}')$.

	Assume to the contrary that this is false. Denote the vertices of $F$ by $v_1, v_2, v_3, v_4$, ordered counterclockwise with $v_1$ is in the southwest corner. Since $\mathcal{P} \le \mathcal{P}'$ and since an upwards switching cannot increase any path in a directed path ensemble, we must have that $\mathcal{S} (\mathcal{P}') \ne \mathcal{P}'$. This implies that there are arrows in $\mathcal{P}'$ connecting $(v_1, v_2)$ and $(v_2, v_3)$ but no arrow connecting either $(v_1, v_4)$ or $(v_4, v_3)$; after the switching, the new directed path ensemble $\mathcal{S} (\mathcal{P}')$ contains the two arrows connecting $(v_1, v_4)$ and $(v_4, v_3)$ but no arrow connecting either $(v_1, v_2)$ or $(v_2, v_3)$. 
	
	The two arrows connecting $(v_1, v_2)$ and $(v_2, v_3)$ are part of some path $\textbf{p}_k' \in \mathcal{P}'$; denote the image of this path under the switching by $\mathcal{S} (\textbf{p}_k') \in \mathcal{S} (\mathcal{P}')$. Similarly define $\mathcal{S} (\textbf{p}_k) \in \mathcal{S} (\mathcal{P})$ to be the image of $\textbf{p}_k \in \mathcal{P}$ after this switching. Then, since $\textbf{p}_k \le \textbf{p}_k'$, if it does not hold that $\mathcal{S} (\textbf{p}_k) \le \mathcal{S} (\textbf{p}_k')$, then the arrows connecting $(v_1, v_2)$ and $(v_2, v_3)$ must be in $\textbf{p}_k$. 
	
	Therefore, neither arrow connecting $(v_1, v_4)$ or $(v_4, v_3)$ can be $\mathcal{P}$ since then $\textbf{p}_{k + 1} \le \textbf{p}_{k + 1}'$ would not hold. Thus, the upwards switching operation on $\mathcal{P}$ also moves the path $\textbf{p}_k$ upwards, meaning that $\mathcal{S} (\textbf{p}_k) \le \mathcal{S} (\textbf{p}_k')$. This is a contradiction, which implies that $\mathcal{S} (\mathcal{P}) \le \mathcal{S} (\mathcal{P}')$. Repeating this for each clock ring, we deduce the lemma. 	 
\end{proof}

\begin{proof}[Proof of \Cref{couplemonotone}]
	
	One can quickly verify (by induction on $m = |\textbf{u}|$, for example) that if $\mathfrak{E} = \mathfrak{E}_{\textbf{f}, \textbf{g}; \Lambda}^{\textbf{u}, \textbf{w}; r}$ and $\mathfrak{E}' = \mathfrak{E}_{\textbf{f}', \textbf{g}'; \Lambda}^{\textbf{u}', \textbf{w}'; r'}$ are both nonempty then there exist directed path ensembles $\mathcal{D} \in \mathfrak{E}$ and $\mathcal{D}' \in \mathfrak{E}'$ on $\Lambda$ such that $\mathcal{D} \le \mathcal{D}'$. Now apply the Glauber dynamics on $\mathfrak{E}$ and $\mathfrak{E}'$, coupled in the way described above \Cref{dynamicalmonotone}. This produces a Markov chain on $\mathfrak{E} \times \mathfrak{E}'$. For any real number $t > 0$, let $\textbf{G}_t (\mathcal{D}, \mathcal{D}') \in \mathfrak{E} \times \mathfrak{E}'$ denote the image of the pair $(\mathcal{D}, \mathcal{D}') \in \mathfrak{E} \times \mathfrak{E}'$ after running these dynamics for $t$.
	
	Since $\mathfrak{E} \times \mathfrak{E}'$ is finite, the sequence $\big\{ \textbf{G}_t (\mathcal{D}, \mathcal{D}') \big\}_{t > 0}$ has a limit point $(\mathcal{P}, \mathcal{P}')$, which is a random variable on $\mathfrak{E} \times \mathfrak{E}'$. By \Cref{stationaryt} and the finiteness of $\mathfrak{E} \times \mathfrak{E}'$, the marginal distributions of $\mathcal{P}$ and $\mathcal{P}'$ are uniform on $\mathfrak{E}$ and $\mathfrak{E}'$, respectively; this therefore produces a coupling between these two uniform probability measures. Applying \Cref{dynamicalmonotone} now implies that $\mathcal{P} \le \mathcal{P}'$ almost surely, from which we deduce the proposition. 
\end{proof}

\section{Tangency of the Added Path}

\label{TangentPath}

In this section we prove \Cref{p1p2tangent} below, which provides a sense in which the rightmost path in a augmented domain-wall $A$-restricted directed path ensemble on $\mathcal{X}$ is approximately tangent to the second path in this ensemble, with high probability. We establish this result in \Cref{TangentEstimate}, after providing some linearity estimates in \Cref{PathLinear}.

\subsection{Linearity of Random Walks} 

\label{PathLinear}

In this section we provide various estimates for the linear behavior of random directed paths. We begin with the following (known) bound on binomial coefficients that will be useful for establishing such estimates. Below, we let $\ell (v_1, v_2)$ denote the line through $v_1$ and $v_2$, for any $v_1, v_2 \in \mathbb{R}^2$.

\begin{lem} 
	
\label{estimateexponenial1} 

Let $M$ be a positive integer and $x, y, R, S \ge 0$ be integers such that $R + S = M$, $x \le R$, and $y \le S$. Let $d$ denote the distance from $(x, y) \in \mathbb{Z}^2$ to the line $\ell = \ell \big( (0, 0), (R, S) \big) \subset \mathbb{R}^2$ connecting $(0, 0)$ to $(R, S)$, that is, $d = d \big( (x, y), \ell \big)$. Then, 
\begin{flalign}
\label{xyrsestimate1} 
\binom{x + y}{x} \binom{R + S - x - y}{R - s} \binom{R + S}{R}^{-1} \le 48 M \exp \left( - \displaystyle\frac{d^2}{4 M} \right).
\end{flalign}
\end{lem} 

\begin{proof}

	Let us assume that $S \ge R$ (as the alternative case $R \ge S$ is entirely analogous), so that $S \ge \frac{M}{2}$. Denote the left side of \eqref{xyrsestimate1} by $P$. Then, the bound 
	\begin{flalign}
	\label{mestimate1} 
	(2 \pi M)^{1 / 2} \left( \displaystyle\frac{M}{e} \right)^M \le M! \le 2 (2 \pi M)^{1 / 2} \left( \displaystyle\frac{M}{e} \right)^M, 
	\end{flalign}
	
	\noindent which holds for any positive integer $M$ (see equations (1) and (2) of \cite{AR}), implies that 
	\begin{flalign}
	\label{pestimate1} 
	\begin{aligned}
	\log P & \le F_{y; R, S} (x)  - y \log y - (S - y) \log (S - y) \\
	& \qquad  - (R + S) \log (R + S) + R \log R + S \log S + \log (48 M), 
	\end{aligned}
	\end{flalign}
	
	\noindent where 
	\begin{flalign*}
	F_{y; R, S} (x) & = (x + y) \log (x + y) - x \log x \\
	& \qquad + (R + S - x - y) \log (R + S - x - y) - (R - x) \log (R - x).
	\end{flalign*}
	
	\noindent In particular, denoting $\theta = \frac{R}{S}$, we have that 
	\begin{flalign*}
	&  F_{y; R, S} (\theta y) = y \log y + (S - y) \log (S - y) +  (R + S) \log (R + S) - R \log R - S \log S; \\
	&	 F_{y; R, S}' (\theta y) = 0; \qquad F_{y; R, S}'' (x)  =   \displaystyle\frac{y - S}{(R - x) (R + S - x - y)} - \displaystyle\frac{y}{x(x + y)} \le - \displaystyle\frac{S}{M^2} \le -\frac{1}{2M},
	\end{flalign*}
	
	\noindent where we have used the bounds $S \ge \frac{M}{2}$; $x (x + y), (R - x)(R + S - x - y) \le M^2$; and $0 \le y \le S$. Integrating, it follows that for any $z \in \mathbb{R}$ such that $\theta y + z \in [0, S]$, we have $F_{y; R, S}' (\theta y + z) \le - \frac{|z|}{2 M}$, and so $F_{y; R, S} (\theta y + z) \le F_{y; R, S} (\theta y) - \frac{z^2}{4 M}$. Setting $x = \theta y + z$ and inserting this bound into \eqref{pestimate1} gives $\log P \le \log (48 M) - \frac{z^2}{4 M}$. Since $x = \theta y + z$ implies $z \ge d$, this yields the lemma.
\end{proof}

\Cref{linearrandom} quickly implies the following (also known) corollary, which states that random walks on $\mathbb{Z}^2$ with fixed endpoints are approximately linear with high probability. 

\begin{cor} 
	
\label{linearrandom}

Let $M \ge 1$ and $R, S \ge 0$ be integers such that $R + S = M$, and let $D > 0$ be a real number. Fix vertices $v_1 = (x_1, y_1) \in \mathbb{Z}^2$ and $v_2 = (x_2, y_2) \in \mathbb{Z}^2$ such that $x_2 - x_1 = R$ and $y_2 - y_1 = S$. If $\mathcal{P}$ denotes a uniformly random (up-right) directed path in $\mathbb{Z}^2$ starting at $v_1$ and ending at $v_2$, then 
\begin{flalign*}
\mathbb{P} \Big[ \displaystyle\max_{u \in \mathcal{P}} d \big( u, \ell (v_1, v_2) \big) \ge D \Big] < 48 M^3 \exp \left( - \displaystyle\frac{D^2}{4 M} \right).
\end{flalign*}
\end{cor} 

\begin{proof}
	
	We may assume that $v_1 = (0, 0)$, that $v_2 = (R, S)$, and that $R, S \ne 0$. Let $\Lambda = \Lambda_{v_1; v_2} = \big( [0, R] \times [0, S] \big) \cap \mathbb{Z}^2$ and, for any lattice point $u \in \Lambda$, define the event $E_u = \{ u \in \mathcal{P} \big\}$. Then, 
	\begin{flalign}
	\label{pdeltaestimate} 
	\mathbb{P} \Big[ d \big( \mathcal{P}, \ell (v_1, v_2) \big) \ge D \Big] < \displaystyle\sum_{d (u, \ell (v_1, v_2)) \ge D} \mathbb{P} [E_u],
	\end{flalign}
	
	\noindent where $u$ is summed over all lattice points in $\Lambda$ of distance greater than $D$ from the line $\ell (v_1, v_2)$. Denoting $u = (x, y)$ and applying \eqref{xyrsestimate1}, we find that
	\begin{flalign*}
	\mathbb{P} [E_u] = \binom{x + y}{x} \binom{R + S - x - y}{R - x} \binom{R + S}{R}^{-1} \le 48 M \exp \left( - \frac{D^2}{4M} \right),
	\end{flalign*}
	
	\noindent from which the lemma follows upon insertion into \eqref{pdeltaestimate} and the bound $|\Lambda| \le M^2$. 
\end{proof}

The next proposition shows that a random directed path on $\mathbb{Z}^2$ with fixed endpoints, conditioned to lie between two paths $\text{f}$ and $\textbf{g}$, is approximately linear if $\textbf{f}$ and $\textbf{g}$ do not ``push'' the path too far to the left or right (see the left side of \Cref{pathshift} for a depiction); its proof uses \Cref{couplemonotone} and \Cref{linearrandom}. In what follows, we define for any $M, D \in \mathbb{R}$ the quantity 
\begin{flalign}
\label{pmd}
p = p(M, D) = \exp \left( - \displaystyle\frac{D^2}{4 M} \right).
\end{flalign}

\begin{prop}
	
\label{linearrandomf}

Let $M \ge 1$ and $R, S \ge 0$ be integers such that $R + S = M$, and let $D, \Delta > 0$ be real numbers. Fix vertices $v_1 = (x_1, y_1) \in \mathbb{Z}^2$ and $v_2 = (x_2, y_2) \in \mathbb{Z}^2$ such that $x_2 - x_1 = R$ and $y_2 - y_1 = S$. Abbreviate $\ell = \ell (v_1, v_2)$ and $p = p (M, D)$. 

Let $\textbf{\emph{f}}$ and $\textbf{\emph{g}}$ denote directed paths on the rectangle $[x_1, x_2] \times [y_1, y_2]$ such that $\textbf{\emph{f}} \le \textbf{\emph{g}}$ and 
\begin{flalign*}  
\displaystyle\max_{u \in \textbf{\emph{f}} \cap \SE (\ell)} d ( u, \ell ) \le \Delta; \qquad \displaystyle\max_{u \in \textbf{\emph{g}} \cap \NW (\ell)} d ( u, \ell) \le \Delta. 
\end{flalign*}

\noindent Let $\textbf{\emph{p}}$ denote a uniformly random (up-right) directed path in $\mathbb{Z}^2$ starting at $v_1$ and ending at $v_2$, conditional on the event that $\textbf{\emph{f}} \le \textbf{\emph{p}} \le \textbf{\emph{g}}$. Then,
\begin{flalign*}
\mathbb{P} \left[ \displaystyle\max_{u \in \textbf{\emph{p}}} d (u, \ell) \ge 2 D + \Delta \right] < 192 M^3 p.
\end{flalign*}
	
\end{prop}

\begin{figure}[t]
	
	\begin{center}
		
		\begin{tikzpicture}[
		>=stealth,
		scale = .6
		]

		\draw[-, black] (0, 0) -- (0, 7) --  (8, 7) -- (8, 0) -- (0, 0);
		\draw[-, black, ultra thick] (0, 3) -- (1, 3) -- (1, 4) -- (2, 4) -- (3, 4) -- (4, 4) -- (5, 4) -- (5, 5) -- (5, 6) -- (5, 7);
		\draw[-, black, ultra thick] (2, 0) -- (2, 1) -- (2, 2) -- (3, 2) -- (4, 2) -- (5, 2) -- (6, 2) -- (6, 3) -- (7, 3) -- (7, 4) -- (8, 4);
		
		\draw[-, black, thick] (0, 1) -- (1, 1) -- (1, 2) -- (1, 3) -- (2, 3) -- (3, 3) -- (4, 3) -- (5, 3) -- (5, 4) -- (6, 4) -- (6, 5) -- (6, 6) -- (6, 7);
		
		\draw[-, dotted] (5, 4) -- (4, 5);
		
		\draw[] (3.3, 5) circle[radius = 0] node[above, scale = .75]{$\ell (v_1, v_2)$};
		\draw[] (4.6, 4.8) circle[radius = 0] node[scale = .75]{$\Delta$};
		\draw[] (0, 6) circle[radius = 0] node[left, scale = .75]{$\Lambda$};
		
		\draw[] (1, 4) circle[radius = 0] node[above, scale = .75]{$\textbf{f}$};
		\draw[] (6, 2) circle[radius = 0] node[below, scale = .75]{$\textbf{g}$};
		\draw[] (6, 5.5) circle[radius = 0] node[right, scale = .75]{$\textbf{p}$};
		
		\filldraw[fill = black] (0, 1) circle[radius = .071] node[left, scale = .75]{$v_1$};
		\filldraw[fill = black] (6, 7) circle[radius = .071] node[above, scale = .75]{$v_2$};
		\draw[-, dashed] (0, 1) -- (6, 7);

		\draw[-, black, ultra thick] (11, 3) -- (12, 3) -- (12, 4) -- (13, 4) -- (14, 4) -- (15, 4) -- (16, 4) -- (16, 5) -- (16, 6) -- (16, 7);
		
		\draw[-, black, thick] (13, -1) -- (14, -1) -- (14, 0) -- (15, 0) -- (15, 1) -- (15, 2) -- (16, 2) -- (16, 3) -- (17, 3) -- (18, 3) -- (18, 4) -- (19, 4) -- (19, 5);

		\draw[] (12, 4) circle[radius = 0] node[above, scale = .75]{$\textbf{f}$};
		\draw[] (18, 3.5) circle[radius = 0] node[right, scale = .75]{$\textbf{p}'$};
		
		\filldraw[fill = black] (11, 1) circle[radius = .071] node[left, scale = .75]{$v_1$};
		\filldraw[fill = black] (17, 7) circle[radius = .071] node[above, scale = .75]{$v_2$};
		
		\draw[->, black, dotted] (11, 1) -- (12.95, -.95);
		\draw[->, black, dotted] (17, 7) -- (18.95, 5.05);
		
		\draw[] (11.4, -.6) circle[radius = 0] node[above, scale = .75]{$D + \Delta$};
		\draw[] (18.6, 6.6) circle[radius = 0] node[below, scale = .75]{$D + \Delta$};
		
		\filldraw[fill = black] (13, -1) circle[radius = .071] node[below, scale = .75]{$v_1'$};
		\filldraw[fill = black] (19, 5) circle[radius = .071] node[right, scale = .75]{$v_2'$};
		
		\draw[-, dashed] (13, -1) -- (19, 5);

		\end{tikzpicture}
		
	\end{center}	
	
	\caption{\label{pathshift} The setting of \Cref{linearrandomf} is depicted to the left. Shown to the right is the shifting procedure used in the proof of that result. } 
\end{figure}
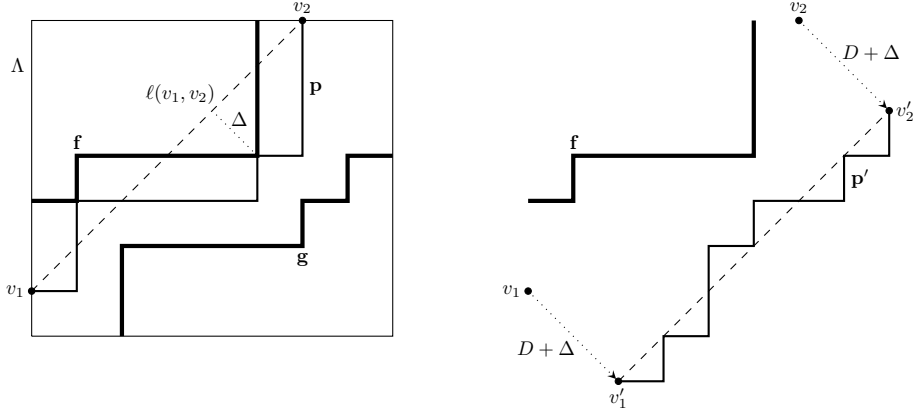

\begin{proof}
	
	As in the proof of \Cref{linearrandom}, we may assume that $v_1 = (0, 0)$ and that $v_2 = (R, S)$; set $\Lambda = \Lambda_{v_1; v_2} = \big( [0, R] \times [0, S] \big) \cap \mathbb{Z}^2$. We will show that 
	\begin{flalign}
	\label{pestimateab}
	\begin{aligned}
	& \mathbb{P} \left[ \displaystyle\max_{u \in \textbf{p} \cap \NW (\ell)} d ( u, \ell ) \ge 2 D + \Delta \right] < 96 M^3 p; \qquad \mathbb{P} \left[ \displaystyle\max_{u \in \textbf{p} \cap \SE (\ell)} d ( u, \ell ) \ge 2 D + \Delta \right] < 96 M^3 p,
	\end{aligned}
	\end{flalign}
	
	\noindent from which the result would follow from a union bound. We may assume in what follows that $48 M^3 p \le \frac{1}{2}$, for otherwise both terms on the right sides of \eqref{pestimateab} are greater than $1$. 
	
	Let us only establish the latter estimate in \eqref{pestimateab}, since the proof of the former is entirely analogous. To that end, we first produce a new random walk $\textbf{q}$ on $\mathbb{Z}^2$ as follows. Let $w = (w_1, w_2) \in \SE \big( (0, 0) \big)$ denote the vector orthogonal to $\ell$ in the southeast quadrant such that $|w| = D + \Delta$; for convenience, we assume that $w \in \mathbb{Z}^2$ (although it can be quickly seen that this is not necessary to implement the proof below). Let $\textbf{q}$ denote the uniformly random (up-right) directed walk on $\mathbb{Z}^2$ starting at $v_1' = v_1 + w = (w_1, w_2)$ and ending at $v_2' = v_2 + w = (R + w_1, S + w_2)$, conditional on the event that $\textbf{q} \ge \textbf{f}$. Observe in particular that $\textbf{q}$ is obtained by shifting the endpoints of $\textbf{p}$ to the southeast by $w$ and then by removing the east boundary $\textbf{g}$ for $\textbf{p}$. 
	
	Then, by \Cref{couplemonotone} and \Cref{couplemonotonedomains}, it is possible to couple the laws of $\textbf{p}$ and $\textbf{q}$ such that $\textbf{p} \le \textbf{q}$ almost surely. Denoting $\ell' = \ell (v_1', v_2')$, this implies that 
	\begin{flalign}
	\label{ddeltapestimatep}
	\begin{aligned}
	\mathbb{P} \left[ \displaystyle\max_{\substack{u \in \textbf{p} \cap \SE (\ell)}} d \big( u, \ell \big) \ge 2 D + \Delta \right] & \le \mathbb{P} \left[ \displaystyle\max_{\substack{u \in \textbf{q} \cap \SE (\ell)}} d \big( u, \ell \big) \ge 2 D + \Delta \right] \le \mathbb{P} \left[ \displaystyle\max_{u \in \textbf{q}} d \big( u, \ell' \big) \ge D \right], 
	\end{aligned}
	\end{flalign}
	
	\noindent where we used that, if $u \in \SE (\ell)$ is such that $d \big( u, \ell \big) \ge 2 D + \Delta$, then $d \big( u, \ell' \big) \ge D$. 
	
	Now let $\textbf{r}$ denote a uniformly random directed path on $\mathbb{Z}^2$ from $v_1'$ to $v_2'$. Let $\Omega_1$ denote the event on which $\max_{u \in \textbf{r}} d \big( u, \ell' \big) \le D$, and let $\Omega_2$ denote the event on which $\textbf{r} \ge \textbf{f}$; then, $\Omega_1 \subseteq \Omega_2$. Denoting the complement of $\Omega_i$ by $\Omega_i^c$ for $i \in \{ 1, 2 \}$, we deduce from \Cref{linearrandom} that $\mathbb{P} [\Omega_1^c] \le 48 M^3 p \le \frac{1}{2}$, and so $\mathbb{P} [\Omega_2] \ge \frac{1}{2}$. 
	
	Thus, since the law of $\textbf{q}$ is that of $\textbf{r}$, conditional on $\Omega_2$, it follows that
	\begin{flalign*}
	\mathbb{P} \left[ \displaystyle\max_{u \in \textbf{q}} d \big( u, \ell' \big) \ge D \right] \le \displaystyle\frac{\mathbb{P}[\Omega_1^c]}{\mathbb{P}[\Omega_2]} \le 96 M^3 p,
	\end{flalign*}
	
	\noindent and so the first estimate in \eqref{pestimateab} follows from \eqref{ddeltapestimatep}. As mentioned previously, the proof of the second estimate in \eqref{pestimateab} is very similar, and the corollary follows from applying a union bound.  	
\end{proof}

\begin{rem} 
	
	\label{tconvexrandomfg} 
	
	Similar to \Cref{monotonepathf}, it can be quickly seen that \Cref{linearrandomf} also holds if $\textbf{p}$ is a directed random path on a rectangular subdomain $\mathcal{X} = \mathcal{X}_{A, B, C; \Psi}$ for any integers $A, B, \Psi \ge 0$ and $C \ge 1$ (instead of on $\mathbb{Z}^2$), if we assume that $\textbf{f}$ does not contain a diagonal edge of $\mathcal{X}$. Indeed, in this case, the region between $\textbf{f}$ and $\textbf{g}$ lies to the right of the triangular face of $\mathcal{X}$. Thus, we can view $\textbf{p}$ as a random directed path on $\mathbb{Z}^2$ instead of on $\mathcal{X}$, in which setting \Cref{linearrandomf} applies.
	
\end{rem}

\subsection{Tangency Estimates} 

\label{TangentEstimate}

The following proposition (see \Cref{figuretangent}) considers a uniformly random augmented domain-wall $A$-restricted directed path ensemble $\mathcal{P}^{\aug}$ on $\mathcal{X}_{\Psi}$ and shows that the location where the rightmost path exits the $x$-axis is with high probability near tangent line through $(0, -\Psi)$ to the second curve in $\mathcal{P}^{\aug}$. 

In what follows, for $v_1, v_2 \in \mathbb{R}^2$, we recall that $\ell (v_1, v_2)$ denotes the line through $v_1$ and $v_2$. 

\begin{prop} 
	
	\label{p1p2tangent} 
	
	Let $N, C \ge 2$ and $A, B, \Psi \ge 0$ be integers such that $A + B + C = N$, and let $D > 1$ be a real number. Let $\mathcal{P}^{\aug} = (\textbf{\emph{p}}_1^{\aug}, \textbf{\emph{p}}_2^{\aug}, \ldots , \textbf{\emph{p}}_{A + C + 1}^{\aug})$ denote a uniformly random directed path ensemble from the set $\mathfrak{G}_{\Psi}^A$. Let $v \in \textbf{\emph{p}}_2^{\aug}$ be some vertex such that, if $\ell = \ell \big( (0, -\Psi), v \big)$, then $\textbf{\emph{p}}_2^{\aug} \subset \NW ( \ell)$. 
	
	Now let $\Phi \in [1, A + 2B + C + 1]$ denote the integer such that $\textbf{\emph{p}}_1^{\aug}$ contains an edge from $(\Phi, 0)$ to $(\Phi, 1)$. Then, 
	\begin{flalign*}
	\mathbb{P} \Big[ d \big( (\Phi, 0), \ell \big) \ge 2D + 1 \Big] \le 2^{24} (N + \Psi)^5 \exp \left( - \displaystyle\frac{D^2}{32 (N + \Psi)} \right).
	\end{flalign*}
	
\end{prop}

\begin{figure}[t]
	
	\begin{center}
		
		\begin{tikzpicture}[
		>=stealth,
		scale = .7
		]

		\draw[->] (1.5, 0) -- (10, 0); 
		\draw[->] (2, -2) -- (2, 6.5);

		\draw[-, black, thick] (2, -1.5) -- (2.45, -1.5);
		\draw[-, black, thick] (2.55, -1.5) -- (2.95, -1.5);
		\draw[-, black, thick] (3.05, -1.5) -- (3.45, -1.5);
		\draw[-, black, thick] (3.5, -1.45) -- (3.5, -1.05);
		\draw[-, black, thick] (3.55, -1) -- (3.95, -1);
		\draw[-, black, thick] (4, -.95) -- (4, -.55);
		\draw[-, black, thick] (4.05, -.5) -- (4.45, -.5);
		\draw[-, black, thick] (4.55, -.5) -- (4.95, -.5);	
		\draw[-, black, thick] (5.05, -.5) -- (5.45, -.5);
		\draw[-, black, thick] (5.5, -.45) -- (5.5, -.05);
		\draw[-, black, thick] (5.55, 0) -- (5.95, 0);
		\draw[-, black, thick] (6.05, 0) -- (6.45, 0);
		\draw[-, black, thick] (6.5, .05) -- (6.5, .45); 
		\draw[-, black, thick] (6.5, .55) -- (6.5, .95); 
		\draw[-, black, thick] (6.5, 1.05) -- (6.5, 1.45); 
		\draw[-, black, thick] (6.55, 1.5) -- (6.95, 1.5); 
		\draw[-, black, thick] (7, 1.55) -- (7, 1.95); 
		\draw[-, black, thick] (7.05, 2) -- (7.45, 2);  
		\draw[-, black, thick] (7.5, 2.05) -- (7.5, 2.45); 
		\draw[-, black, thick] (7.55, 2.5) -- (7.95, 2.5); 
			\draw[-, black, thick] (8, 2.55) -- (8, 2.95); 
		\draw[-, black, thick] (8, 3.05) -- (8, 3.45); 
		\draw[-, black, thick] (8, 3.55) -- (8, 3.95); 
		\draw[-, black, thick] (8, 4.05) -- (8, 4.45); 
		\draw[-, black, thick] (8.05, 4.5) -- (8.45, 4.5); 
		\draw[-, black, thick] (8.5, 4.55) -- (8.5, 4.95); 
		\draw[-, black, thick] (8.5, 5.05) -- (8.5, 5.45);
		
		\draw[->, black, thick] (8.55, 5) -- (8.95, 5);
		\draw[->, black, thick] (8.55, 5.5) -- (8.95, 5.5);
		
		\draw[-, black, very thick] (2, .5) -- (6, .5) -- (6, 1.5) -- (6.5, 1.5) -- (6.5, 2) -- (7, 2) -- (7, 4) -- (7.5, 4) -- (7.5, 4.5) -- (8, 4.5) -- (8, 5) -- (8.5, 5);
		
		\draw[->, black, dashed] (.875, -2.0625) -- (9.3, 2.15) node[right, scale = .7]{$\ell$};	
		
		\draw[->, dotted] (.875, -2) --  (9.3125, 1.75) node[right, scale = .7]{$\ell'$}; 
		
		\filldraw[fill=gray!50!white, draw=black] (2, -1.5) circle [radius = 0] node[above = 2, left = 0, scale = .7]{$(0, -\Psi)$};
		
		\filldraw[fill=gray!50!white, draw=black] (6, .5) circle [radius = 0] node[above = 4, left = 0, scale = .7]{$v$};
		\filldraw[fill=gray!50!white, draw=black] (7, 4) circle [radius = 0] node[above, scale = .7]{$\textbf{p}_2^{\aug}$};
		\filldraw[fill=gray!50!white, draw=black] (5, -.5) circle [radius = 0] node[below, scale = .7]{$\textbf{p}_1^{\aug}$};
		
		\filldraw[fill=gray!50!white, draw=black] (2.5, -1.5) circle [radius=.05];
		\filldraw[fill=gray!50!white, draw=black] (2.5, .5) circle [radius=.05];

		\filldraw[fill=gray!50!white, draw=black] (3, -1.5) circle [radius=.05];
		\filldraw[fill=gray!50!white, draw=black] (3, .5) circle [radius=.05];

		\filldraw[fill=gray!50!white, draw=black] (3.5, -1.5) circle [radius=.05];
		\filldraw[fill=gray!50!white, draw=black] (3.5, -1) circle [radius=.05];
		\filldraw[fill=gray!50!white, draw=black] (3.5, .5) circle [radius=.05];

		\filldraw[fill=gray!50!white, draw=black] (4, -1) circle [radius=.05];
		\filldraw[fill=gray!50!white, draw=black] (4, -.5) circle [radius=.05];
		\filldraw[fill=gray!50!white, draw=black] (4, .5) circle [radius=.05];

		\filldraw[fill=gray!50!white, draw=black] (4.5, -.5) circle [radius=.05];
		\filldraw[fill=gray!50!white, draw=black] (4.5, .5) circle [radius=.05];

		\filldraw[fill=gray!50!white, draw=black] (5, -.5) circle [radius=.05];
		\filldraw[fill=gray!50!white, draw=black] (5, .5) circle [radius=.05];

		\filldraw[fill=gray!50!white, draw=black] (5.5, -.5) circle [radius=.05];
		\filldraw[fill=gray!50!white, draw=black] (5.5, 0) circle [radius=.05];
		\filldraw[fill=gray!50!white, draw=black] (5.5, .5) circle [radius=.05];

		\filldraw[fill=gray!50!white, draw=black] (6, 0) circle [radius=.05];
		\filldraw[fill=gray!50!white, draw=black] (6, .5) circle [radius=.05];
		\filldraw[fill=gray!50!white, draw=black] (6, 1) circle [radius=.05];
		\filldraw[fill=gray!50!white, draw=black] (6, 1.5) circle [radius=.05];

		\filldraw[fill=gray!50!white, draw=black] (6.5, 0) circle [radius=.05] node[below = 1, scale = .7]{$(\Phi, 0)$};
		\filldraw[fill=gray!50!white, draw=black] (6.5, .5) circle [radius=.05] node[below = 1, right = 1, scale = .7]{$w$};
		\filldraw[fill=gray!50!white, draw=black] (6.5, 1) circle [radius=.05];
		\filldraw[fill=gray!50!white, draw=black] (6.5, 1.5) circle [radius=.05];
		\filldraw[fill=gray!50!white, draw=black] (6.5, 2) circle [radius=.05];

		\filldraw[fill=gray!50!white, draw=black] (7, 1.5) circle [radius=.05];
		\filldraw[fill=gray!50!white, draw=black] (7, 2) circle [radius=.05];
		\filldraw[fill=gray!50!white, draw=black] (7, 2.5) circle [radius=.05];
		\filldraw[fill=gray!50!white, draw=black] (7, 3) circle [radius=.05];
		\filldraw[fill=gray!50!white, draw=black] (7, 3.5) circle [radius=.05];
		\filldraw[fill=gray!50!white, draw=black] (7, 4) circle [radius=.05];

		\filldraw[fill=gray!50!white, draw=black] (7.5, 2) circle [radius=.05];
		\filldraw[fill=gray!50!white, draw=black] (7.5, 2.5) circle [radius=.05];	
		\filldraw[fill=gray!50!white, draw=black] (7.5, 4) circle [radius=.05];
		\filldraw[fill=gray!50!white, draw=black] (7.5, 4.5) circle [radius=.05];

		\filldraw[fill=gray!50!white, draw=black] (8, 2.5) circle [radius=.05];
		\filldraw[fill=gray!50!white, draw=black] (8, 3) circle [radius=.05];
		\filldraw[fill=gray!50!white, draw=black] (8, 3.5) circle [radius=.05];
		\filldraw[fill=gray!50!white, draw=black] (8, 4) circle [radius=.05];
		\filldraw[fill=gray!50!white, draw=black] (8, 4.5) circle [radius=.05];
		\filldraw[fill=gray!50!white, draw=black] (8, 5) circle [radius=.05];

		\filldraw[fill=gray!50!white, draw=black] (8.5, 4.5) circle [radius=.05];
		\filldraw[fill=gray!50!white, draw=black] (8.5, 5) circle [radius=.05];
		\filldraw[fill=gray!50!white, draw=black] (8.5, 5.5) circle [radius=.05];

		\end{tikzpicture}
		
	\end{center}	
	
	\caption{\label{figuretangent} 	The setting for \Cref{p1p2tangent} and its proof is depicted above.}
\end{figure}

\begin{proof}
	
	Throughout this proof, we abbreviate $\mathcal{P} = \mathcal{P}^{\aug}$ and $\textbf{p}_i = \textbf{p}_i^{\aug}$ for each $i \in [1, A +  C + 1]$. To establish this proposition, we will essentially let $w \in \mathcal{X}$ denote the location where $\textbf{p}_1$ ``almost intersects $\ell$'' and use \Cref{property} and \Cref{linearrandomf} to show that $\textbf{p}_1$ is approximately linear between $(0, -\Psi)$ and $w$. This will imply that $\textbf{p}_1$ is close to $\ell$ and thus that $(\Phi, 0) \in \textbf{p}_1$ is near $\ell$.

	To make this precise, we set $M = 8 (\Psi + N)$, which is an upper bound for the diameter of $\mathcal{X} = \mathcal{X}_{A, B, C; \Psi}$; further abbreviate $p = p (M, D)$ (recall \eqref{pmd}). Now define the event 
	\begin{flalign*} 
	 E = \Big\{ d \big( (\Phi, 0), \ell \big) \ge 2 D + 1 \Big\},
	 \end{flalign*}
	 
	 \noindent and let $E^c$ denote the complement of $E$. Furthermore, for any $w \in \SE (\ell) \cap \mathcal{X}$, also define the event 
	\begin{flalign} 
	\label{omegaw}
	\Omega (w) = \{ w \in \textbf{p}_1 \} \cap \big\{ \textbf{p}_2 \subset \NW (\ell') \big\} \cap \left\{ \displaystyle\max_{u \in \textbf{p}_1 \cap \SW (w)} d \big( u, \ell' \big) \ge 2D \right\},
	\end{flalign}
	
	\noindent where we have abbreviated $\ell' = \ell \big( (0, -\Psi), w \big)$. See \Cref{figuretangent} for an example.
	
	We claim that 
	\begin{flalign}
	\label{fw1}
	E \subseteq \bigcup_{w \in \mathcal{X}} \Omega (w); \qquad \displaystyle\max_{w \in \mathcal{X}} \mathbb{P} \big[ \Omega(w) \big] \le 192 M^3 p. 
	\end{flalign} 
	
	Let us first establish the former statement in \eqref{fw1}, to which end we restrict to $E$ and show that $\bigcup_{w \in \mathcal{X}} \Omega (w)$ holds. To do this, recall that $v \in \textbf{p}_2 \cap \ell$, that $\textbf{p}_2 \le \ell$, and that $\textbf{p}_1 \ge \textbf{p}_2$. Thus, since the sites at which $\textbf{p}_1$ and $\textbf{p}_2$ exit $\mathcal{X}$ are of distance one from each other, there must exist a vertex $w \in \SE (\ell) \cap \textbf{p}_1 \cap \NE (v)$ such that $d (w, \ell) \le 1$.
	 
	Restricting to $E$, we will show that $\Omega (w)$ holds. To that end, observe that $w \in \textbf{p}_1$ and $\textbf{p}_2 \subset \NW (\ell')$, the latter since $v \in \ell$, $\textbf{p}_2 \subset \NW (\ell)$, $w \in \NE (v) \cap \SE (\ell)$, and $\ell \cap \ell' = \big\{ (0, -\Psi) \big\}$. 
	
	Thus, it remains to verify the last condition in \eqref{omegaw}; we will in fact show that it holds with $u = (\Phi, 0)$. Under this choice of $u$, we have that $d (u, \ell) \ge 2D + 1$ since we are restricting to $E$. Since $d (w, \ell) \le 1$ and $u \in \SW (w)$, it follows that $d (u, \ell') \ge d (u, \ell) - 1 \ge 2D$. This confirms the last statement of \eqref{omegaw} and thus the first claim of \eqref{fw1}. 
	
	To establish the second claim of \eqref{fw1}, fix some $w = (x, y) \in \SE (\ell) \cap \mathcal{X}$, and define the rectangular subdomain $\Lambda = \Lambda_w = \big( [0, x] \times [-\Psi, y] \big) \cap \mathcal{X}$ of $\mathcal{X}$. Now restrict the directed path ensemble $\mathcal{P}$ to $\Lambda$, and condition on its $A + C$ north paths $(\textbf{p}_2, \textbf{p}_3, \ldots , \textbf{p}_{A + C + 1})$, as well as on the event that $\textbf{p}_1$ exits $\Lambda$ through $w$. 
	
	In view of the Gibbs property \Cref{property}, the law of $\textbf{p}_1 |_{\Lambda}$ is given by that of a uniformly random walk on $\Lambda$ conditioned to enter at $(0, -\Psi)$, exit at $w$, and satisfy $\textbf{p}_1 |_{\Lambda} \ge \textbf{p}_2 |_{\Lambda}$. Since $\textbf{p}_2 \subset \NW (\ell)$ and $\textbf{p}_2$ is to the right of the triangular face of $\mathcal{X}$ (since $C \ge 2$), \Cref{linearrandomf} and \Cref{tconvexrandomfg} (applied with the $\textbf{f}$, $\textbf{g}$, and $\Delta$ there equal to $\textbf{p}_2$, $\infty$, and $0$ here, respectively) together yield the second estimate in \eqref{fw1}. 

	Now the proposition follows from \eqref{fw1}, a union bound, and the fact that $|\mathcal{X}| \le M^2$. 
\end{proof}

\section{Proximity of \texorpdfstring{$\textbf{p}_1$}{} and \texorpdfstring{$\textbf{p}_2$}{}}

\label{PathsNear}

In this section we show that if $\mathcal{P} = (\textbf{p}_1, \textbf{p}_2, \ldots , \textbf{p}_{A + C}) \in \mathfrak{F}$ (recall \Cref{definitionf}) is sampled uniformly at random then $\textbf{p}_1$ and $\textbf{p}_2$ are close to each other with high probability. We begin in \Cref{PathConvex} by introducing a notion of approximate convexity and showing that $\textbf{p}_1$ and $\textbf{p}_2$ are approximately convex (with high probability) with respect to this notion; then we use this to establish the proximity statement between $\textbf{p}_1$ and $\textbf{p}_2$ in \Cref{PathsNearp1p2}.

\subsection{Convexity Estimates}

\label{PathConvex}

We begin with the following definition. 

\begin{definition} 
	
\label{xidefinition} 

For any $u, v \in \mathbb{R}^2$ with $v \in \NE (u)$, define the sets
\begin{flalign}
\label{setsuv}
\NW (u, v) = \NE (u) \cap \SW (v) \cap \NW \big( \ell (u, v) \big); \qquad \SE (u, v) = \NE (u) \cap \SW (v) \cap \SE \big( \ell (u, v) \big).
\end{flalign} 

\noindent Stated alternatively, if $u = (x, y)$ and $v = (x', y')$, then $\NW (u, v)$ is the part of the rectangle $[x, x'] \times [y, y']$ above (northwest) of the line $\ell (u, v)$; a similar statement holds for $\SE (u, v)$. 

Now, for any directed up-right path $\textbf{p}$ on a rectangular subdomain of some three-bundle domain $\mathcal{T}_{A, B, C}$, define the quantity
\begin{flalign}
\label{xip}
\Xi (\textbf{p}) = \displaystyle\max_{u \in \textbf{p}} \displaystyle\max_{ v \in \textbf{p} \cap \NE (u)} \displaystyle\max_{w \in \textbf{p} \cap \NW (u, v)} d \big( w, \ell (u, v) \big). 
\end{flalign}

\end{definition}

One might view $\Xi (\textbf{p})$ as some measure of ``approximate convexity'' for the path $\textbf{p}$; see the left side of \Cref{figureconvex} for a depiction. Observe in particular that if $\Xi (\textbf{p}) = 0$ then $\textbf{p}$ is (weakly) convex; a more general estimate in terms of the lower convex envelope of $\textbf{p}$ is given by \eqref{dxi} below. 

We would like to bound $\Xi (\textbf{p}_1)$ and $\Xi (\textbf{p}_2)$ if $\mathcal{P} = (\textbf{p}_1, \textbf{p}_2, \ldots , \textbf{p}_{A + C}) \in \mathfrak{F}$ is sampled uniformly at random. The first lemma below bounds the former, and the second bounds the latter.

\begin{figure}[t]
	
	\begin{center}
		
		\begin{tikzpicture}[
		>=stealth,
		scale = .5
		]

		\draw[-, black, thick] (5, 1) -- (6, 1) -- (7, 1) -- (8, 1) -- (8, 2) -- (9, 2) -- (9, 3) -- (9, 4) -- (9, 5) -- (10, 5) -- (11, 5) -- (12, 5) -- (12, 6) -- (13, 6) -- (13, 7) -- (13, 8) node[scale = .7, above]{$\textbf{p}$};

		\draw[-, black, dotted] (9, 5) -- (10.5, 3.5);
		\draw[<->, black, dashed] (7, 0) -- (14, 7) node[scale = .7, right]{$\ell (u, v)$};
		
		\draw[] (10.5, 4.4) circle[radius = 0] node[scale = .7]{$\Xi (\textbf{p})$};
		
		\filldraw[fill = black] (9, 2) circle[radius = .075] node[scale = .7, right]{$u$};
		\filldraw[fill = black] (12, 5) circle[radius = .075] node[scale = .7, right]{$v$};
		\filldraw[fill = black] (9, 5) circle[radius = .075] node[scale = .7, above]{$w$};

		\draw[-, black] (20, 0) -- (20, 8) --  (28, 8) -- (28, 0) -- (20, 0);
		\draw[-, black, ultra thick] (20, 3) -- (21, 3) -- (21, 4) -- (22, 4) -- (23, 4) -- (24, 4) -- (25, 4) -- (25, 5) -- (25, 6) -- (25, 7) -- (25, 8);
		\draw[-, black, ultra thick] (22, 0) -- (22, 1) -- (22, 2) -- (23, 2) -- (24, 2) -- (25, 2) -- (26, 2) -- (26, 3) -- (27, 3) -- (27, 4) -- (28, 4);
		
		\draw[-, black, thick] (21, 0) -- (21, 1) -- (21, 2) -- (22, 2) -- (22, 3) -- (23, 3) -- (24, 3) -- (25, 3) -- (25, 4) -- (26, 4) -- (27, 4) -- (27, 5) -- (27, 6) -- (27, 7) -- (28, 7);
		
		\draw[-, dotted] (22, 2) -- (22.5, 1.5) node[right = 3, below, scale = .75]{$\Delta$};
		
		\draw[] (20, 6) circle[radius = 0] node[left, scale = .75]{$\Lambda$};
		
		\draw[] (21, 4) circle[radius = 0] node[above, scale = .75]{$\textbf{p}_3$};
		\draw[] (26, 2) circle[radius = 0] node[below, scale = .75]{$\textbf{p}_1$};

		\draw[] (27, 7) circle[radius = 0] node[left, scale = .75]{$\textbf{p}_2$};
		
		\filldraw[fill = black] (21, 0) circle[radius = .071] node[below, scale = .75]{$v_1$};
		\filldraw[fill = black] (28, 7) circle[radius = .071] node[right, scale = .75]{$v_2$};
		\draw[-, dashed] (21, 0) -- (28, 7);
		
		\end{tikzpicture}
		
	\end{center}	
	
	\caption{\label{figureconvex} Depicted to the left is the quantity $\Xi (\textbf{p})$ for some directed path $\textbf{p}$. Depicted to the right is the setting for the proof of \Cref{convexp2x}. } 
\end{figure}
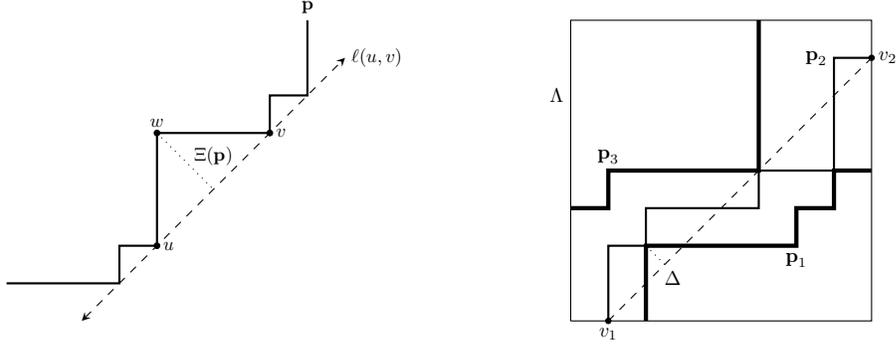

\begin{lem}
	
	\label{convexp1x}

	Let $A, B \ge 0$ and $N, C \ge 3$ be integers such that $A + B + C = N$, and let $D > 0$ be a real number. Further let $\mathcal{P} = (\textbf{\emph{p}}_1, \textbf{\emph{p}}_2, \ldots , \textbf{\emph{p}}_{A + C})$ denote a uniformly random directed path ensemble from $\mathfrak{F} = \mathfrak{F}_{A, B, C}$. Then, 
	\begin{flalign*}
	\mathbb{P} \big[ \Xi (\textbf{\emph{p}}_1) \ge 2 D \big] \le 2^{18} N^7 \exp \left( - \displaystyle\frac{D^2}{16 N} \right).
	\end{flalign*}
	
\end{lem}

\begin{proof}
	
	Define the event $E = \big\{ \Xi (\textbf{p}_1) \ge 2D \big\}$. Furthermore, for any $u = (x, y) \in \mathcal{T} = \mathcal{T}_{A, B, C}$ and $v = (x', y') \in \mathcal{T} \cap \NE(u)$, define the rectangular subdomain $\Lambda = \Lambda_{u, v} = \big( [x, x'] \times [y, y'] \big) \cap \mathcal{T} \subseteq mathcal{T}$. For any directed path $\textbf{p}$ on $\Lambda$ that enters and exits through $u$ and $v$, respectively, define the event
	\begin{flalign*}
	E (u, v) = E_{\textbf{p}} (u, v) = \{ u, v \in \textbf{p} \} \cap \left\{ \displaystyle\max_{w \in \NW (u, v) \cap \textbf{p}} d ( w, \ell) \ge 2 D \right\},
	\end{flalign*} 
	
	\noindent where we abbreviated $\ell = \ell (u, v)$. Observe that $E \subseteq \bigcup_{u \in \mathcal{T}} \bigcup_{v \in \mathcal{T} \cap \NE (u)} E_{\textbf{p}_1} (u, v)$. Thus, we will estimate $\mathbb{P} \big[ E_{\textbf{p}_1} (u, v) \big]$, to which end we will use \Cref{property}, \Cref{couplemonotone}, and \Cref{linearrandomf}.
	
	More specifically, let us fix vertices $u \in \mathcal{T}$ and $v \in \mathcal{T} \cap \NE (u)$ and then condition on the paths $(\textbf{p}_2, \textbf{p}_3, \ldots , \textbf{p}_{A + C})$; on the restriction $\textbf{p}_1 |_{\mathcal{T} \setminus \Lambda}$ (namely, the part of $\textbf{p}_1$ not in $\Lambda$); and on the event that $\textbf{p}_1$ enters and exits $\Lambda$ through $u$ and $v$, respectively. Then by \Cref{property}, the conditional law of $\textbf{p}_1 |_{\Lambda}$ is given by the uniform measure on the set of $0$-restricted directed paths that enter and exit $\Lambda$ through $u$ and $v$, respectively, and satisfy $\textbf{p}_1 \ge \textbf{p}_2$. The fact that $\textbf{p}_1$ must be $0$-restricted is due to the fact that it lies to the right of the triangular face of $\mathcal{T}$, since $C \ge 2$.
	
	Now let $\textbf{q}$ denote a uniformly random $0$-restricted directed path on $\mathbb{Z}^2$, conditioned to start and end at $u$ and $v$, respectively. In view of \Cref{couplemonotone} and \Cref{monotonepathf}, it is possible to couple the laws of $\textbf{p}_1$ and $\textbf{q}$ such that $\textbf{p}_1 \ge \textbf{q}$. Therefore, $\mathbb{P} \big[ E_{\textbf{p}_1} (u, v) \big] \le \mathbb{P}\big[ E_{\textbf{q}} (u, v) \big]$.
	
	To bound $\mathbb{P}\big[ E_{\textbf{q}} (u, v) \big]$, we apply \Cref{linearrandomf} and \Cref{tconvexrandomfg} with the $(v_1, v_2)$ there equal to the $(u, v)$ here, the $(\textbf{f}, \textbf{g})$ there equal to $(-\infty, \infty)$ here, the $\Delta$ there equal to $0$ here, and the $M$ bounded above by $4N$ here. This yields 
	\begin{flalign}
	\label{xip1d}
	\mathbb{P} \big[ E_{\textbf{q}} (u, v) \big] \le 2^{14} N^3 \exp \left( - \displaystyle\frac{D^2}{16 N} \right).
	\end{flalign}
	
		Therefore, combining the previously mentioned estimate $\mathbb{P} \big[ E_{\textbf{p}_1} (u, v) \big] \le \mathbb{P}\big[ E_{\textbf{q}} (u, v) \big]$ and containment of events $E \subseteq \bigcup_{u \in \mathcal{T}} \bigcup_{v \in \mathcal{T} \cap \NE (u)} E_{\textbf{p}_1} (u, v)$ with \eqref{xip1d} yields through a union bound that
	\begin{flalign*}
	\mathbb{P} [E] \le \displaystyle\sum_{u \in \mathcal{T}} \displaystyle\sum_{v \in \mathcal{T} \cap \NE (u)} \mathbb{P} \big[ E_{\textbf{p}_1} (u, v) \big] & \le \displaystyle\sum_{u \in \mathcal{T}} \displaystyle\sum_{v \in \mathcal{T} \cap \NE (u)} \textbf{P} \big[ E_{\textbf{q}} (u, v) \big] \le 2^{14} N^3 |\mathcal{T}|^2 \exp \left( - \displaystyle\frac{D^2}{16 N} \right),
	\end{flalign*} 
	
	\noindent from which we deduce the lemma since $|\mathcal{T}| \le 4 N^2$. 
\end{proof}

\begin{lem}
	
	\label{convexp2x}
	
	Adopting the notation of \Cref{convexp1x}, we have that 
	\begin{flalign*}
	\mathbb{P} \big[ \Xi (\textbf{\emph{p}}_2) \ge 4D \big] \le 2^{19} N^7 \exp \left( - \displaystyle\frac{D^2}{16 N} \right).
	\end{flalign*}
	
\end{lem}

\begin{proof}
	
	The proof of this lemma will be similar to that of \Cref{convexp1x}.
	
	As in that proof, we first define the events $E = \big\{ \Xi (\textbf{p}_2) \ge 4 D \big\}$ and $F = \big\{ \Xi (\textbf{p}_1) \ge 2 D \big\}$; let $F^c$ denote the complement of $F$. By \Cref{convexp1x}, we have that 
	\begin{flalign}
	\label{probabilityep1p2}
	\mathbb{P} [F] \le 2^{18} N^7 \exp \left( - \displaystyle\frac{D^2}{16 N} \right).
	\end{flalign}
	
	Moreover, for any $u = (x, y) \in \mathcal{T} = \mathcal{T}_{A, B, C}$ and $v = (x', y') \in \mathcal{T} \cap \NE (u)$, denote the rectangular subdomain $\Lambda = \Lambda_{u, v} = \big( [x, x'] \times [y, y'] \big) \cap \mathcal{T} \subseteq \mathcal{T}$ (as in the proof of \Cref{convexp1x}). For any directed path $\textbf{p}$ on $\Lambda_{u, v}$ that enters and exits through $u$ and $v$, respectively, define the event
	\begin{flalign*}
	E_{\textbf{p}} (u, v) = \{ u, v \in \textbf{p} \} \cap  \left\{ \displaystyle\max_{w \in \NW (u, v) \cap \textbf{p}} d (w, \ell) \ge 4D \right\},
	\end{flalign*}
	
	\noindent where we have abbreviated $\ell = \ell (u, v)$. 
	
	Since 
	\begin{flalign}
	\label{evfe}
	E \subseteq \bigcup_{u \in \mathcal{T}} \bigcup_{v \in \mathcal{T} \cap \NE (u)} E_{\textbf{p}_2} (u, v)  = F \cup \bigcup_{u \in \mathcal{T}} \bigcup_{v \in \mathcal{T} \cap \NE (u)} \big( E_{\textbf{p}_2} (u, v) \cap F^c \big),
	\end{flalign}
	
	\noindent we will estimate the probability $\mathbb{P} \big[ E_{\textbf{p}_2} (u, v) \cap F^c \big]$ for each $u$ and $v$.
	
	To that end, let us fix vertices $u \in \mathcal{T}$ and $v \in \mathcal{T} \cap \NE (u)$, and then condition on the paths $(\textbf{p}_1, \textbf{p}_3, \ldots , \textbf{p}_{A + C})$; on the restriction $\textbf{p}_2 |_{\mathcal{T} \setminus \Lambda}$ (namely, the part of $\textbf{p}_2$ not in $\Lambda$); and on the event that $\textbf{p}_2$ enters and exits $\Lambda$ through $u$ and $v$, respectively. Then by \Cref{property}, the conditional law of $\textbf{p}_2 |_{\Lambda}$ is given by the uniform measure on the set of $0$-restricted directed path ensembles that enter and exit $\Lambda$ through $u$ and $v$, respectively, and satisfy $\textbf{p}_3 \le \textbf{p}_2 \le \textbf{p}_1$. 
	
	Now let $\textbf{q}$ denote a uniformly random $0$-restricted directed path on $\mathbb{Z}^2$, conditioned on the event that $\textbf{q} \le \textbf{p}_1$ and also conditioned to start and end at $u$ and $v$, respectively. In view of \Cref{couplemonotone} and \Cref{monotonepathf} (recall that $\textbf{p}_3$ lies to the right of the triangular face of $\mathcal{T}$, since $C \ge 3$), it is possible to couple the laws of $\textbf{p}_2$ and $\textbf{q}$ such that $\textbf{p}_1 \ge \textbf{q}$. Therefore, 
	\begin{flalign} 
	\label{fp2q} 
	\mathbb{P} \big[ E_{\textbf{p}_2} (u, v) \cap F^c \big] \le \mathbb{P} \big[ E_{\textbf{q}} (u, v) \cap F^c \big].
	\end{flalign}
	
	To bound the right side of \eqref{fp2q}, observe that $\textbf{1}_{F^c} \Xi (\textbf{p}_1) \le 2D$, which since $\textbf{p}_1 \ge \textbf{q}$ implies that $\textbf{1}_{F^c} \max_{u \in \textbf{p}_1 \cap \NW (\ell)} d (u, \ell) \le 2D$. Thus, we may apply \Cref{linearrandomf} and \Cref{tconvexrandomfg} with the $(v_1, v_2)$ there equal to $(v_1, v_2)$ here; the $\textbf{p}$ there equal to $\textbf{q}$ here; the $(\textbf{f}, \textbf{g})$ there equal to $(-\infty, \textbf{p}_1)$ here; the $\Delta$ there equal to $2D$ here; and the $M$ there bounded above by $4N$ here, to deduce that
	\begin{flalign}
	\label{xip1d2}
	\mathbb{P} \big[ E_{\textbf{q}} (u, v) \cap F^c \big] \le 2^{14} N^3 \exp \left( - \displaystyle\frac{D^2}{16 N} \right). 
	\end{flalign}
	
	\noindent Combining \eqref{probabilityep1p2}, \eqref{evfe}, \eqref{fp2q}, and \eqref{xip1d2} yields
	\begin{flalign*}
	\mathbb{P} [E] & \le \mathbb{P} [F] + \displaystyle\sum_{u \in \mathcal{T}} \displaystyle\sum_{v \in \mathcal{T} \cap \NE (u)} \mathbb{P} \big[ E_{\textbf{p}_2} (u, v) \cap F^c \big] \\
	& \le \displaystyle\sum_{u \in \mathcal{T}} \displaystyle\sum_{v \in \mathcal{T} \cap \NE (u)} \mathbb{P} \big[ E_{\textbf{q}} (u, v) \cap F^c \big] + 2^{18} N^7 \exp \left( - \displaystyle\frac{D^2}{16 N} \right) \\
	& \le 2^{14} N^3 |\mathcal{T}|^2 \exp \left( - \displaystyle\frac{D^2}{16 N} \right) + 2^{18} N^7 \exp \left( - \displaystyle\frac{D^2}{16 N} \right) \le 2^{19} N^7 \exp \left( - \displaystyle\frac{D^2}{16 N} \right),
	\end{flalign*} 
	
	\noindent where we have used the fact that $|\mathcal{T}| \le 4 N^2$. This implies the lemma. 
\end{proof}

\subsection{Proximity Estimates}

\label{PathsNearp1p2}

In this section we show that, if $\mathcal{P} = (\textbf{p}_1, \textbf{p}_2, \ldots , \textbf{p}_{A + C})$ is a directed path ensemble chosen uniformly at random from $\mathfrak{F}$, then any point on the curve $\textbf{p}_1$ is close to $\textbf{p}_2$ with high probability. 

Before doing so, however, it will be useful to recall the notion of a lower convex envelope. Let $\Lambda = \big( [x_1, x_2] \times [y_1, y_2] \big) \cap \mathbb{Z}^2$ denote a rectangular subdomain of $\mathbb{Z}^2$, and let $\textbf{p}$ denote an (up-right) directed path on $\Lambda$. The \emph{lower convex envelope} $\textbf{h} (\textbf{p}) \subset [x_1, x_2] \times [y_1, y_2]$ of $\textbf{p}$ is defined to be the lower boundary of $\Lambda \cap \bigcup_{u, v \in \textbf{p}} \NW  (u, v)$ (where in the union we assume that $v \in \NE (u)$). Equivalently, it is the topmost convex curve (restricted to $\Lambda$) that lies below $\textbf{p}$. We refer to the left side of \Cref{u1w1p} for a depiction. 

Recalling the definition of $\Xi (\textbf{p})$ from \eqref{xip}, observe for any directed path $\textbf{p}$ on some rectangular subdomain $\Lambda \subset \mathbb{Z}^2$ that 
\begin{flalign}
\label{dxi}
\displaystyle\max_{u \in \textbf{p}} d \big( u, \textbf{h} (\textbf{p}) \big) + \displaystyle\max_{v \in \textbf{h} (p)} d (v, \textbf{p}) \le 2 \Xi (\textbf{p}).
\end{flalign}

Now we have the following proposition stating that $\textbf{p}_1$ and $\textbf{p}_2$ are ``close'' in a certain sense. 

\begin{prop}
	
	\label{p1p2near} 
	
	Adopting the notation of \Cref{convexp1x}, we have that 
	\begin{flalign*}
	\mathbb{P} \left[ \displaystyle\max_{v_1 \in \textbf{\emph{p}}_1}  (v_1, \textbf{\emph{p}}_2) \ge 26 D + 1 \right] \le 2^{21} N^7 \exp \left( - \displaystyle\frac{D^2}{16 N} \right).
	\end{flalign*}
\end{prop}

\begin{figure}[t]
	
	\begin{center}
		
		\begin{tikzpicture}[
		>=stealth,
		scale = .57
		]

		\draw[-, black] (5, 1) -- (6, 1) -- (7, 1) -- (8, 1) -- (8, 2) -- (9, 2) -- (10, 2) -- (10, 3) -- (11, 3) -- (12, 3) -- (12, 4) -- (12, 5) -- (12, 6) -- (13, 6) -- (13, 7) -- (13, 8);

		\draw[-, black] (9, 2) circle[radius = 0] node[scale = .7, above]{$\textbf{p}$};
		
		\draw[-, black] (12.7, 5) circle[radius = 0] node[scale = .7, right]{$\textbf{h}$};
		
		\draw[-, black, dashed, ultra thick] (5, 1) -- (8, 1) -- (10, 2) -- (12, 3) -- (13, 6) -- (13, 8);

		\draw[-, black] (20, 0) -- (20, 8) --  (30, 8) -- (30, 0) -- (20, 0);
		\draw[-, black, thick] (20, 3) -- (21, 3) -- (21, 4) -- (22, 4) -- (23, 4) -- (24, 4) -- (25, 4) -- (25, 5) -- (25, 6) -- (25, 7) -- (26, 7) -- (26, 8);
		\draw[-, black, thick] (20, 1) -- (21, 1) -- (21, 2) -- (22, 2) -- (23, 2) -- (24, 2) -- (24, 3) -- (25, 3) -- (26, 3) -- (26, 4) -- (27, 4) -- (28, 4) -- (28, 8);

		\draw[-, black, dashed, ultra thick] (20, 3) -- (21, 3)  -- (25, 4) -- (26, 7) -- (26, 8);
		\draw[-, black, dashed, ultra thick] (20, 1) -- (21, 1) -- (24, 2) -- (26, 3) -- (28, 4) -- (28, 8);

		\draw[] (20, 6) circle[radius = 0] node[left, scale = .75]{$\Lambda$};
		
		\draw[] (21, 4) circle[radius = 0] node[above, scale = .75]{$\textbf{p}_2$};
		\draw[] (21, 2) circle[radius = 0] node[left, scale = .75]{$\textbf{p}_1$};
		
		\draw[] (21.5, .5) circle[radius = 0] node[right, scale = .75]{$\ell_1$};

		\filldraw[fill = black] (25, 4) circle[radius = .071] node[above = 4, left, scale = .75]{$v_2$};
		\filldraw[fill = black] (26, 3) circle[radius = .071] node[right = 3, below = 4, scale = .75]{$v = v_1$};

		\filldraw[fill = black] (25.7, 6) circle[radius = 0] node[right, scale = .75]{$\textbf{h}_2$};
		\filldraw[fill = black] (27, 3.2) circle[radius = 0] node[right, scale = .75]{$\textbf{h}_1$};

		\filldraw[fill = black] (22.6, 1.6) circle[radius = .071] node[below = 2, scale = .75]{$u_1$};
		
		\filldraw[fill = black] (22, 2) circle[radius = .071] node[above, scale = .75]{$u$};
		
		\filldraw[fill = black] (28, 7) circle[radius = .071] node[right, scale = .75]{$w_1 = w$};
		
		\draw[-, dashed] (21, 0) -- (29, 8);
		
		\draw[->, dotted] (25, 4) -- (25.95, 3.05);
		
		\filldraw[fill = black] (25.3, 3.6) circle[radius = 0] node[below, scale = .75]{$s$};
		
		\end{tikzpicture}
		
	\end{center}	
	
	\caption{\label{u1w1p} Depicted to the left is the lower convex envelope $\textbf{h} = \textbf{h} (\textbf{p})$ of a directed path $\textbf{p}$. Depicted to the right is the setting for the proof of \Cref{p1p2near}.} 
\end{figure}
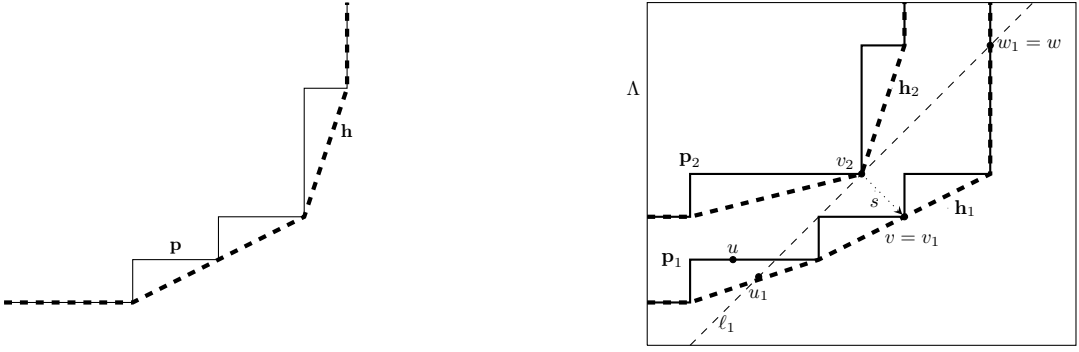

\begin{proof} 
	
	Define the events 
	\begin{flalign*} 
	E = \left\{ \displaystyle\max_{v_1 \in \textbf{p}_1}  (v_1, \textbf{p}_2) \ge 26 D + 1 \right\}; \quad F_1 = \big\{ \Xi (\textbf{p}_1) \ge 2D \big\}; \quad F_2 = \big\{ \Xi (\textbf{p}_2) \ge 4 D \big\}; \quad F = F_1 \cup F_2. 
	\end{flalign*}
	
	\noindent Further let $F^c$ denote the complement of $F$. Then, \Cref{convexp1x} and \Cref{convexp2x} together imply that 
	\begin{flalign}
	\label{gestimate2}
	\mathbb{P} [F] \le 2^{20} N^7 \exp \left( - \displaystyle\frac{D^2}{16 N}\right).
	\end{flalign}
	
	\noindent Now let $\textbf{h}_1 = \textbf{h} (\textbf{p}_1)$ and $\textbf{h}_2 = \textbf{h} (\textbf{p}_2)$ denote the lower convex envelopes of $\textbf{p}_1$ and $\textbf{p}_2$, respectively. Then \eqref{dxi} implies that
	\begin{flalign}
	\label{gcdeltah1p1h2p2}
	\textbf{1}_{F^c} \left( \displaystyle\max_{u \in \textbf{p}_1} d(u, \textbf{h}_1) + \displaystyle\max_{v \in \textbf{h}_1} d (v, \textbf{p}_1) \right) \le 4D; \qquad \textbf{1}_{F^c} \left( \displaystyle\max_{u \in \textbf{p}_2} d(u, \textbf{h}_2) + \displaystyle\max_{v \in \textbf{h}_2} d (v, \textbf{p}_2) \right) \le 8D.
	\end{flalign}
	
	\noindent Next, for any $u \in \mathcal{T}$ and $w \in \mathcal{T} \cap \NE (u)$, we define the event
	\begin{flalign} 
	\label{omegauw} 
	\begin{aligned}
	\Omega (u, w) = \big\{ u, w \in \textbf{p}_1 \big\} \cap \left\{ \displaystyle\max_{v \in \textbf{p}_2 \cap \SE (u, w)} d \big( v, \ell (u, w) \big) \le 4 D \right\}  \cap \left\{ \displaystyle\max_{v \in \textbf{p}_1 \cap \text{R} (u, w)} d \big( v, \ell (u, w) \big) \ge 6D \right\},
	\end{aligned} 
	\end{flalign}
	
	\noindent where we have set $\text{R} (u, w) = \SE (u, w) \cup \NW (u, w)$ and recalled the definitions of $\NW (u, w)$ and $\SE (u, w)$ from \eqref{setsuv}; observe that $\text{R} (u, w)$ denotes the rectangle whose southwest and northeast corners are at $u$ and $v$, respectively. Thus, $\Omega (u, w)$ denotes the event on which $u$ and $w$ are on $\textbf{p}_1$; $\textbf{p}_2$ is not too far to the right of $\ell (u, w)$ when restricted to $\text{R} (u, w)$; and the path $\textbf{p}_1$ is not approximately linear on $\text{R} (u, w)$. 
	
	We claim that 
	\begin{flalign}
	\label{uvwf}
	E \cap F^c \subseteq \bigcup_{u \in \mathcal{T}} \bigcup_{w \in \mathcal{T} \cap \NE (u)} \Omega (u, w); \qquad \displaystyle\max_{u \in \mathcal{T}} \displaystyle\max_{w \in \mathcal{T} \cap \NE (u)} \mathbb{P}\big[ \Omega (u, w) \big] \le 2^{14} N^3 \exp \left( - \frac{D^2}{16 N}\right). 
	\end{flalign}
	
	\noindent To establish the first statement of \eqref{uvwf}, we first for any $v \in \mathcal{T}$ define the event 
	\begin{flalign*} 
	E (v) = \big\{ v \in \textbf{h}_1 \big\} \cap \big\{ d (v, \textbf{h}_2) \ge 14 D + 1 \big\}.
	\end{flalign*}
	
	\noindent Then \eqref{gcdeltah1p1h2p2} implies that $E \cap F^c \subseteq \bigcup_{v \in \mathcal{T}} \big( E (v) \cap F^c \big)$; thus it suffices to show that
	\begin{flalign}
	\label{evfomega} 
	E(v) \cap F^c \subseteq  \bigcup_{u \in \mathcal{T}} \bigcup_{w \in \mathcal{T} \cap \NE (u)} \Omega (u, w),
	\end{flalign}
	
	\noindent for any $v \in \mathcal{T}$.
	
	To that end, let us fix $v_1 = (x_1, y_1) \in \textbf{h}_1$ and assume that $E(v_1) \cap F^c$ holds. Let $v_2 = (x_2, y_2) \in \textbf{h}_2$ denote a point on $\textbf{h}_2$ such that $d(v_1, \textbf{h}_2) = d(v_1, v_2) \ge 14 D + 1$, and let $s = v_1 - v_2 = (x_1 - x_2, y_1 - y_2) \in \mathbb{R}_{\ge 0} \times \mathbb{R}_{\le 0}$ denote the vector from $v_2$ to $v_1$; further let $\ell_1$ denote the line through $v_2$ orthogonal to $s$. Then, since $d(v_1, v_2) = d(v_1, \textbf{h}_2)$ and $\textbf{h}_2$ is convex, we must have that $\textbf{h}_2 \subset \NW (\ell_1)$. See the right side of \Cref{u1w1p} for a depiction. 
	
	Since the vertices at which $\textbf{p}_1$ and $\textbf{p}_2$ enter (and exit) $\mathcal{T}$ are of distance one from each other, the same holds for $\textbf{h}_1$ and $\textbf{h}_2$. Combining this with the facts that $v_1 \in \textbf{h}_1 \cap \SE (\ell_1)$ and $\textbf{h}_2 \subset \NW (\ell_1)$, we deduce that there must exist points $u_1 \in \textbf{h}_1 \cap \SW (v_1)$ and $w_1 \in \textbf{h}_1 \cap \NE (v_1)$ such that $\max \big\{ d (u_1, \ell_1), d (w_1, \ell_1) \big\} \le 1$. Typically, $u_1$ and $w_1$ will lie on $\ell_1$, as depicted in \Cref{u1w1p}
	
	Now let $u$, $v$, and $w$ denote the closest vertices in $\textbf{p}_1$ to $u_1$, $v_1$, and $w_1$, respectively (in \Cref{u1w1p} we have drawn $v = v_1$ and $w = w_1$ for convenience, but we do not assume this to necessarily be the case). We claim that $\Omega (u, w)$ holds. To that end, first observe that 
	\begin{flalign}
	\label{v1u1w1}
	\begin{aligned}
	& \min \big\{ d (v_1, u_1), d (v_1, w_1) \big\} \ge 14 D; \qquad \max \big\{ d(u, u_1), d(v, v_1), d (w, w_1) \big\} \le 4 D; \\
	& \min \big\{ d (v, u), d (v, w) \big\} \ge 6 D.
	\end{aligned}
	\end{flalign}
	
	Indeed, the first estimate in \eqref{v1u1w1} follows from the facts that $d (v_1, \ell_1) \ge 14 D + 1$ and that $\max \big\{ d (u_1, \ell_1), d (w_1, \ell_1) \big\} \le 1$. Furthermore, the second is a consequence of the first bound in \eqref{gcdeltah1p1h2p2}, and the third follows from the first and second bounds in \eqref{v1u1w1}. Thus, since $u_1 \in \SW (v_1)$ and $w_1 \in \NE (v_1)$, we have that $u \in \SW (v)$ and $w \in \NE (v)$. 
	
	Now observe that the first condition in \eqref{omegauw} holds. The second condition in \eqref{omegauw} also holds since
	\begin{flalign*}
	\displaystyle\max_{z \in \textbf{p}_2 \cap \SE (u, w)} d (z, \ell) \le \displaystyle\max \big\{ d (u, u_1), d (w, w_1) \big\} \le 4D,
	\end{flalign*}
	
	\noindent where we have used the second estimate in \eqref{v1u1w1} and the fact that $\textbf{p}_2$ is north of $\textbf{h}_2$, which is north of $\ell_1 = \ell (u_1, w_1)$ (since $\textbf{h}_2$ is convex). 
	
	To verify the third constraint in \eqref{omegauw}, observe that since $d ( v_1, \ell_1 ) \ge 14 D + 1$ we have that $d (v, \ell_1) \ge 10 D + 1$ by the second estimate in \eqref{v1u1w1}. Applying this bound again, and the fact that $\max \big\{ d (u_1, \ell_1), d (w_1, \ell_1) \big\} \le 1$, we deduce that $d (v, \ell) \ge 6 D$. This yields the third condition in \eqref{omegauw} and therefore the first statement of \eqref{uvwf}. 
	
	Now let us establish the second statement of \eqref{uvwf}; this will follow from a suitable application of \Cref{property} and \Cref{linearrandomf}. More specifically, let us fix vertices $u = (x, y) \in \mathcal{T}$ and $w = (x', y') \in \mathcal{T} \cap \NE (u)$, and define the rectangular subdomain $\Lambda = \Lambda_{u, w} = \big( [x, y] \times [x', y'] \big) \cap \mathcal{T}$ of $\mathcal{T}$. Next, we condition on the paths $(\textbf{p}_2, \textbf{p}_3, \ldots , \textbf{p}_{A + C})$; on the restriction $\textbf{p}_1 |_{\mathcal{T} \setminus \Lambda}$; and on the event that $\textbf{p}_1$ enters and exits $\Lambda$ through $u$ and $w$, respectively. 
		
	Then by \Cref{property}, the conditional law of $\textbf{p}_1 |_{\Lambda}$ is uniform on the set of $0$-restricted path ensembles that enter and exit $\Lambda$ through $u$ and $v$, respectively, and satisfy $\textbf{p}_1 \ge \textbf{p}_2$. The fact that $\textbf{p}_1$ must be $0$-restricted is due to the fact that it lies to the right of the triangular face of $\mathcal{T}$, since $C \ge 2$. The second bound in \eqref{uvwf} then follows from applying \Cref{linearrandomf} and \Cref{tconvexrandomfg} with the $\textbf{p}$ there equal to $\textbf{p}_1$ here; the $(v_1, v_2)$ there equal to the $(u, w)$ here; the $(\textbf{f}, \textbf{g})$ there equal to $(\textbf{p}_2, \infty)$ here; the $M$ bounded above by $4N$ here; and the $\Delta$ equal there equal to $4 D$ here. 
	
	Now the proposition follows from \eqref{gestimate2}, \eqref{uvwf}, a union bound, and the fact that $|\mathcal{T}| \le 4N^2$. 
\end{proof}

\section{Exit Location of \texorpdfstring{$\textbf{p}_1^{\aug}$}{} From the \texorpdfstring{$x$}{}-Axis} 

\label{Equationp1} 

In this section we locate where the rightmost path $\textbf{p}_1^{\aug}$ of a uniformly random, $A$-restricted directed path ensemble $\mathcal{P}^{\aug}$ on $\mathcal{X}_{A, B, C; \Psi}$ with augmented domain-wall boundary data leaves the $x$-axis. We state this result as \Cref{lzpsip1} in \Cref{LocationAugmentedPath} and reduce it to \Cref{estimater} below; we then establish the latter proposition in \Cref{REstimate}.

\subsection{Location of \texorpdfstring{$\Phi \big(\mathcal{P}^{\aug} \big)$}{}}

\label{LocationAugmentedPath}

We begin with the following definition that introduces several functions; the relevance of these functions to the arctic boundary $\mathfrak{B}$ is explained by \Cref{lztangentb} below.

\begin{definition} 
	
	\label{nuzsigmaz} 
	
	For any real number $z > 0$, we recall $\zeta (z)$ from \eqref{zeta1z} and define the function $\nu (z) = z^{-1} \zeta (z)$. In particular,
	\begin{flalign}
	\label{nuxidefinition}
	\nu (z) = \displaystyle\frac{\sqrt{\big( z(b + c) + a + c\big)^2 - 4abz} + bz - cz - a - c}{2z} + \displaystyle\frac{\sqrt{z^2 + z + 1} - 1}{z}.
	\end{flalign}

\end{definition}

\begin{lem}
	
\label{lztangentb} 

Let $z \ge 0$ be a real number, and define the point $v = \big( x(z), y(z) \big) \in \mathfrak{B}$ (recall \Cref{zetacurves}). If $\ell = \ell_z$ denotes the tangent line to $\mathfrak{B}$ at $v$, then the equation for $\ell$ is given by $y = z x - \zeta (z)$, where $\zeta (z)$ is given by \eqref{nuxidefinition}.
\end{lem} 

The proof of this lemma follows quickly from \Cref{zetacurves} and \eqref{nuxidefinition} and is therefore omitted. 

The goal of this section is to establish the following proposition; in what follows, we will use the fact that $\zeta (z)$ is invertible on $\mathbb{R}_{\ge 0}$; see \Cref{sigmamuzincreasing} below. 

\begin{prop} 
	
	\label{lzpsip1}

	Let $N > 0$ be an integer, and fix real numbers $a, b, c \in [0, 1]$ and $\psi > 0$ with $a + b + c = 1$; denote $A = \lfloor aN \rfloor$, $B = \lfloor bN \rfloor$, $C = \lfloor cN \rfloor$, and $\Psi = \lfloor \psi N \rfloor$. Further let $D > 0$ be some real number, and let $z \ge 0$ be such that $\zeta = \zeta (z) = \psi$. Set $\nu = \nu (z)$.
	
	Consider an $A$-restricted random directed path ensemble $\mathcal{P}^{\aug} = \big( \textbf{\emph{p}}_1^{\aug}, \textbf{\emph{p}}_2^{\aug}, \ldots , \textbf{\emph{p}}_{A + C + 1}^{\aug} \big)$ on the augmented three-bundle domain $\mathcal{X} = \mathcal{X}_{A, B, C; \Psi}$ with augmented domain-wall boundary data, chosen uniformly at random. Let $\Phi = \Phi (\mathcal{P}^{\aug}) \in [1, A + 2B + C]$ denote the integer such that $\textbf{\emph{p}}_1^{\aug}$ contains an edge from $(\Phi, 0)$ to $(\Phi, 1)$. 
	
	Then, 
	\begin{flalign*}
	\mathbb{P} \big[ | \Phi - \nu N | \ge D \big] \le 2^{96} (\Psi + N)^{20} \exp \left( - \displaystyle\frac{\psi D^2}{16 (\Psi + N + 1)} \right). 
	\end{flalign*} 
	
\end{prop} 

To establish this proposition, we first evaluate the probability $\mathbb{P} \big[ \Phi (\mathcal{P}^{\aug}) = k \big]$. This quantity can be expressed in terms of $k$-refined correlation functions, which are defined as follows.

\begin{definition} 
	
	\label{hkabc} 
	
	For any integers $A, B, C \ge 0$ and $k \in [1, A + 2B + C]$, define the \emph{(singly) $k$-refined correlation function} $H (k) = H_{A, B, C} (k)$ to be the probability that a uniformly random $A$-restricted directed path ensemble $\mathcal{P} = (\textbf{p}_1, \textbf{p}_2, \ldots , \textbf{p}_{A + C})$ on $\mathcal{T}_{A, B, C}$ contains an arrow from $(k, 1)$ to $(k, 2)$ in $\textbf{p}_1$. 

\end{definition} 

In our specific setting on the three-bundle domain $\mathcal{T} = \mathcal{T}_{A, B, C}$, we have the following proposition that explicitly evaluates the $k$-refined correlation function; it was originally due to (59) of \cite{PRC} but appears as below as equation (5.1) of \cite{ACSVMGD} (see also Appendix D of \cite{ACSVMGD}), with the $C + 1$ here equal to $C$ from \cite{ACSVMGD}. In what follows, for any integers $X \in [1, A + 2B + C + 1]$ and $Y \in [1, X]$, we define the quantity
\begin{flalign}
\label{rxydefinition} 
R (X, Y) = \binom{2N - Y + 1}{N} \binom{N + Y - 1}{N} \binom{C + X - Y}{C} \binom{A + B - X + Y - 1}{A - 1}.
\end{flalign}

\begin{prop}[{\cite[Equation (5.1)]{ACSVMGD}}]
	
\label{hkabcidentity} 
	
Let $A, B, C \ge 0$ and $N, k \ge 1$ be integers such that $A + B + C = N$ and $1 \le k \le A + 2B + C + 1$. Then, 
\begin{flalign*}
H_{A, B, C + 1} (k) = \binom{3 N + 1}{N}^{-1} \binom{N}{B}^{-1} \displaystyle\sum_{Y = 1}^k & R(k, Y). 
\end{flalign*}

\end{prop}

Now we have the following result that evaluates $\mathbb{P}[\Phi = X]$ explicitly; its proof proceeds by first expressing this probability through a $k$-refined correlation function and then using \Cref{hkabcidentity} to evaluate the latter explicitly. 

\begin{cor}

\label{xipsiabckidentity} 

Let $A, B, C, \Psi \ge 0$ and $N, X \ge 1$ be integers such that $A + B + C = N$ and $1 \le X \le A + 2B + C + 1$. Recalling the definition of $\Phi$ from \Cref{lzpsip1} and of $R(X, Y)$ from \eqref{rxydefinition}, there exists a constant $Z = Z(A, B, C, \Psi) > 0$ (independent of $X$) such that 
\begin{flalign} 
\label{omegak}
\begin{aligned}
\mathbb{P} [\Phi = X] & = Z^{-1} \binom{\Psi + X - 1}{\Psi}  \displaystyle\sum_{Y = 1}^X R (X, Y). 
\end{aligned}
\end{flalign}

\end{cor} 

\begin{proof}

Observe that a domain-wall $A$-restricted directed path ensemble $\mathcal{P}^{\text{aug}}$ on $\mathcal{X}_{\Psi}$ such that $\Phi (\mathcal{P}^{\text{aug}}) = X$ is the union of a directed up-right path from $(0, - \Psi)$ to $(X, 0)$ and an $A$-restricted directed path ensemble on $\mathcal{T}_{A, B, C + 1}$ with singly refined domain-wall boundary data at $(X, 0)$ (recall \Cref{uw2}). Since the number of such directed up-right paths is equal to $\binom{\Psi + X - 1}{\Psi}$ and the number of such directed path ensembles is proportional to the refined correlation function $H_{A, B, C + 1} (X)$, it follows that $\mathbb{P} \big[ \Phi (\mathcal{P}^{\text{aug}}) = X \big]$ is proportional to $\binom{\Psi + X - 1}{\Psi} H_{A, B, C + 1} (X)$. Now the corollary follows from \Cref{hkabcidentity}.
\end{proof}

In view of \eqref{omegak}, we would like to estimate $R(X, Y)$; this is given by the following proposition, which will be established in \Cref{REstimate}.

\begin{prop}
	
	\label{estimater}
	
	Adopt the notation of \Cref{lzpsip1}. For any integers $X \in [1, A + 2B + C + 1]$ and $Y \in [1, X]$, define $R(X, Y)$ as in \eqref{rxydefinition}, and set 
	\begin{flalign} 
	\label{zrho} 
	\sigma (z) = \frac{\sqrt{z^2 + z + 1} - 1}{z}.
	\end{flalign}
	
	\noindent If $y \in [0, a + 2b + c]$ and $x \in [0, 1 + a]$ satisfy $Nx, Ny \in \mathbb{Z}$ and $d \big( (x N, y N), (\nu N, \sigma N) \big) \ge D$, then 
	\begin{flalign*}
	\displaystyle\frac{R \big( x N, y N \big) \binom{\Psi + x N - 1}{\Psi}}{ R \big( \lfloor \nu N \rfloor, \lfloor \sigma N \rfloor \big) \binom{\Psi + \lfloor \nu N \rfloor - 1}{\Psi} } \le	 2^{94} (\Psi + N)^{18} \exp \left( - \displaystyle\frac{\psi D^2}{4 (\psi + 2) N} \right).
	\end{flalign*}

\end{prop}

Assuming \Cref{estimater}, we can now establish \Cref{lzpsip1}. 

\begin{proof}[Proof of \Cref{lzpsip1} Assuming \Cref{estimater}]
	
Recalling the constant $Z = Z(A, B, C, \Psi)$ from \Cref{xipsiabckidentity}, we have that 
\begin{flalign*}
\mathbb{P} & \big[ |\Phi - \nu N| \ge D \big] \\
& = Z^{-1} \displaystyle\sum_{\substack{|X - \nu N| \ge D \\ X \in [1, A + 2B + C + 1]}} \displaystyle\sum_{Y = 1}^X R(X, Y) \binom{\Psi + X - 1}{\Psi} \\
& \le Z^{-1}R \big(\lfloor \nu N \rfloor, \lfloor \sigma N \rfloor \big) \binom{\Psi + \lfloor \nu N \rfloor - 1}{\Psi} \Bigg(2^{96} (\Psi + N)^{20}  \exp \bigg( - \displaystyle\frac{\psi D^2}{16 (\Psi + N + 1)} \bigg) \Bigg) \\
& \le 2^{96} (\Psi + N)^{20}  \exp \bigg( - \displaystyle\frac{\psi D^2}{16 (\Psi + N + 1)} \bigg),
\end{flalign*}

\noindent where we have used \Cref{estimater} and the fact that $Z \ge R \big(\lfloor \nu N \rfloor, \lfloor \sigma N \rfloor \big) \binom{\Psi + \lfloor \nu N \rfloor - 1}{\Psi}$ (which holds since the left side of \eqref{omegak} is at most equal to $1$). 
\end{proof}

\subsection{Proof of \Cref{estimater}} 

\label{REstimate} 

In this section we establish \Cref{estimater}. To that end, we first define the domain 
\begin{flalign*}
\mathcal{D} = \mathcal{D}_{\psi} = \Big\{ (x, y) \in \mathbb{R}_{> 0}^2: x < 1 + b; x - b < y <	 \min \{ x, 1 \} \Big\}  \subset \mathbb{R}^2.
\end{flalign*}

\noindent Define the functions $h (x, y) = h_{a, b, c} (x, y)$ and $f (x) = f_{\psi} (x)$ on $\mathcal{D}$ and $(0, 1 + b)$, respectively, by
\begin{flalign}
\label{definitionh}
\begin{aligned}
h (x, y) & =  (2 - y) \log (2 - y) - (1 - y) \log (1 - y) + (1 + y) \log (1 + y) - y \log y \\
& \qquad + (c + x - y) \log (c + x - y) - c \log c - (x - y) \log (x - y) \\
& \qquad + (a + b + y - x) \log (a + b + y - x) - a \log a - (b + y - x) \log (b + y - x); \\
f (x) & = (\psi + x) \log (\psi + x) - \psi \log \psi - x \log x.
\end{aligned}
\end{flalign}

\noindent Under this notation, we deduce from \eqref{mestimate1} and \eqref{rxydefinition} that
\begin{flalign}
\label{rhf}
\begin{aligned}
& N h(x, y) - 7 \log (32 N) \le \log R \big( \lfloor x N \rfloor, \lfloor y N \rfloor \big) \le  N h(x, y) + 7 \log (32 N); \\
& N f (x) - 2 \log \big( 64 (\Psi + N) \big) \le \log \binom{\Psi + \lfloor x N \rfloor - 1}{\Psi} \le N f (x) + 2 \log \big (64 (\Psi + N) \big), 
\end{aligned}
\end{flalign}

\noindent for any $(x, y) \in \mathcal{D}$ as above. 

The following proposition now shows that $h(x, y) + f (x)$ is concave on $\mathcal{D}$; in the below, we denote the Hessian of $f + h$ by 
  \begin{flalign*}
 \textbf{H} = \textbf{H} (x, y) = \left[ \begin{array}{cc} \displaystyle\frac{\partial^2 (f + h)}{\partial^2 x} & \displaystyle\frac{\partial^2 (f + h)}{\partial x \partial y}  \\ \displaystyle\frac{\partial^2 (f + h)}{\partial x \partial y}  & \displaystyle\frac{\partial^2 (f + h)}{\partial^2 y}  \end{array}  \right].
 \end{flalign*}

 \begin{lem} 
 	
 \label{fhsum} 
 
 Letting $\lambda_1, \lambda_2$ denote the eigenvalues of $\textbf{\emph{H}}$, we have that $\max\{ \lambda_1, \lambda_2 \} \le - \frac{\psi}{2 (\psi + 2)}$. Furthermore, there is a unique pair $(\varphi, \rho) \in \mathcal{D}$ satisfying 
 \begin{flalign}
 \label{nusigma}
 \displaystyle\frac{\partial (f + h)}{\partial x} (\varphi, \rho) = 0; \qquad \displaystyle\frac{\partial (f + h)}{\partial y} (\varphi, \rho) = 0,
 \end{flalign}
 
 \noindent which is the unique maximum of $f + h$. In particular, for any $d > 0$ and $(x, y) \in \mathcal{D}$ such that $d \big( (x, y), (\varphi, \rho) \big) \ge d$, we have that
 \begin{flalign}
 \label{fxynusigma} 
 f (x, y) + h(x, y)  \le f(\varphi, \rho) + h(\varphi, \rho)  - \displaystyle\frac{\psi d^2}{4 (\psi + 2)}.
 \end{flalign}
 
 \end{lem} 

\begin{proof}
	
Define 
\begin{flalign*}
& S = \frac{1}{(2 - y)(1 - y)} + \frac{1}{y (1 + y)} \ge 1; \quad T =  \frac{a}{(a + b - x + y) (b - x + y)} +  \frac{c}{(c + x - y) (x - y)} \ge 0, 
\end{flalign*}

\noindent where the first estimate above is due to the fact that $y \in [0, 1]$ and the second is due to the fact that $y \le x \le b + y$. Then the definitions \eqref{definitionh} quickly yield  
 \begin{flalign*}
\textbf{H} = \left[ \begin{array}{cc}  - T - \displaystyle\frac{\psi}{x (\psi + x)} & T  \\ T &  - S - T \end{array}  \right] = \left[ \begin{array}{cc}  - \displaystyle\frac{\psi}{x(\psi + x)} & 0 \\ 0 &  - S \end{array}  \right] - T \left[ \begin{array}{cc}  1 & - 1 \\ - 1 & 1 \end{array}  \right].
\end{flalign*}

\noindent Since $S \ge 1$, $x \le 2$, $\frac{\psi}{2 (\psi + 2)} \le 1$, and the matrix $\big[ \begin{smallmatrix} 1 & -1 \\ -1 & 1 \end{smallmatrix} \big]$ nonnegative semi-definite, we deduce that both eigenvalues of $\textbf{H}$ are at most $- \frac{\psi}{2 (\psi + 2)}$. This confirms the first part of the lemma, which implies that $f (x) + h (x, y)$ is concave on $\mathcal{D}$. 

The fact that $f + h$ admits a unique maximum $(\varphi, \rho)$ on $\mathcal{D}$ satisfying \eqref{nusigma} follows from the fact that $f (x) + h(x, y)$ is concave on $\mathcal{D}$ that $f + h$ is smooth on $\mathcal{D}$, and that
\begin{flalign*}  
& \displaystyle\lim_{y \rightarrow \max \{ x - b, 0 \} } \frac{\partial (f + h)}{\partial y} (x, y) = \infty; \qquad \qquad \displaystyle\lim_{y \rightarrow \min \{ x, 1\} } \frac{\partial (f + h)}{\partial y} (x, y) = -\infty; \\
& \displaystyle\lim_{x \rightarrow y} \frac{\partial (f + h)}{\partial x} (x, y) = \infty; \qquad \qquad \qquad \qquad \displaystyle\lim_{x \rightarrow b + y } \frac{\partial (f + h)}{\partial x} (x, y) = -\infty.
\end{flalign*}

\noindent The last estimate \eqref{fxynusigma} then follows from the fact that both eigenvalues of $\textbf{H}$ are bounded above by $- \frac{\psi}{2 (\psi + 2)}$. 
\end{proof}

The following corollary explicitly evaluates the maximum $(\varphi, \rho)$ of $f + h$; in what follows, we observe from \Cref{sigmamuzincreasing} below that $\zeta (z)$ (from \eqref{nuxidefinition}) is invertible on $\mathbb{R}_{\ge 0}$.  

\begin{cor}
	
\label{nusigmafhmaximum} 

Adopting the notation of \Cref{fhsum}, we have that $\varphi = \nu (z_{\psi})$ and $\rho = \sigma (z_{\psi})$, where $z_{\psi}$ is such that $\zeta (z_{\psi}) = \psi$; here, $\nu$ and $\zeta$ were defined in \eqref{nuxidefinition}, and $\sigma$ was defined in \eqref{zrho}.

\end{cor} 

\begin{proof}

The first and second equations in \eqref{nusigma} yield 
\begin{flalign}
\label{nusigmaidentities}
\displaystyle\frac{(\varphi - \rho) (a + b - \varphi + \rho)}{(c + \varphi - \rho) (b - \varphi + \rho) }   = \displaystyle\frac{2 \rho - \rho^2}{1 - \rho^2 }; \qquad 1 + \displaystyle\frac{\psi}{\varphi} = \displaystyle\frac{(\varphi - \rho) (a + b - \varphi + \rho)}{(b - \varphi + \rho) (c + \varphi - \rho)}, 
\end{flalign}

\noindent respectively. Thus, denoting $z = \frac{\psi}{\varphi}$ and combining the two identities in \eqref{nusigmaidentities} yields $z + 1 = \frac{2 \rho - \rho^2}{1 - \rho^2}$. Solving this equation, we find that $\rho = \sigma (z)$ (where we have selected the positive solution of this equation, since $\rho > 0$). Letting $\mu = \varphi - \rho$, we find from the second identity in \eqref{nusigmaidentities} that 
\begin{flalign}
\label{muidentity} 
\displaystyle\frac{\mu (a + b - \mu)}{(b - \mu)(c + \mu)} = 1 + z.
\end{flalign}

\noindent Solving this equation, we find that $\mu = \nu(z) - \sigma (z)$ (where we have selected the positive solution of \eqref{muidentity} since $\mu = \varphi - \rho > 0$). Thus, $\varphi = \nu (z)$, so the fact that $z = \frac{\psi}{\varphi}$ implies that $\zeta (z) = \psi$, which yields $z = z_{\psi}$. Since $\varphi = \nu (z)$ and $\rho = \sigma (z)$, this implies the corollary.
\end{proof}

\begin{rem} 

\label{sigmamuzincreasing} 

From \eqref{zrho}, $\sigma (z)$ is increasing in $z$. Furthermore, \eqref{muidentity} implies that $\mu$ is increasing in $z$. It in particular follows that, $\nu (z) = \rho + \mu$ is increasing in $z$, so that $\zeta (z) = z \nu (z)$ is also in $z$. 

Furthermore, $\lim_{z \rightarrow 0} \zeta (0) = 0$ and $\lim_{z \rightarrow \infty} \zeta (z) = \infty$ (since $\lim_{z \rightarrow 0} \nu (z) = \frac{1}{2} + \frac{bc}{a + c} < \infty$, $\lim_{z \rightarrow \infty} \nu (z) = 1 + b > 0$, and $\zeta (z) = z \nu (z)$). Thus, for any $\psi \in \mathbb{R}_{\ge 0}$ there exists a unique $z = z_{\psi} \in \mathbb{R}_{\ge 0}$ such that $\zeta (z) = \psi$.

\end{rem} 

Now we can establish \Cref{estimater}.

\begin{proof}[Proof of \Cref{estimater}]

Denoting $g_{\psi} (x, y) = h (x, y) + f_{\psi} (x)$, we have from \eqref{rhf} that
\begin{flalign*}
 \log & \Bigg( R \big(  x N ,   y N  \big) \binom{\Psi +  x N - 1}{\Psi} \Bigg) - \log \Bigg( R \big( \lfloor \nu N \rfloor, \lfloor \sigma N \rfloor \big) \binom{\Psi + \lfloor \nu N \rfloor - 1}{\Psi} \Bigg) \\
& \le N \big(  g_{\psi} (x, y) - g_{\psi} (\nu, \sigma) \big) + 18 \log  (\Psi + N) + 94 \log 2,
\end{flalign*}

\noindent upon which the lemma follows by exponentiation and \eqref{fxynusigma}. 
\end{proof}

\section{Proof of Theorem \ref{p1testimate}}

\label{ProofBoundary}

In this section we establish \Cref{p1testimate}. We do this by first addressing the case when $v \in \SE (\mathfrak{B})$ in \Cref{BSouthEast} and then addressing the case when $v \in \NW (\mathfrak{B})$ in \Cref{BNorthWest}. 

\subsection{The Case When \texorpdfstring{$v \in \SE (\mathfrak{B})$}{}}

\label{BSouthEast}

In this section we establish the following result, which confirms \Cref{p1testimate} when $v \in \SE (\mathfrak{B})$.

\begin{prop}

\label{curvepabove}

Adopt the notation of \Cref{p1testimate}, where we assume that $\delta > N^{-1 / 5}$ (instead of $\delta > N^{-1 / 30}$), and fix some $v \in \SE (\mathfrak{B}) \cap \mathcal{T}$ satisfying $d \big( N^{-1} v, \mathfrak{B} \big) > \delta$. Then there exists a constant $\gamma > 0$ such that, off of an event of probability at most $\gamma^{-1} \exp \big( - \gamma \delta^4 N \big)$, we have that $v \notin \textbf{\emph{p}}_1$. 

\end{prop}

Before we establish \Cref{curvepabove}, let us observe that the following corollary follows from \Cref{curvepabove}, a union bound over $v \in \SE (\mathfrak{B}) \cap \mathcal{T}$, and the fact that $\textbf{p}_1$ is nondecreasing and continuous.  

\begin{cor}
	
	\label{curvepabove2}
	
	Adopt the notation of \Cref{p1testimate}, but where we assume that $\delta > N^{-1 / 5}$ (instead of $\delta > N^{-1 / 30}$). Then, there exists a constant $\gamma > 0$ such that the following holds off of an event of probability at most $\gamma^{-1} \exp \big( - \gamma \delta^4 N \big)$. For any $v \in \SE (\mathfrak{B}) \cap \mathcal{T}$ satisfying $d \big( N^{-1} v, \mathfrak{B} \big) > \delta$, we have that $v \in \SE (\textbf{\emph{p}}_1)$.
	
\end{cor}

Now let us establish \Cref{curvepabove}. 

\begin{proof}[Proof of \Cref{curvepabove}]
	
	Set $D = \frac{\delta N - 5}{40} > 0$; we may assume that $N$ is sufficiently large so that $D > 0$. Also let $u = (x, y) \in \mathbb{R}_{\ge 0}^2$ denote the vertex in $N \mathfrak{B}$ closest to $v$; since $N^{-1} u \in \mathfrak{B}$, there exists some $z \in \mathbb{R}_{\ge 0}$ such that $u = \big( N x(z), N y(z) \big)$ (recall \Cref{zetacurves}). Let $\ell = N \ell_z$ denote the tangent line to $N \mathfrak{B}$; let it meet the $x$-axis and $y$-axis at $(\nu N, 0)$ and $(0, -\psi N)$, respectively, for some real numbers $\nu \in \big( \frac{1}{2}, 1 + b \big)$ and $\psi > 0$. In view of \Cref{lztangentb}, we have that $\nu = \nu (z)$ and $\psi = \zeta (z)$, where $\nu (z)$ and $\zeta (z)$ are given by \eqref{nuxidefinition}. 

	Denote $\Psi = \psi N$; in what follows, it will be convenient to assume that $\Psi$ is an integer (although one can replace $\Psi$ with $\lfloor \Psi \rfloor$, and the estimates below will continue to hold). Then the convexity of $\mathfrak{B}$; the fact that $d (N^{-1} v, \mathfrak{B}) \ge \delta$; the fact that $\nu \ge \frac{1}{2}$; and the fact that $1 + a, 1 + b \le 2$ together imply that $\frac{\delta}{4} \le \psi \le \frac{4}{\delta}$. Thus, $\Psi \le \frac{4 N}{\delta}$. 
	
	Now define the event $E = \big\{ v \in \textbf{p}_1 \big\}$, and also the events
	\begin{flalign*}
	F = \left\{ \displaystyle\max_{v_1 \in \textbf{p}_1} d (v_1, \textbf{p}_2) \ge 26 D + 1 \right\}; \qquad G = \left \{ \displaystyle\max_{w \in \textbf{p}_2 \cap \SE (\ell)} d (w, \ell) \ge 12D + 4 \right\}.
	\end{flalign*}
	
	We claim that $E \subseteq F \cup G$. So, let $E^c$, $F^c$, and $G^c$ denote the complements of $E$, $F$, and $G$, respectively; restricting to $F^c \cap G^c$, we will show that $E^c$ holds. To that end, restrict to the former event; there exists some $v_1 \in \textbf{p}_1$ such that $d (v, v_1) = d (v, \textbf{p}_1)$ and some $v_2 \in \textbf{p}_2$ such that $d (v_1, v_2) = d (v_1, \textbf{p}_2)$. Under this notation, we have that $d (v, \textbf{p}_1) = d (v, v_1) > d (v, v_2) - 26D - 1$, since we are restricting to $F^c$. 
	
	Since we are also restricting to $G^c$, we have that $d \big( v_2, \NW(\ell) \big) < 12 D + 4$, meaning that $d (v, \textbf{p}_1) > d \big( v, \NW (\ell) \big) - 38D - 5$. Now, since $N \mathfrak{B}$ is convex, $v \in \SE (\ell)$, and $u \in N \mathfrak{B}$ satisfies $d (v, N \mathfrak{B}) = d(v, u)$, we deduce that $d \big( v, \NW (\ell) \big) = d (v, u) = \delta N$. Thus, $d (v, \textbf{p}_1) > \delta N - 38D - 5 = 2D > 0$, meaning that $v \notin \textbf{p}_1$, so $E^c$ holds. Thus, $E \subseteq F \cup G$. 
		
	Since $\mathbb{P} [F]$ is bounded by \Cref{p1p2near}, we must bound $\mathbb{P}[G]$. We will do so by comparing it to the probability of a similar event with respect to a directed path ensemble on an augmented three-bundle domain. 
	
	More specifically, recall the augmented three-bundle domain $\mathcal{X} = \mathcal{X}_{A, B, C; \Psi}$ from \Cref{xpsi}, and consider a $A$-restricted directed path ensemble $\mathcal{P}^{\aug} = \big(\textbf{p}_1^{\aug}, \textbf{p}_2^{\aug}, \ldots , \textbf{p}_{A + C + 1}^{\aug} \big)$ on $\mathcal{X}$ with augmented domain-wall boundary data (recall \Cref{domainaugmentedboundary}), chosen uniformly at random. Let $\Phi \in [1, A + 2B + C]$ denote the integer such that there is an edge in $\textbf{p}_1^{\aug}$ from $(\Phi, 0)$ to $(\Phi, 1)$. Further let $\ell'$ denote the line through $(0, -\Psi)$ that passes through a vertex $v_2' \in \textbf{p}_2^{\aug}$ such that $\textbf{p}_2^{\aug} \subset \NW (\ell')$ (see \Cref{figuretangent}, where the $\ell$ there is the $\ell'$ here). 
	
	Define the events
	\begin{flalign*}
	\Omega_1 = \big\{ |\Phi - \nu N| \ge D \big\}; \qquad \Omega_2 = \Big\{ d \big( (\Phi, 0), \ell' \big) \ge 2D + 1 \Big\}; \qquad \Omega = \Omega_1 \cup \Omega_2,
	\end{flalign*}
	
	\noindent and let $\Omega^c$ denote the complement of $\Omega$. 
	
	We claim that $\mathbb{P} [G] \le \mathbb{P} [\Omega]$. To that end, let us restrict to the event $\Omega^c$; then, $d \big( (\nu N, 0), \ell' \big) < 3 D + 1$. Since $\nu > \frac{1}{2}$, $1 + b \le 2$, and $\ell$ and $\ell'$ intersect the $y$-axis at the same point $(0, -\Psi)$, we have that $d (w, \ell') <  12 D + 4$ for any $w \in \ell \cap \mathcal{T}$. Thus, since $\textbf{p}_2^{\aug} \subset \NW \big( \ell' \big)$, it follows that 
	\begin{flalign}
	\label{lestimatefc}
	\textbf{1}_{\Omega^c} \displaystyle\max_{w \in \textbf{p}_2^{\aug} \cap \SE (\ell)} d (w, \ell) < 12D + 4.
	\end{flalign}
	
	Now recall from \Cref{ensemblesnn1} that there is a coupling between the random directed path ensembles $\mathcal{P}$ and $\mathcal{P}^{\aug}$ such that $\textbf{p}_2 \le \textbf{p}_2^{\aug}$ almost surely. Thus, \eqref{lestimatefc} implies that $\mathbb{P} [G] \le \mathbb{P} [\Omega]$. 
	
	Hence, $\mathbb{P} [E] \le \mathbb{P} [F] + \mathbb{P} [G] \le \mathbb{P} [F] + \mathbb{P} [\Omega]$. By \Cref{p1p2near}, \Cref{p1p2tangent}, \Cref{lzpsip1}, and the fact that $\frac{\delta}{4} \le \psi \le \frac{4}{\delta}$ we have that 
	\begin{flalign}
	\label{fprobability}
	\mathbb{P} [F] \le 2^{21} N^7 \left( - \displaystyle\frac{D^2}{16 N} \right); \qquad \mathbb{P} [\Omega] \le 2^{150} \delta^{-20} N^{20} \exp \left( - \displaystyle\frac{\delta^2 D^2}{320 N} \right).
	\end{flalign} 
	
	 Now the proposition follows from summing the two estimates in \eqref{fprobability}. 
\end{proof}

\subsection{The Case When \texorpdfstring{$v \in \NW(\mathfrak{B})$}{}}

\label{BNorthWest}

In this section we establish \Cref{p1testimate} in the case when $v \in \NW (\mathfrak{B})$, from which we deduce (using \Cref{curvepabove2}) that \Cref{p1testimate} holds. 

We begin with the following lemma, which is a deterministic statement that quantifies the convexity of the curve $\mathfrak{B}$. In what follows, we let $\theta = \frac{1}{2} + \frac{bc}{a + c}$ and $\lambda = \frac{1}{2} + \frac{ab}{b + c}$, which are the real numbers such that $\mathfrak{B}$ is tangent to the $x$-axis at $(\theta, 0)$ and to the line $x = 1 + b$ at $(1 + b, \lambda)$.

\begin{lem}
	
	\label{vertexvdistance} 
	
	Let $\delta > 0$ be a real number, and let $v = (x, y) \in \mathbb{R}_{\ge 0}^2 \cap \NW (\mathfrak{B})$ be such that $d (v, \mathfrak{B}) > \delta$, $x \in [\theta + \delta, 1 + b - \delta]$, and $y \in [\delta, \lambda - \delta]$. Let $u \in \mathfrak{B}$ be such that $d (v, \mathfrak{B}) = d (v, u)$, and let $\ell$ denote the tangent line to $\mathfrak{B}$ at $u$.
	
	Then there exists a (small) constant $\varpi = \varpi (a, b, c) > 0$ such that, for any $w \in \mathfrak{B} \cap \big( \mathbb{R}^2 \setminus \SE (v) \big)$, we have that $d (w, \ell) > \varpi \delta^5$. 
	
\end{lem}

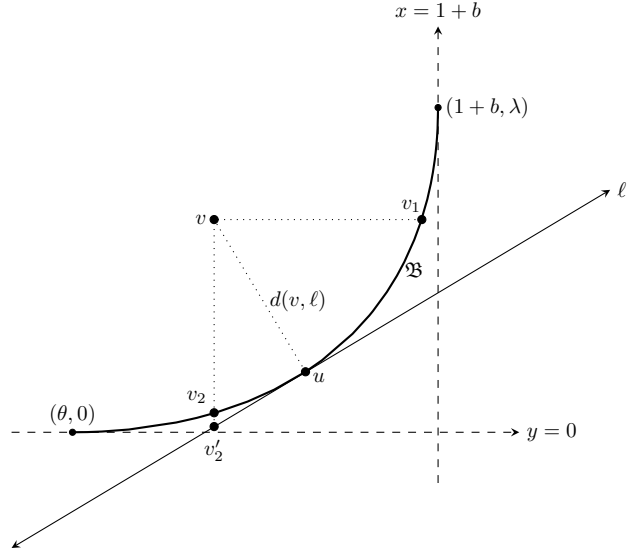
\begin{figure}[t]
	
	\begin{center}
		
		\begin{tikzpicture}[
		>=stealth,
		scale = 1.35
		]

		\draw[->, black, dashed] (22.2, -.5) -- (22.2, 4) node[above, scale = .8]{$x = 1 + b$};
		\draw[->, black, dashed] (18, 0) -- (23, 0) node[right, scale = .8]{$y = 0$};

		\draw[-, black, thick] (18.6, 0) -- (18.7472, 0.00182795) -- (18.8887, 0.00712272) -- (19.1547, 0.0269822) -- (19.6211, 0.0962246) -- (19.8235, 0.141696) -- (20.0074, 0.192169) -- (20.3251, 0.302868) -- (20.5863, 0.41996) -- (20.8947, 0.596203) -- (21.0572, 0.709709) -- (21.3576, 0.96995) -- (21.5234, 1.1514) -- (21.5575, 1.19315) -- (21.6957, 1.38218) -- (21.7946, 1.54232) -- (21.8676, 1.67872) -- (21.9656, 1.89696) -- (22.0262, 2.06259) -- (22.0939, 2.29544) -- (22.1287, 2.4503) -- (22.1488, 2.56034) -- (22.1662, 2.67609) -- (22.176, 2.75652) -- (22.1836, 2.83198) -- (22.1897, 2.90688) -- (22.1949, 2.99173) -- (22.2, 3.2);

		\draw[black] (21.8, 1.6) circle [radius = 0] node[right, scale = .8]{$\mathfrak{B}$};

		\draw[black, fill = black] (18.6, 0) circle [radius = .03] node[above, scale=  .8]{$(\theta, 0)$};
		
		\draw[black, fill = black] (20.8947, 0.596203) circle [radius = .04] node[below = 2, right, scale=  .8]{$u$}; 
		\draw[black, fill = black] (19.9947, 2.096203) circle [radius = .04] node[left, scale=  .8]{$v$}; 
		
		\draw[-, black, dotted] (20.8947, 0.596203) -- (19.9947, 2.096203);
		\draw[-, black, dotted] (19.9947, .056203) -- (19.9947, 2.096203);
		\draw[-, black, dotted] (22.04, 2.096203) -- (19.9947, 2.096203);

		\draw[black] (20.47, 1.3) circle [radius = 0] node[right = 0, scale = .8]{$d (v, \ell)$};

		\draw[black, fill = black] (19.9947, .192169) circle [radius = .04] node[above = 6, left = 0, scale=  .8]{$v_2$}; 
		\draw[black, fill = black] (19.9947, .056203) circle [radius = .04] node[below = 2, scale=  .8]{$v_2'$}; 	
		\draw[black, fill = black] (22.04, 2.096203) circle [radius = .04] node[left = 4, above = 0, scale=  .8]{$v_1$}; 
		
		\draw[black, fill = black] (22.2, 3.2) circle [radius = .03] node[right, scale=  .8]{$(1 + b, \lambda)$};
		
		\draw[<->, black] (18, -1.14062) -- (23.8947, 2.389203) node[right, scale = .8]{$\ell$};

		\end{tikzpicture}
		
	\end{center}	
	
	\caption{\label{bvertexv} 	The setting for \Cref{vertexvdistance} and its proof is depicted above.} 
\end{figure}

\begin{proof}
	
	Define $v_1 = (x', y)$ and $v_2 \in (x, y')$ to be the points on $\mathfrak{B}$ such that $\ell (v, v_1)$ and $\ell (v, v_2)$ are parallel to the $x$-axis and $y$-axis, respectively; see \Cref{bvertexv}. The fact that $\mathfrak{B}$ is nondecreasing and convex then quickly implies that $d(w, \ell) \ge \min \big\{ d(v_1, \ell), d (v_2, \ell) \big\}$. Thus, it suffices to establish the lemma in the case when $w \in \{ v_1, v_2 \}$. Let us assume that $w = v_2$, since the case $w = v_1$ is entirely analogous.
		
	To that end, recall from \Cref{zetacurves} that there exist positive real numbers $z' < z$ such that $u = \big( x(z), y (z) \big)$ and $v_2 = \big( x (z'), y (z') \big)$; then, $\ell = \ell_z$ (from \Cref{lztangentb}). Now, since $v \in [\theta + \delta, 1 + b - \delta] \times [\delta, \lambda - \delta]$, $\theta > \frac{1}{2}$, and $1 + b < 2$, the convexity of $\mathfrak{B}$ implies that the slope of $\ell$ is between $\frac{\delta}{4}$ and $\frac{4}{\delta}$. From this and the fact that $d (v, \ell) > \delta$, we deduce that that $ x(z) - x(z') > \frac{\delta^2}{4}$. 
	
	Next, from the equations \eqref{zeta1z} and \eqref{x1y1z} defining $x(t)$, we deduce the existence of a constant $\varpi_1 = \varpi_1 (a, b, c) > 0$ such that $0 \le x' (t) \le \varpi_1^{-1}$. Hence, $\frac{\partial}{\partial r} \big( x^{-1} (r) \big) \ge \varpi_1$. Since $t$ denotes the slope of the line tangent to $\mathfrak{B}$ at $\big( x(t), y(t) \big)$, this implies that the slope of $\ell (u, v_2)$ is at most $z - \frac{\varpi_1 \delta^2}{8}$. This and the fact that $x (z) - x (z') > \frac{\delta^2}{4}$ together imply that $d (v_2, v_2') > \frac{\varpi \delta^4}{32}$, where $v_2' = (x, y'')$ is the point where $\ell (v, v_2)$ meets $\ell$ (see \Cref{bvertexv}).

	Now, since the slope of $\ell$ is between $\frac{\delta}{4}$ and $\frac{4}{\delta}$, we deduce the existence of a constant $\varpi = \varpi (a, b, c) > 0$ such that $d (v_2, \ell) > \varpi \delta^5$. This implies the lemma. 
\end{proof}

Using this lemma, we can establish the following proposition.

\begin{prop}
	
	\label{curvepbelow}

	Adopt the notation of \Cref{p1testimate}, and fix some $v \in \NW (\mathfrak{B}) \cap \mathcal{T}$ such that $v \in [\theta + \delta, 1 + b - \delta] \times [\delta, \lambda - \delta]$ and $d \big( N^{-1} v, \mathfrak{B} \big) > \delta$. Then, there exists a constant $\gamma = \gamma (a, b, c) > 0$ such that $N^{-1} v \in \NW (\textbf{\emph{p}}_1)$ holds off of an event of probability at most $\gamma^{-1} \exp \big( - \gamma \delta^{24} N \big)$. 
	
\end{prop}

\begin{proof}	
	
	Let $\varepsilon = \delta^6 > N^{-1 / 5}$ and $D = \frac{\varpi \delta^6 N - \varepsilon \delta N - 25}{60}$, where $\varpi = \varpi (a, b, c) > 0$ denotes the constant from \Cref{vertexvdistance}; we may assume that $N$ is sufficiently large so that $D > 0$. 

	As in the proof of \Cref{curvepabove}, let $u = (x, y) \in \mathbb{R}_{\ge 0}^2$ denote the vertex in $N \mathfrak{B}$ closest to $v$; since $N^{-1} u \in \mathfrak{B}$, there exists some $z \in \mathbb{R}_{\ge 0}$ such that $u = \big( N x(z), N y(z) \big)$. Let $\ell = N \ell_z$ denote the tangent line to $N \mathfrak{B}$ at $u$; let it meet the $x$-axis and $y$-axis at $(\nu N, 0)$ and $(0, -\psi N)$, respectively, where $\nu = \nu (z) \in \big( \frac{1}{2}, 1 + b \big)$ and $\psi = \zeta (z) > 0$. Observe that the slope of $\ell$ is between $\frac{\delta}{4}$ and $\frac{4}{\delta}$. Denote $\Psi = \psi N$, which we assume to be an integer. 
	
	Next define the event $E = \big\{ v \in \SE (\textbf{p}_1) \big\}$, and recall the augmented three-bundle domain $\mathcal{X} = \mathcal{X}_{A, B, C; \Psi}$. Let $\mathcal{P}^{\aug} = \big( \textbf{p}_1^{\aug}, \textbf{p}_2^{\aug}, \ldots , \textbf{p}_{A + C + 1}^{\aug} \big)$ denote a uniformly random $A$-restricted directed path ensemble on $\mathcal{X}$ with augmented domain-wall boundary data. By \Cref{ensemblesnn1}, there is a coupling between $\mathcal{P}$ and $\mathcal{P}^{\aug}$ such that $\textbf{p}_1 \ge \textbf{p}_2^{\aug}$ almost surely. Therefore, defining the event $F = \big\{ v \in \SE (\textbf{p}_2^{\aug}) \big\}$, we find that $\mathbb{P} [E] \le \mathbb{P} [F]$. Hence, it suffices to bound $\mathbb{P} [F]$. 
	
	To that end, as in the proof of \Cref{curvepabove}, let $\Phi \in [1, A + 2B + C + 1]$ denote the integer such that there is an edge in $\textbf{p}_1^{\aug}$ from $(\Phi, 0)$ to $(\Phi, 1)$. Further let $\ell'$ denote the line through $(0, -\Psi)$ that passes through a vertex $v_2 \in \textbf{p}_2^{\aug}$ such that $\textbf{p}_2^{\aug} \subset \NW (\ell')$.  Define the events 
	\begin{flalign*}
	& G = \big\{ v_2 \in \SE (N \mathfrak{B}) \big\} \cap \big\{  d (v_2, N \mathfrak{B}) \ge \varepsilon N \big\}; \\
	& \Omega_1 = \big\{ |\Phi - \nu N| \ge D \big\}; \qquad \Omega_2 = \Big\{ d \big( (\Phi, 0), \ell' \big) \ge 2D + 1 \Big\}; \qquad \Omega = \Omega_1 \cup \Omega_2.
	\end{flalign*}
	
	 We claim that $F \subseteq G \cup \Omega$. Letting $\Omega^c$ denote the complement of $\Omega$, it suffices to show that $F \cap \Omega^c \subseteq G$. So, let us restrict to $F \cap \Omega^c$ and show that $G$ holds.
	 
	 To that end, observe that $d \big( (\nu N, 0), \ell' \big) \le 3 D + 1$, since we are restricting to $\Omega^c$. Due to the facts that $\nu > \frac{1}{2}$ and $1 + b \le 2$, this implies that $d (v_2, \ell) < 12 D + 4$. Moreover, since we are also restricting to $F$, we have that $v \in \SE \big( \textbf{p}_2^{\aug} \big)$. Thus, no vertex of $\textbf{p}_2^{\aug}$ is in the interior of $\SE (v)$, which in particular implies that $v_2$ is in the closure of $\mathbb{R}^2 \setminus \SE (v)$. The fact that $d (v_2, \ell) < 12 D + 4$ and the fact that the slope of the line $\ell$ is between $\frac{\delta}{4}$ and $\frac{4}{\delta}$ then yields the existence of a vertex $w \in \ell \cap \big( \mathbb{R}^2 \setminus \SE (v) \big)$ such that $d(v_2, w) < \frac{5}{\delta} (12 D + 4)$.  
	 
	 By \Cref{vertexvdistance}, we have that $w \in \SE (N \mathfrak{B})$ and $d (w, N \mathfrak{B}) > \varpi \delta^5 N$. Therefore, since $\varpi \delta^5 N > \frac{5}{\delta} (12 D + 4) + \varepsilon N$, we have that $v_2 \in \textbf{p}_2^{\aug} \cap \SE (N \mathfrak{B})$ and $d (v_2, N \mathfrak{B}) > \varepsilon N$. Therefore, the event $G$ holds, and so $F \subseteq G \cup \Omega$. 
	 
	 Hence, $\mathbb{P} [E] \le \mathbb{P} [F] \le \mathbb{P} [G] + \mathbb{P} [\Omega]$. Thus, the proposition follows from using \Cref{curvepabove2} (applied with the $\delta$ there equal to the $\varepsilon$  here) to bound $\mathbb{P} [F]$ and \Cref{p1p2tangent} and \Cref{lzpsip1} to bound $\mathbb{P} [\Omega]$ (as in the proof of \Cref{curvepabove}). 
\end{proof}

Now we can establish \Cref{p1testimate}. 

\begin{proof}[Proof of \Cref{p1testimate}]
	
	If $v \in \SE (\mathfrak{B})$, the result follows from \Cref{curvepabove2}, so assume that $v \in \NW (\mathfrak{B})$. By \Cref{curvepbelow} and a union bound we deduce that there exists a constant $\gamma = \gamma (a, b, c) > 0$ such that, off of an event of probability at most $\gamma^{-1} \exp \big( - \gamma \delta^{24} N \big)$, we have that $v \in \NW (\textbf{p}_1)$ for each $v \in \NW (\mathfrak{B}) \cap \mathcal{T}$ such that $N^{-1} v \in [\theta + \delta, 1 + b - \delta] \times [\delta, \lambda - \delta]$ and $d (N^{-1} v, \mathfrak{B}) > \delta$. Now that the same statement holds for all $v \in \NW(\mathfrak{B}) \cap \mathcal{T}$ follows from the fact that $\textbf{p}_1$ is nondecreasing. 
\end{proof}

\end{document}